\documentclass[12pt]{article}
\usepackage[active]{srcltx}
\usepackage[utf8]{inputenc}
\usepackage[english]{babel}
\usepackage[T1]{fontenc}
\usepackage{lmodern}
\usepackage{amsfonts}
\usepackage{amsmath}
\usepackage{graphicx,epsfig}
\usepackage{amssymb}
\usepackage{theorem,subfigure}
\usepackage[usenames,dvipsnames]{color} 
\usepackage{colortbl}
\definecolor{gray73}{RGB}{186,186,186}
\usepackage{graphicx}
\usepackage{array}
\usepackage{tikz}
\usetikzlibrary{matrix,arrows}
\usepackage{verbatim}
\usepackage{geometry}
\usepackage{multirow}

\usepackage{capt-of}
\usepackage{caption}

\DeclareCaptionLabelFormat{mylabelformat}{#1 #2:}
\usepackage{hyperref} 
\usepackage[capitalise]{cleveref}
\newcommand\mycom[2]{\genfrac{}{}{0pt}{}{#1}{#2}}
\newtheorem{theorem}{Theorem}[section]

\newtheorem{remark}[theorem]{Remark}

\newtheorem{algorithm}[theorem]{Algorithm}

\newenvironment{proof}[1][Proof]{\noindent\textbf{#1:} }{\ \rule{0.5em}{0.5em}}

\numberwithin{equation}{section}

\newcommand{\kk}{\mathbf k} 

\newcommand{\elll}{\boldsymbol \ell}
\newcommand{\nuu}{\boldsymbol \nu}

\newcommand{\jj}{\mathbf j}

\newcommand{\w}{\mathsf{w}}

\newcommand{\rank}{\mathrm{rank}}

\newcommand{\diag}{\mathrm{diag}}
\renewcommand{\vec}{\mathrm{vec}}

\newcommand{\argmin}{\mathop{\mathrm{argmin}}}

 at12pt
 at9pt

\begin{document}
\title{Image reconstruction from structured subsampled 2D Fourier  data}
\author {Gerlind Plonka\footnote{University of G\"ottingen, Institute for Numerical and Applied Mathematics, Lotzestr.\ 16-18,  37083 G\"ottingen, Germany. Email: (plonka,a.riahi)@math.uni-goettingen.de, Corresponding Author: Gerlind Plonka (Orcid: https://orcid.org/0000-0002-3232-0573)} \qquad and \qquad 
Anahita Riahi$^*$
}

\maketitle

\abstract{
In this paper we study the performance of image reconstruction methods from incomplete samples of the 2D discrete Fourier transform. Inspired by requirements in parallel MRI, we  focus on a special sampling pattern with a small number of acquired rows of the Fourier transformed image $\hat{\mathbf A}$. We show the importance of the low-pass set of acquired rows around zero in the Fourier space for image reconstruction. A suitable choice of the width $L$ of this index set depends on the image data and  is crucial to achieve optimal reconstruction results.
We prove that non-adaptive reconstruction approaches cannot lead to satisfying recovery results. We propose a new hybrid algorithm which connects the TV minimization technique based on primal-dual optimization with a recovery algorithm which exploits properties of the special sampling pattern for reconstruction. Our method shows very good performance for natural images as well as for cartoon-like images for a data reduction rate up to $8$ in the complex setting and even $16$ for real images.
}
\smallskip


\section{Introduction}
\label{sec1}
In this paper we study the problem of how to efficiently reconstruct a discrete 
image ${\mathbf A} \in {\mathbb C}^{N \times M}$ from structured undersampled 2D  discrete Fourier transform (DFT) data.

Incomplete Fourier data arise in different application fields, as for example in magnetic resonance imaging (MRI), see 
 e.g.\ \cite{Block07,Feng14,grappa,Keeling,spirit,ESPIRIT,Zhang22}, seismic imaging \cite{Wu21}, or computerized tomography \cite{Beinert22,Karp88,Sezan84}.
Depending on the applications, the patterns of incomplete Fourier measurements have different structures. Discrete subsampling is mostly performed on a two-dimensional  Cartesian grid or on a polar grid in frequency domain.  
Several recovery algorithms focus on the polar grid, see e.g.\ \cite{Block07, Feng14, Knoll11,Lustig07,Yang10} or on (structured) random undersampling on the  Cartesian grid \cite{Gelb16,Candes06,Krahmer14,Muck15,Xiao22}, where in both cases the Fourier sampling pattern is denser for low frequencies.

However, in most applications, a random undersampling is not possible or technically too expensive. For example, in MRI only samples along a line or a smooth curve can be efficiently acquired, and every new line takes additional acquisition time. Moreover, in parallel MRI, one acquires simultaneously subsampled Fourier measurements along lines corresponding to different coils, and the 
challenge is to reconstruct the complete magnetization image from the incomplete Fourier samples of these  coil images, see e.g.\ \cite{Block07,grappa,Hyun18,spirit,MOCCA,ESPIRIT}.

\smallskip

Recovery methods in MRI based on interpolation or approximation of non-aquired data in the Fourier domain as \cite{grappa,spirit}, or on subspace methods \cite{ESPIRIT}, are particularly based on  special sampling patterns  consisting of a bounded number of horizontal (or vertical) lines. 

In the discrete setting, we assume that an image ${\mathbf A}$ has to be recovered  from  a subsampled set of components of its discrete Fourier transform $\hat{\mathbf A}= {\mathbf F}_{N} {\mathbf A}{\mathbf F}_{M}$.
This reconstruction problem is ill-posed, since the Fourier basis is orthonormal, and the reconstruction is therefore not unique. 
In order to still achieve suitable reconstruction results, certain a priori assumptions on the image to be recovered are essential.

Considering  the image in a  vectorized form $\vec({\mathbf A})$, the reconstruction problem can also be rewritten as the problem to find a solution vector of an incomplete linear system, where only a subset of the equations of the system 
$({\mathbf F}_{N} \otimes {\mathbf F}_{M}) \vec({\mathbf A}) = \vec(\hat{\mathbf A})$ is available.
Here, one cannot assume that $\vec({\mathbf A})$ is a sparse vector or has a particularly small $1$-norm. Instead, one natural a priori assumption on the image ${\mathbf A}$ is that it is piecewise smooth.
In most reconstruction methods, this is transferred to a constraint that ${\mathbf A}$ has a small total variation and/or a sparse representation in some wavelet frame. These constraints can be incorporated into a minimization problem to cope for the missing Fourier data. Since the total variation constraint often  produces staircasing artifacts, many papers either focus on numerical examples with piecewise constant images (as ``phantom'' images) or try  to extend or to generalize the constraints on the image, for example by using generalized TV \cite{Knoll11}, special filters \cite{Gelb16,Xiao22}, or sparsity in wavelet frames \cite{Yang10}. 
\smallskip

Inspired by the special requirements  for MRI reconstructions, 
 we study in this paper the reconstruction from subsampled 2D DFT measurements for the special sampling grid which consists  of a fixed number of horizontal lines, i.e., we assume that only a certain number of rows of $\hat{\mathbf A}$ is available to recover ${\mathbf A}$.
We will investigate, how these rows should be taken, in order to achieve very good recovery results based on a suitable reconstruction method.
In particular, we will examine, how the special structure of the considered sampling pattern influences the reconstruction, and how the pattern can be exploited to improve the image recovery. 
We will  survey some currently used methods for image reconstruction based on our sampling pattern and propose a new hybrid method, which outperforms other approaches for natural images as well as cartoon-like images.

\paragraph{Notation and problem description.}
 Throughout the paper we will use  the following notations for index sets.  For $N \in {\mathbb N}$  we consider one-dimensional centered index sets 
$$ \Lambda_N := \left\{ \begin{array}{ll} \{-\frac{N}{2}, -\frac{N}{2}+1, \ldots , \frac{N}{2}-1 \} &  N \, \text{even}, \\[1ex]
\{-\frac{N-1}{2}, -\frac{N-1}{2}+1, \ldots , \frac{N-1}{2} \} & N \, \text{odd}, \end{array} \right. $$
with  $N$ indices
as well as, for even $K \in {\mathbb N}$, index sets containing only every second index, 
\begin{align}\label{setK} \Lambda_K^{(2)} := \{ -K+1, -K+3, \ldots , K-3, K-1 \}, \end{align}
with $K$ indices.
Further, 
$ \Lambda_{N,M} := \Lambda_N \times \Lambda_M$ denotes a two-dimensional index set  with $NM$ indices, 
and  $\Lambda_K^{(2)} \times \Lambda_M$ is a subset of $\Lambda_{2K,M}$ with $KM$ indices, where even-indexed rows of the full set $\Lambda_{2K,M}$ are omitted.
We will use the 2D index notation $\nuu=(\nu_{1}, \nu_{2}) \in \Lambda_{N,M}$ with $\nu_1\in \Lambda_N$, $\nu_2\in \Lambda_M$.

For a given discrete image $
{\mathbf A}=(a_{\mathbf k})_{{\mathbf k} \in \Lambda_{N,M}} = (a_{k_{1},k_{2}})_{k_{1}=-n,k_{2}=-m}^{n-1,m-1} \in {\mathbb C}^{N \times M}$ 
with $N,M$ being even positive integers with  $N=2n$ and $M= 2m$, the discrete two-dimensional Fourier transform is given by
$$ \hat{\mathbf A} = {\mathbf F}_{N} \, {\mathbf A} \, {\mathbf F}_{M}, $$
where ${\mathbf F}_{N} = \frac{1}{\sqrt{N}}(\omega_{N}^{jk})_{j,k=-n}^{n-1}$ and ${\mathbf F}_{M} = \frac{1}{\sqrt{M}} (\omega_{M}^{jk})_{j,k=-m}^{m-1}$ denote the centered unitary Fourier matrices, and $\omega_{N}:= {\mathrm e}^{-2\pi {\mathrm i}/N}$. 
The components $ \hat{a}_{\nuu}=\hat{a}_{\nu_{1},\nu_{2}}$ of $\hat{\mathbf A}$ have the form 
$$  \textstyle \hat{a}_{\nuu} = \hat{a}_{\nu_{1},\nu_{2}} = 
\frac{1}{\sqrt{MN}}\sum\limits_{k_{1}=-n}^{n-1} \sum\limits_{k_{2}=-m}^{m-1} a_{k_{1},k_{2}}\, \omega_{N}^{k_{1}\nu_{1}} \omega_{M}^{k_{2}\nu_{2}} . $$
The inverse two-dimensional discrete Fourier transform is  given by 
$ {\mathbf A} =  \overline{\mathbf F}_{N} \, \hat{\mathbf A} \, \overline{{\mathbf F}}_{M}$, 
since we have ${\mathbf F}_{N}^{-1} = \overline{\mathbf F}_{N}$.
Throughout the paper, we assume for simplicity that $N$ is a multiple of $8$, such that $\frac{n}{2}= \frac{N}{4}$ and $\frac{n}{4}  = \frac{N}{8}$  are integers.

Further, 
$\vec({\mathbf A})$ denotes the columnwise vectorization of a matrix ${\mathbf A}$, and ${\mathbf A} \otimes {\mathbf B}$ is the Kronecker product of ${\mathbf A} \in {\mathbb C}^{2n \times 2m}$, ${\mathbf B}\in {\mathbb C}^{N_{2} \times M_{2}}$ given by 
$${\mathbf A} \otimes {\mathbf B} = (a_{k_{1},k_{2}} {\mathbf B})_{k_{1}=-n,k_{2}=-m}^{n-1,m-1} \in {\mathbb C}^{2nN_{2} \times 2m M_{2}}.$$ 
The identity matrix is always denoted by ${\mathbf I}$, and its dimension follows from the context. 

\medskip

We will study the problem of how to reconstruct the image ${\mathbf A}$ from structured incomplete 2D DFT data 
$\hat{a}_{\nuu}$, $\nuu \in \Lambda \subset \Lambda_{N,M}$.
In particular, we assume that  only  $\lfloor \frac{N}{r} \rfloor$  rows of $\hat{\mathbf A}$ are available.
Here, $r \in {\mathbb N}$ is the \textit{reduction rate}. 
Since the low-pass data contain 
the most important image information, we will focus on a sampling mask given by 
\begin{equation}\label{set} \Lambda  \times \Lambda_M := (\Lambda_{L} \cup \Lambda_{K}^{(2)}) \times \Lambda_M,
\end{equation}
where $L=2\ell+1 \le \lfloor\frac{N}{r}\rfloor$ is odd, and $K$ is even. 
Here, $\Lambda: =\Lambda_{L} \cup \Lambda_{K}^{(2)}$ denotes the one-dimensional index set of acquired rows, where the \textit{low-pass set} $\Lambda_{L}$ contains the indices of centered rows.
Depending on a fixed reduction rate $r$, we assume that the number of all
acquired rows, i.e., the number of indices in $\Lambda$, is $\lfloor \frac{N}{r} \rfloor$ (or  $\lfloor \frac{N}{r} \rfloor-1$), i.e., we fix $K$ such that $L+ (K-\ell) = \lfloor \frac{N}{r} \rfloor$ (or $L+ (K-\ell) = \lfloor \frac{N}{r} \rfloor-1$).
In other words, outside the low-pass rows, $\lfloor \frac{N}{r} \rfloor-L$ further rows are acquired, namely only every second row symmetrically with respect to the image center, see Figure \ref{fig1} for $N=M=128$ and $L=11$. 
For $r=2$, the mask has $63$ rows, 
for $r=4$, it has $31$ rows, for $r=6$, $21$ rows and for $r=8$ only $15$ rows acquired.  
We will show in Section 5 that this mask gives better reconstruction results than taking different rows despite the low-pass set.
{\small
\begin{figure}[t]
\begin{center}
	\includegraphics[width=0.20\textwidth,height=0.20\textwidth]{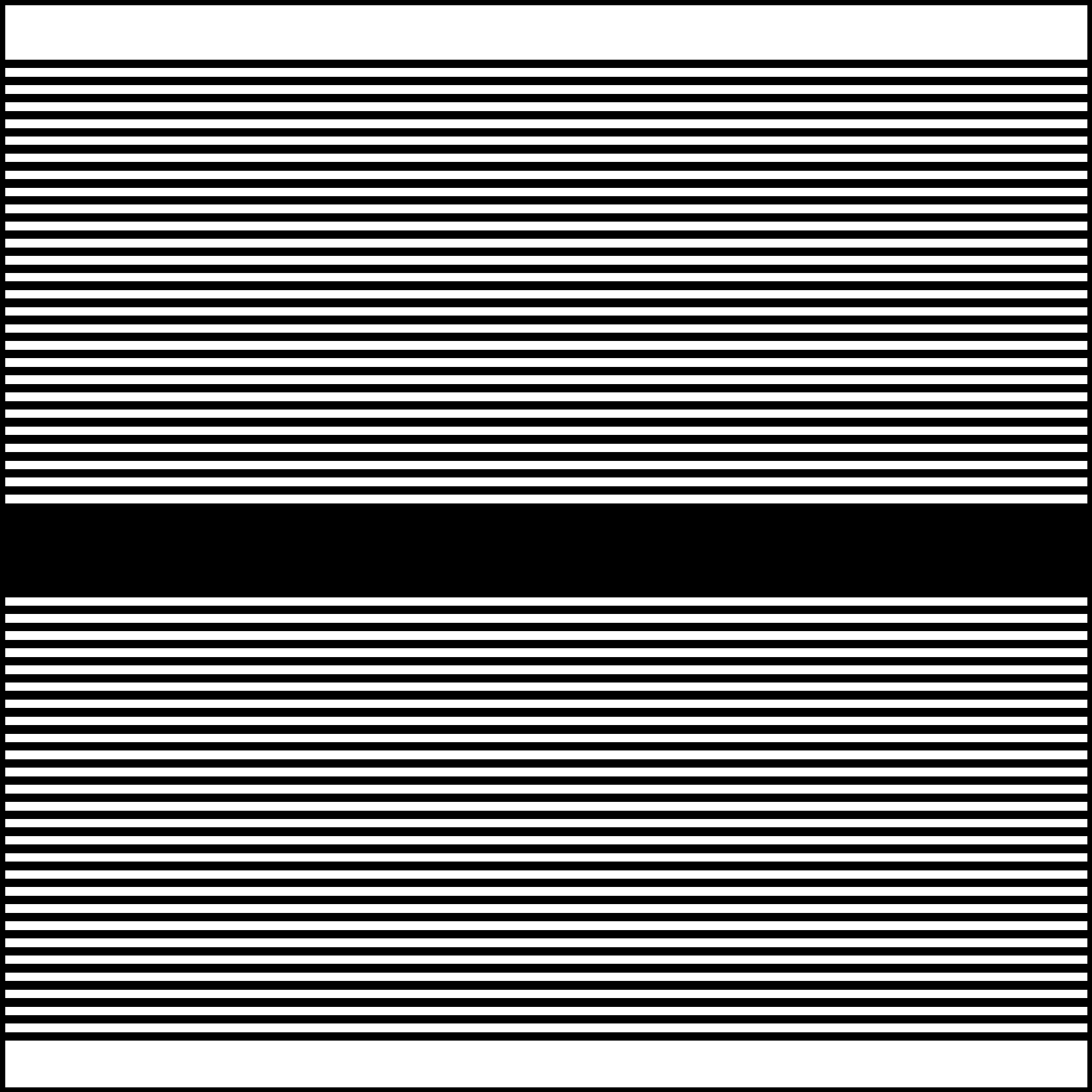} \hspace{1mm}
	\includegraphics[width=0.20\textwidth,height=0.20\textwidth]{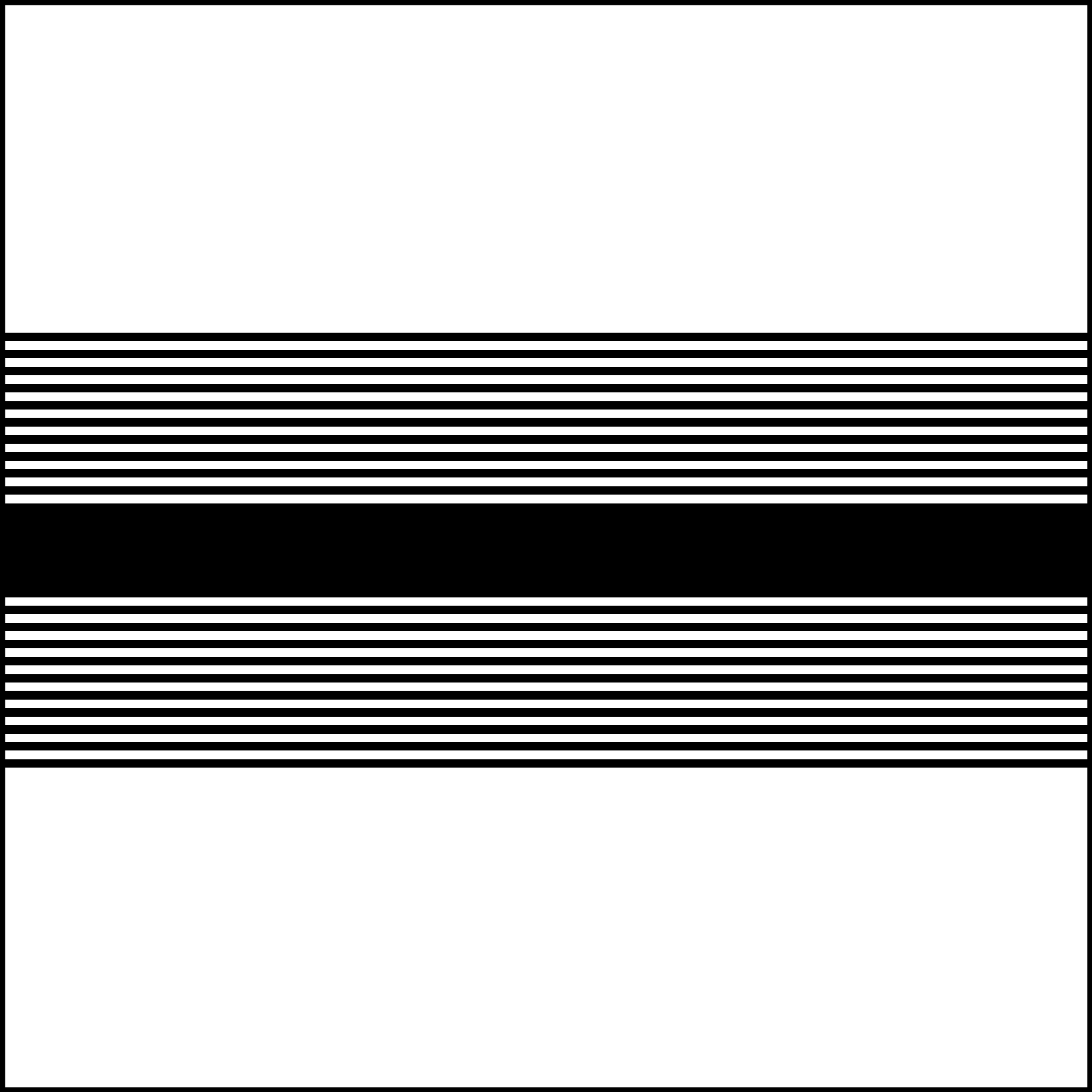} \hspace{1mm}
	\includegraphics[width=0.20\textwidth,height=0.20\textwidth]{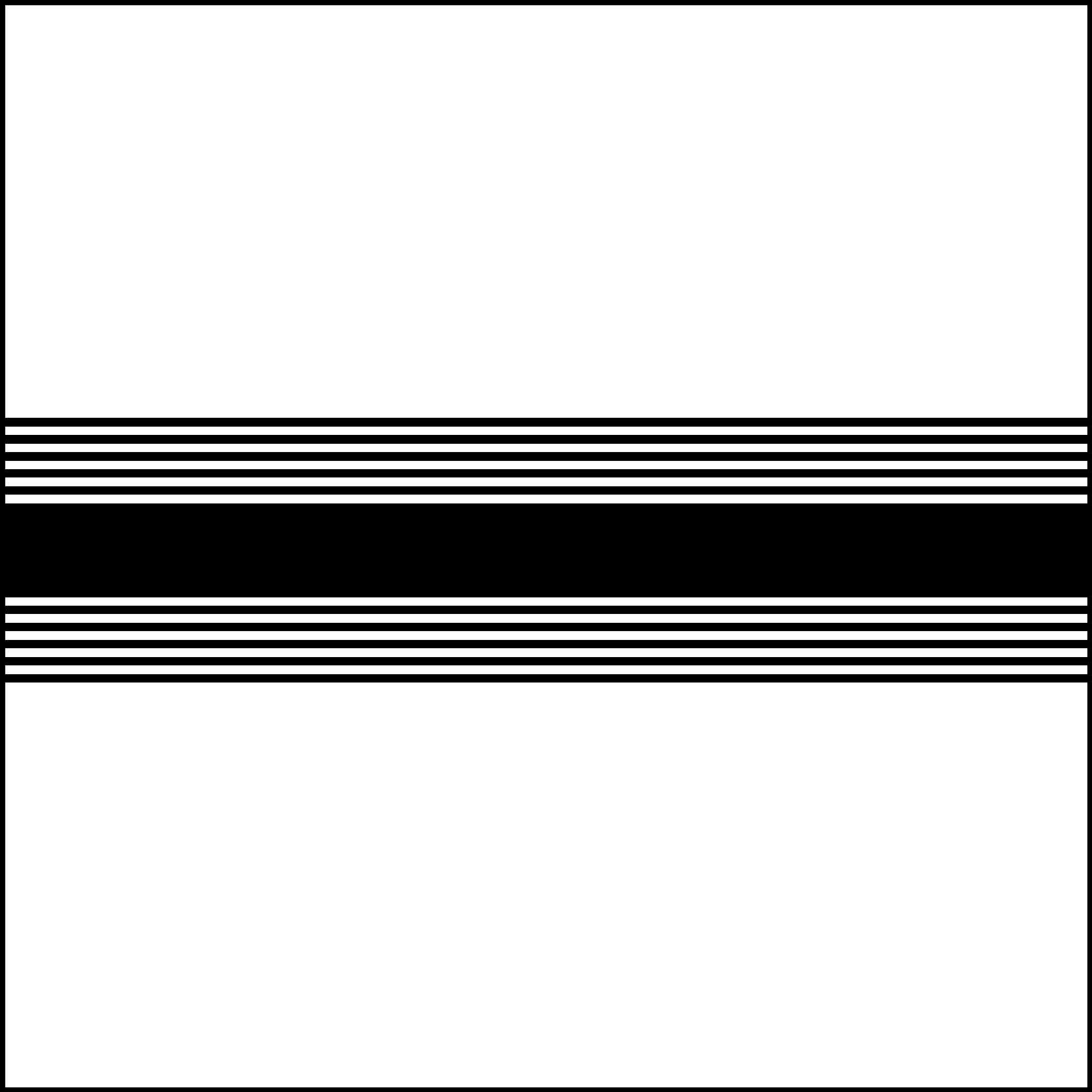} \hspace{1mm}
	\includegraphics[width=0.20\textwidth,height=0.20\textwidth]{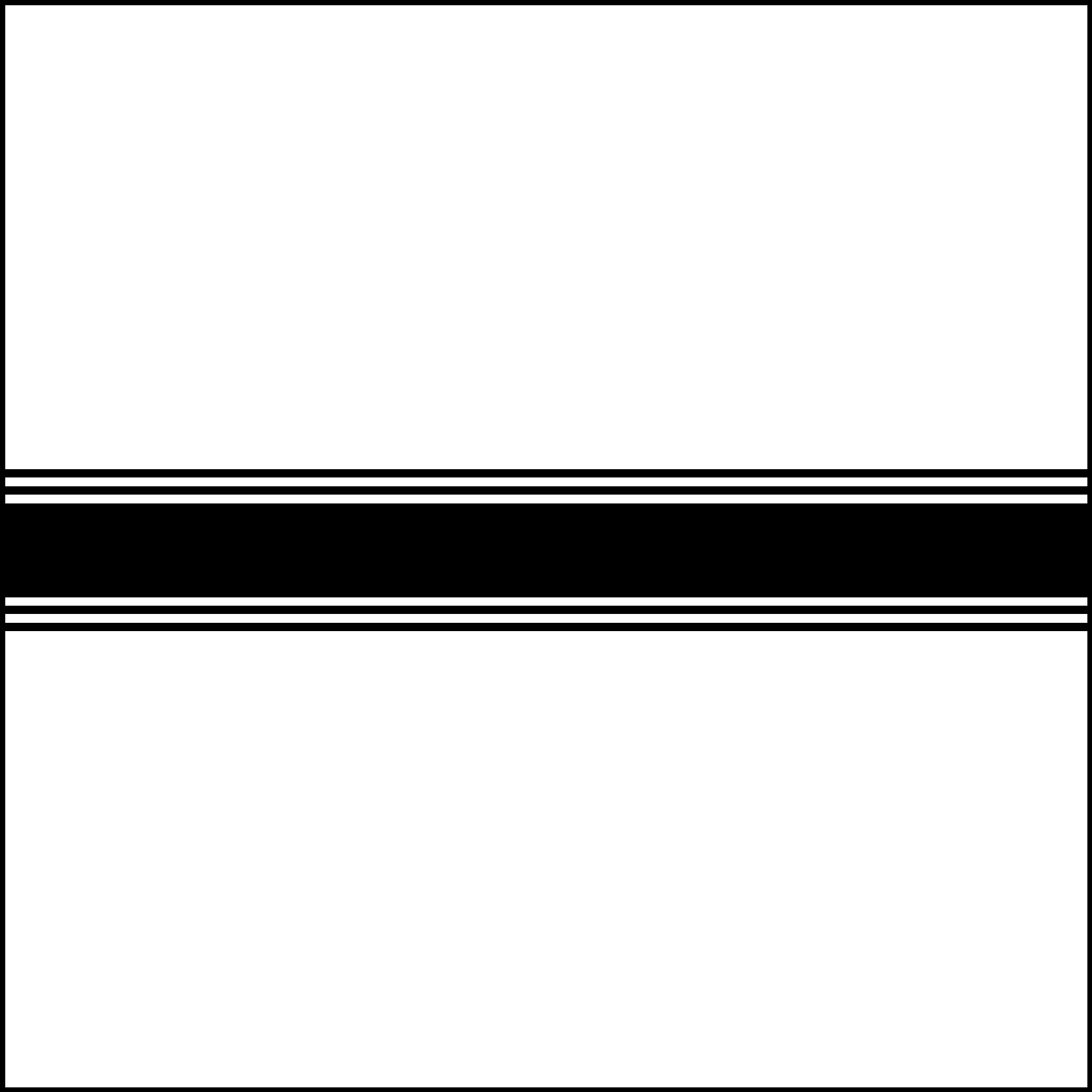}
	\end{center}
\caption{
Masks for acquired Fourier data for an $128 \times 128$ image with width $L=11$ of the low-pass set and reduction rates $r=2,4,6,8$, where black lines illustrate the acquired rows.}
\label{fig1}
\end{figure}
}

\noindent
We denote with ${\mathbf P}^{(\Lambda)}$ the \textit{projection matrix} (mask)  to the sampling set $\Lambda$ in (\ref{set}), i.e.,
\begin{align}\label{PL}
 {\mathbf P}^{(\Lambda)} =(p_{\nuu}^{(\Lambda)})_{\nuu \in \Lambda_{N,M}} = {\mathbf p}^{(\Lambda)} \, {\mathbf 1}^T
  \qquad \text{with} \qquad 
  p_{\nuu}^{(\Lambda)}=p_{\nu_1,\nu_2}^{(\Lambda)} := \left\{ \begin{array}{ll}1 &  \nu_1 \in \Lambda, \\
0 &  \nu_1 \in  \Lambda_N \setminus \Lambda. \end{array} \right.  
  \end{align}
In particular, $p_{\nu_1,\nu_2}^{(\Lambda)} =p_{\nu_1,1}^{(\Lambda)}$ for all $\nu_2 \in \Lambda_M$ and 
 ${\mathbf p}^{(\Lambda)} = (p_{\nu_1,1}^{(\Lambda)})_{\nu_1=-n}^n$. 
Thus, ${\mathbf P}^{(\Lambda)}$ contains rows of ones and zero rows,  see Figure \ref{fig1}.
Then we can formulate the reconstruction problem as follows:\\ Find an optimal reconstruction $\tilde{\mathbf A}$ of ${\mathbf A}$ from the incomplete 2D DFT data, i.e., with 
\begin{align} \label{aufgabe}
{\mathbf P}^{(\Lambda)} \circ \hat{\tilde{\mathbf A}}= {\mathbf P}^{(\Lambda)} \circ \hat{\mathbf A} = {\mathbf P}^{(\Lambda)} \circ ({\mathbf F}_{N} {\mathbf A} {\mathbf F}_{M}), 
\end{align}
where $\circ$ denotes the pointwise product.
The reconstruction quality will be measured using the peak signal to noise ratio PSNR$=10 \log_{10}\big( \frac{NM}{\|\tilde{\mathbf A} - {\mathbf A}\|_{2}^{2}}\big)$, where 
$\|\tilde{\mathbf A} - {\mathbf A}\|_{2} := 
\sum\limits_{\kk \in \Lambda_{N,M}} |\tilde{a}_{\kk} - a_{\kk}|^{2})^{1/2}
$.
\smallskip

\noindent
For complex images ${\mathbf A}= {\mathbf A}_{R} + {\mathrm i} {\mathbf A}_{I} = (a_{R,\kk}+ {\mathrm i}a_{I,\kk})_{\kk\in \Lambda_{N,M}}$  with ${\mathbf A}_{R}, {\mathbf A}_{I} \in {\mathbb R}^{N \times M}$,
we can reconstruct ${\mathbf A}_{R}$ and  ${\mathbf A}_{I}$ separately, since the sampling set is taken symmetrically, i.e., for acquired $\hat{a}_{\nuu}=\hat{a}_{\nu_{1},\nu_{2}}$, we also have acquired
 $\hat{a}_{-\nuu}=\hat{a}_{-\nu_{1},-\nu_{2}}$. For all $\nuu \in \Lambda_{{N,M}}$ we have
 \begin{align*}
\textstyle \frac{1}{2} (\hat{a}_{\nuu}+ \overline{\hat{a}_{-\nuu}}) &= \textstyle \frac{1}{2} \Big( \sum\limits_{\kk \in \Lambda_{N,M}} \!\!\!\!\!(a_{R,\kk}+ {\mathrm i} a_{I,\kk})  \omega^{\nu_1k_1}\omega_{M}^{\nu_2k_2}+ \sum\limits_{\kk \in \Lambda_{N,M}} \!\!\!\!\!\overline{(a_{R,\kk}+ {\mathrm i} a_{I,\kk})}    \omega^{\nu_1k_1}\omega_{M}^{\nu_2k_2} \Big)
 = \hat{a}_{R,\nuu}, \\
\textstyle \frac{1}{2} (\hat{a}_{\nuu}- \overline{\hat{a}_{-\nuu}}) &= \textstyle \frac{1}{2} \Big(\sum\limits_{\kk \in \Lambda_{N,M}} \!\!\!\!\!(a_{R,\kk}- {\mathrm i} a_{I,\kk}) \omega^{\nu_1k_1}\omega_{M}^{\nu_2k_2}+ \sum\limits_{\kk \in \Lambda_{N,M}} \!\!\!\!\!\overline{(a_{R,\kk}- {\mathrm i} a_{I,\kk})} \omega^{\nu_1k_1}\omega_{M}^{\nu_2k_2} \Big)
 = \hat{a}_{I,\nuu}. 
 \end{align*}
 In other words, to reconstruct a real image ${\mathbf A}$ we can reduce the index set  $\Lambda \times \Lambda_M$ in (\ref{set}) of acquired data to 
 $$   \Lambda'  \times \Lambda_M= \Big(\{0, 1, \ldots , \ell\} \cup \{1, 3, \ldots ,{K}-1\} \Big) \times \Lambda_{M}$$
 such that a reduction rate $r$ considered in this paper corresponds  in the real case to a reduction rate of almost $2r$.

\paragraph{Outline of this paper.}
In Section 2, we will explain in detail, why this reconstruction problem is indeed challenging. 
We will show, why a suitable choice of the width $L$ of the low-pass  set is crucial  to achieve a desired reconstruction quality.

\noindent
Observe that  the inverse Fourier transform of rows, 
 $({\mathbf P}^{(\Lambda)}\circ\hat{\mathbf A}) {\mathbf F}_M^{-1} = 
 {\mathbf P}^{(\Lambda)}  \circ  ({\mathbf F}_N  {\mathbf A})$, provides
 the incomplete columns  $ {\mathbf p}^{(\Lambda)} \circ  \hat{\mathbf a}_{\nu_2}$, $\nu_2=-m, \ldots , m$,  of ${\mathbf F}_N {\mathbf A}$, where ${\mathbf p}^{(\Lambda)}$ denotes one column of ${\mathbf P}^{(\Lambda)}$. 
Therefore, one may be tempted to simplify the reconstruction problem into separated one-dimensional reconstruction problems for every column of ${\mathbf A}$. 
We discuss the limitations of such 1D reconstruction approaches, particularly zero refilling and filtering with 1D  kernels, in Section 2.2.

\noindent
Also adaptive interpolation methods, such as GRAPPA  \cite{grappa} and SPIRiT \cite{spirit}, which  are  frequently applied in parallel MRI, fail to produce satisfactory reconstruction results, see Sections  2.3--2.4.

In Section 3, we propose an iterative reconstruction algorithm based on total variation minimization of the resulting image. 
The arising functional is minimized using the primal dual algorithm \cite{CP10}.
In Section 4, we present a new hybrid algorithm, which connects the TV-minimization reconstruction with an adaptive reconstruction improvement by exploiting the data knowledge that can be extracted from the special sampling pattern.

In Section 5, we show several examples of image reconstruction and show the very good performance of our hybrid algorithm for data reduction rates of up to 8 (that is,  $16$ in the real case). In particular, our hybrid algorithm always  improves all other reconstruction methods for the acquired set of Fourier data. We also show numerically that other sampling schemes, where besides the low-pass area every third or every fourth row of $\hat{\mathbf A}$ is acquired until the upper bound of $\lfloor \frac{N}{r} \rfloor$ is reached, do not lead to better reconstruction results. 
 Our approach can be extended to other structured sampling patterns.

\section{Challenges of the reconstruction problem}
\label{sec2}

 In this section, we first summarize  some facts  about the reconstruction challenges  from incomplete  2D-DFT data for the described sampling pattern. Then we  will give a short overview  of simple reconstruction ideas and show their weaknesses  for the reconstruction problem at hand.
 
\subsection{Importance of the low-pass set}
As introduced in (\ref{set}), we assume that acquired Fourier data are given with respect to the index set 
$\Lambda \times \Lambda_M$, which consists of the union of the low-pass set 
$\Lambda_{L,M}= \Lambda_L \times \Lambda_M$ containing the rows of $\hat{\mathbf A}$ with indices $-\ell, \ldots, \ell$ and the set $\Lambda_{K}^{(2)} \times \Lambda_M$ containing each second (odd-indexed) row of 
$\hat{\mathbf A}$ until the bound of $\lfloor N/r\rfloor$  acquired rows is reached. Assume for a moment that $L=0$ and $r=2$, i.e., $\Lambda= \Lambda_N^{(2)}$,
such that the index set corresponding to acquired entries is given by  $\Lambda_{N}^{(2)} \times \Lambda_M$.
In this case,  every second row of $\hat{\mathbf A}$ is given, i.e., 
$\hat{a}_{2\nu_{1}+1,\nu_{2}}$ for $\nu_{1}=-\frac{n}{2}, \ldots , \frac{n}{2}-1$, $\nu_{2}=-m, \ldots , m$.
With this knowledge, we can recover the differences $a_{k_{1},k_{2}}- a_{k_{1}+n,k_{2}}$ of the components in ${\mathbf A}= (a_{k_{1},k_{2}})_{k_{1}=-n,k_{2}=-m}^{n-1,m-1}$ exactly for $k_{1}=0, \ldots , n-1$, $k_{2}=-m, \ldots , m$, since 
\begin{align}
\nonumber
\textstyle \frac{1}{2} (a_{k_{1},k_{2}}- a_{k_{1}-n,k_{2}})
&= \textstyle  \frac{1}{2\sqrt{MN}} \sum\limits_{\nu_{2}=-m}^{m-1} \Big(\sum\limits_{\nu_{1} = -n}^{n-1}  \hat{a}_{\nu_{1},\nu_{2}} (\omega_{N}^{-k_{1}\nu_{1}}   - \omega_{N}^{-(k_{1}-n) \nu_{1}})\Big) \, \omega_{M}^{-k_{2}\nu_{2}}\\
\nonumber
&= \textstyle  \frac{1}{2\sqrt{MN}} \sum\limits_{\nu_{1} = -n}^{n-1} \sum\limits_{\nu_{2}=-m}^{m-1} \hat{a}_{\nu_{1},\nu_{2}} (\omega_{N}^{-k_{1}\nu_{1}} - (-1)^{\nu_{1}} \omega_{N}^{-k_{1}\nu_{1}}) \, \omega_{M}^{-k_{2}\nu_{2}} \\
&= \textstyle 
\frac{1}{\sqrt{MN}} \sum\limits_{\nu_{1} = -n/2}^{n/2-1} \sum\limits_{\nu_{2}=-m}^{m-1} \hat{a}_{2\nu_{1}+1,\nu_{2}} \omega_{N}^{-k_{1}(2\nu_{1}+1)} \omega_{M}^{-k_{2}\nu_{2}}.
\label{two}
\end{align}
In other words, considering the upper and the lower half ${\mathbf A}_1, {\mathbf A}_2 \in {\mathbb R}^{n \times M}$ of ${\mathbf A}= \Big( \!\!\begin{array}{c} {\mathbf A}_1 \\  {\mathbf A}_2 \end{array} \!\! \Big)$, the given Fourier samples provide the difference ${\mathbf A}_1 - {\mathbf A}_2$.
Unfortunately, the remaining information needed to recover ${\mathbf A}$, namely the sums $(a_{k_{1},k_{2}}+ a_{k_{1}+n,k_{2}})$
for $k_{1}=0, \ldots , n-1$, $k_{2}=-m, \ldots , m$,  or equivalently ${\mathbf A}_1 + {\mathbf A}_2$, only depends on the non-acquired even-indexed Fourier data $\hat{a}_{2\nu_{1},\nu_{2}}$, and cannot be reconstructed.

\noindent
 We investigate what can be gained by a local interpolation scheme. Let 
${\mathbf W} =(\w_{\nuu})_{\nuu \in \Lambda_{n,M}} \in {\mathbb C}^{n \times M}$
be a matrix of arbitrary complex weights and 
$\check{\mathbf W} := (\check{\w}_{\kk})_{\kk \in \Lambda_{n,M}} = \sqrt{Mn} \, {\mathbf F}_{n}^{-1} {\mathbf W} {\mathbf F}_{M}^{-1}$
its $2D$ inverse Fourier transform multiplied  with $\sqrt{Mn}$.\\
Using a local interpolation scheme,  one wants to approximate a missing Fourier value from the acquired Fourier values in a local neighborhood. 
Assume that the interpolation scheme is of the form 
\begin{align} \nonumber
\hat{\tilde{a}}_{2\nu_{1}+1, \nu_{2}}& :=\textstyle \hat{a}_{2\nu_{1}+1, \nu_{2}} 
\\
\label{interpol}
\hat{\tilde{a}}_{2\nu_{1}, \nu_{2}} &:=  \textstyle \!\sum\limits_{\ell_{1}=-\frac{n}{2}}^{\frac{n}{2}-1}
\sum\limits_{\ell_{2}=-m}^{m-1} \w_{\nu_{1}-\ell_{1}, \nu_{2}-\ell_{2}} \, \hat{a}_{2\ell_{1}+1,\ell_{2}}, 
\end{align}
for $\nu_{1}=-\frac{n}{2}, \ldots , \frac{n}{2}-1$, $\nu_{2}=-m, \ldots , m-1$,
where for local approximation of unknown Fourier data 
the weight matrix ${\mathbf W}$ has nonzero entries $\w_{\nu_1,\nu_2}$ in a window around  the center ${\mathbf 0}$.
In (\ref{interpol}) we assume periodicity of the weights in ${\mathbf W}$ 
with 
$\w_{\nu_{1}+ \ell_{1}n, \nu_{2}+ \ell_{2}M}= \w_{\nu_{1}, \nu_{2}}$ for all $\ell_{1}, \ell_{2} \in {\mathbb Z}$, $\nu_{1}=-\frac{n}{2}, \ldots , \frac{n}{2}-1$, $\nu_{2}= -m, \ldots , m-1$. 
We can  actually prove the following theorem for the interpolation scheme (\ref{interpol}).

\begin{theorem}\label{thm:linear}
Let 
 ${\mathbf A}=(a_{\kk})_{\kk \in \Lambda_{N,M}} \in {\mathbb C}^{N \times M}$, 
$\hat{\mathbf A}=(\hat{a}_{\nuu})_{\nuu \in \Lambda_{N,M}} \in {\mathbb C}^{N \times M}$
its $2$-D Fourier transform, and assume that half of the Fourier data, namely 
$$ \textstyle \hat{a}_{2\nu_{1}+1,\nu_{2}}, \qquad \nu_{1}=-\frac{n}{2}, \ldots , \frac{n}{2}-1, \; \nu_{2}=-m, \ldots , m-1,$$ 
are acquired. Let  
${\mathbf W} =(\w_{\nuu})_{\nuu \in \Lambda_{n,M}}$
be a matrix of arbitrary complex weights as given above.
 Then, the interpolation scheme $(\ref{interpol})$
leads to an image approximation $\tilde{\mathbf A}= (\tilde{a}_{k_{1},k_{2}})_{k_{1}=-n,k_{2}=m}^{n-1,m-1}$, where 
\begin{align}\label{prod}
\tilde{a}_{k_{1}, k_{2}} = \textstyle \frac{1}{2} ( 1 +  \check{\w}_{k_{1},k_{2}}  \, \omega_{N}^{k_{1}} ) ({a}_{k_{1}, k_{2}}- {a}_{k_{1}-n, k_{2}})
\end{align}
for all $k_{1}=0, \ldots , n-1$, $k_{2}=-m, \ldots , m-1$.
\end{theorem}

\begin{proof}
Taking the interpolation scheme (\ref{interpol}), it follows with (\ref{two}) for the components of $\tilde{\mathbf A} = {\mathbf F}_{N}^{-1} \hat{\tilde{\mathbf A}} {\mathbf F}_{M}^{-1}$ that 
\begin{align*}
&\tilde{a}_{k_{1},k_{2}} = \textstyle \frac{1}{\sqrt{MN}}\Big(
\sum\limits_{\nu_{1} = -\frac{n}{2}}^{\frac{n}{2}-1} \sum\limits_{\nu_{2}=-m}^{m-1} \!\! \hat{a}_{2\nu_{1}+1,\nu_{2}} \omega_{N}^{-k_{1}(2\nu_{1}+1)} \omega_{M}^{-k_{2}\nu_{2}}
+ \!\!\!
\sum\limits_{\nu_{1} = -\frac{n}{2}}^{\frac{n}{2}-1} \sum\limits_{\nu_{2}=-m}^{m-1} \!\!\hat{\tilde{a}}_{2\nu_{1},\nu_{2}} \omega_{N}^{-2k_{1}\nu_{1}} \omega_{M}^{-k_{2}\nu_{2}}\Big) \\
 &= \textstyle \frac{1}{2}(a_{k_{1},k_{2}} - a_{k_{1}-n,k_{2}}) +  \frac{1}{\sqrt{MN}} \!\!\sum\limits_{\nu_{1} = -\frac{n}{2}}^{\frac{n}{2}-1} \sum\limits_{\nu_{2}=-m}^{m-1} 
\sum\limits_{\ell_{1}=-\frac{n}{2}}^{\frac{n}{2}-1}
\sum\limits_{\ell_{2}=-m}^{m-1}\!\!\!\w_{\nu_{1}-\ell_{1}, \nu_{2}-\ell_{2}} \, \hat{a}_{2\ell_{1}+1,\ell_{2}} 
 \omega_{N}^{-2k_{1}\nu_{1}} \omega_{M}^{-k_{2}\nu_{2}}\\
  &= \textstyle \frac{1}{2}(a_{k_{1},k_{2}} - a_{k_{1}-n,k_{2}}) \\
  & \textstyle  + 
  \frac{1}{\sqrt{MN}} \!\!
  \sum\limits_{\ell_{1}=-\frac{n}{2}}^{\frac{n}{2}-1}
\sum\limits_{\ell_{2}=-m}^{m-1}  \hat{a}_{2\ell_{1}+1,\ell_{2}}\, 
\omega_{N}^{-2\ell_{1}k_{1}} \omega_{M}^{-\ell_{2}k_{2}} \!\!
\sum\limits_{\nu_{1} = -\frac{n}{2}}^{\frac{n}{2}-1} \sum\limits_{\nu_{2}=-m}^{m-1} \!\!
\w_{\nu_{1}-\ell_{1}, \nu_{2}-\ell_{2}} 
 \omega_{N}^{-2k_{1}(\nu_{1}-\ell_{1)}} \omega_{M}^{-k_{2}(\nu_{2}-\ell_{2})}
 \\
  &= \textstyle \frac{1}{2}(a_{k_{1},k_{2}} - a_{k_{1}-n,k_{2}}) + 
  \frac{1}{\sqrt{MN}} 
  \Big(\sum\limits_{\ell_{1}=-\frac{n}{2}}^{\frac{n}{2}-1}
\sum\limits_{\ell_{2}=-m}^{m-1}  \hat{a}_{2\ell_{1}+1,\ell_{2}}
\omega_{N}^{-2\ell_{1}k_{1}} \omega_{M}^{-\ell_{2}k_{2}} \Big)\, 
\check{\w}_{k_{1},k_{2}} 
 \\
  &= \textstyle \frac{1}{2}(a_{k_{1},k_{2}} - a_{k_{1}-n,k_{2}}) + 
  \frac{1}{\sqrt{MN}} \check{\w}_{k_{1},k_{2}} \, \omega_{N}^{k_{1}} \, 
  \Big(\sum\limits_{\ell_{1}=-\frac{n}{2}}^{\frac{n}{2}-1}
\sum\limits_{\ell_{2}=-m}^{m-1}  \hat{a}_{2\ell_{1}+1,\ell_{2}}
\omega_{N}^{-(2\ell_{1}+1)k_{1}} \omega_{M}^{-\ell_{2}k_{2}} \Big)\, 
 \\
 &= \textstyle \frac{1}{2} ( 1 +  \check{\w}_{k_{1},k_{2}}  \, \omega_{N}^{k_{1}} ) (a_{k_{1},k_{2}} - a_{k_{1}-n,k_{2}}),\end{align*}
 where $ \check{\mathbf W}=(\check{\w}_{k_{1},k_{2}})_{k_{1}=-\frac{n}{2},k_{2}=-m}^{\frac{n}{2}-1,m-1}
 = \sqrt{Mn} {\mathbf F}_{N}^{-1} {\mathbf W} {\mathbf F}_{M}$ as defined above. \end{proof}
\medskip

Formula (\ref{prod}) shows that we are in fact unable to find a good  reconstruction of an image only from the data 
 corresponding to the index set  $\Lambda_{N}^{(2)} \times \Lambda_{M}$, i.e., for the sampling pattern, where every second row of the image is missing. Any (local) interpolation scheme does not get rid of the problem that we obtain just twice a (scaled) version of the difference of the upper and the lower half of the image.
 In other words, the low-pass set is crucial for the image reconstruction.
Therefore, in the remaining sections, we always assume that $L=2\ell+1>0$, i.e., we have a certain low-pass part of the image which is completely acquired.

\begin{remark}
Obviously, a similar observation  as in Theorem $\ref{thm:linear}$ can be shown if instead of all odd-indexed rows of $\hat{\mathbf A}$ all even-indexed rows are acquired. Then, we have for $k_{1}=0, \ldots , n-1$, $k_{2}=-m, \ldots , m-1$,
$$ \textstyle a_{k_{1},k_{2}}+ a_{k_{1}-n,k_{2}}= \frac{2}{\sqrt{MN}}
\sum\limits_{\nu_{1} = -\frac{n}{2}}^{\frac{n}{2}-1} \sum\limits_{\nu_{2}=-m}^{m-1} \hat{a}_{2\nu_{1},\nu_{2}} \omega_{N}^{-2k_{1}\nu_{1}} \omega_{M}^{-k_{2}\nu_{2}},
$$ 
and every interpolation scheme for the unacquired Fourier values of the form 
$$ \hat{\tilde{a}}_{2\nu_{1}+1, \nu_{2}} :=  \textstyle \!\sum\limits_{\ell_{1}=-\frac{n}{2}}^{\frac{n}{2}-1}\!
\sum\limits_{\ell_{2}=-m}^{m-1} \!\w_{\nu_{1}-\ell_{1}, \nu_{2}-\ell_{2}} \, \hat{a}_{2\ell_{1},\ell_{2}}, \, \nu_{1}=-\frac{n}{2}, {\ldots}, \frac{n}{2}-1, \, \nu_{2}=-m, {\ldots} , m-1,
$$ leads to a reconstruction $\tilde{\mathbf A}$, where
$\tilde{a}_{k_{1}, k_{2}} = \textstyle \frac{1}{2} ( 1 +  \omega_{N}^{-k_{1}}\check{\w}_{k_{1},k_{2}}) ({a}_{k_{1}, k_{2}}+ {a}_{k_{1}-n, k_{2}})$
for all $k_{1}=0, \ldots , n-1$, $k_{2}=-m, \ldots , m-1$. In this case, we do not get rid of the factor $({a}_{k_{1}, k_{2}}+ {a}_{k_{1}-n, k_{2}})$.
\end{remark}
\medskip

\subsection{Zero refilling  and one-dimensional low-pass filtering}
\label{ssec:zero}

\paragraph{Zero refilling.}
The simplest approach  to reconstruct ${\mathbf A}$ from the incomplete Fourier data 
${\mathbf P}^{(\Lambda)} \circ \hat{\mathbf A}$ is 
to replace the missing Fourier data by zero before applying the inverse 2D Fourier transform, i.e., 
\begin{align}\label{refill} \tilde{\mathbf A}^{(z)} := {\mathbf F}_{N}^{-1} ( {\mathbf P}^{(\Lambda)} \circ  \hat{\mathbf A} ) {\mathbf F}_{M}^{-1}  =  {\mathbf F}_{N}^{-1} ( {\mathbf P}^{(\Lambda)} \circ ( {\mathbf F}_{N} {\mathbf A})).
\end{align}
For the columns $\tilde{\mathbf a}_{k_2}^{(z)}$, $k_2=-m, \ldots, m-1$,  of the reconstruction image $\tilde{\mathbf A}^{(z)}$
we obtain 
\begin{align*} 
 \tilde{\mathbf a}_{k_2}^{(z)} = \textstyle  {\mathbf F}_N^{-1} ( {\mathbf p}^{(\Lambda)} \circ \hat{{\mathbf a}}_{k_2})
= {\mathbf F}_N^{-1} \hat{{\mathbf a}}_{k_2} +  {\mathbf F}_N^{-1} (({\mathbf p}^{(\Lambda)}  - {\mathbf 1}) \circ \hat{{\mathbf a}}_{k_2}) = {\mathbf a}_{k_2} +  {\mathbf F}_N^{-1} (({\mathbf p}^{(\Lambda)}  - {\mathbf 1}) \circ \hat{{\mathbf a}}_{k_2}) 
\end{align*}
with ${\mathbf 1}$ being the vector of ones of length $N$ and ${\mathbf p}^{(\Lambda)}$ a column of ${\mathbf P}^{(\Lambda)}$ in (\ref{PL}).
The representation of 
$ \tilde{\mathbf a}_{k_2}^{(z)}$  
can also be understood as a convolution in space domain, 
\begin{align*} \tilde{\mathbf a}_{k_2}^{(z)} &=  \textstyle 
{\mathbf F}_N^{-1} ({\mathbf p}^{(\Lambda)} \circ \hat{{\mathbf a}}_{k_2})
={\mathbf F}_N^{-1} {\mathbf p}^{(\Lambda)}  \star {\mathbf F}_N^{-1} \hat{\mathbf a}_{k_2} 
= \check{\mathbf p}^{(\Lambda)} \star {\mathbf a}_{k_2} 
= \frac{1}{\sqrt{N}} \big(\sum\limits_{\nu_1 \in \Lambda} \omega_N^{-k_1\nu_1}\big)_{k_1=-n}^n \, \star {\mathbf a}_{k_2} \\
&=  \textstyle \frac{1}{\sqrt{N}} \big(\sum\limits_{\nu_1 \in \Lambda} 
{a}_{\nu_1,k_2} \omega_N^{-\nu_1k_1} \big)_{k_1=-n}^n. 
\end{align*}
For every column $k_2$, the occurring  error has the form 
\begin{align}\label{error}  \textstyle \| \tilde{\mathbf a}_{k_2} - {\mathbf a}_{k_2}\|_2^2 = 
\| {\mathbf F}_N^{-1} ({\mathbf p}^{(\Lambda)}  - {\mathbf 1}) \circ \hat{{\mathbf a}}_{k_2}) \|_{2}^{2} =
\sum\limits_{\nu_1 \in \Lambda_N \setminus \Lambda} 
|\hat{{a}}_{\nu_1,k_2}|^{2}, 
\end{align}
i.e., the smaller the sum  of squares of all missing  Fourier values, 
the smaller the reconstruction error. Obviously, if the non-acquired Fourier values are zero, then the error vanishes.

\paragraph{One-dimensional low-pass filtering with $\Lambda= \Lambda_L$.}
Natural images  ${\mathbf A}$ usually possess some smoothness properties, which can in the discrete case be measured by local variations of pixel values  and lead to decaying Fourier values corresponding to higher frequencies.
 Therefore the acquired low-pass Fourier data contain  most of  the information and 
 the special choice $\Lambda= \Lambda_{L}$ with $L= \big\lfloor \frac{N}{r} \big\rfloor$ (or $L= \big\lfloor \frac{N}{r} \big\rfloor-1$ if $\big\lfloor \frac{N}{r} \big\rfloor$ is even), i.e., a sampling set consisting only of  centered rows seems to be favorable for the application of zero refilling.
In this case, $\check{\mathbf p}^{(\Lambda)}=\check{\mathbf p}^{(\Lambda_{L})}$ has the components 
$\sum_{j=-\ell}^{\ell} \omega_{N}^{-j k}$, $k=-n, \ldots, n-1$, which are samples of the Dirichlet kernel $D_{\ell}(\omega) = \sum_{j=-\ell}^{\ell} {\mathrm e}^{{\mathrm i}\omega j}$,  i.e., $\check{\mathbf p}^{(\Lambda_L)} = (D_{\ell}(\frac{2\pi k}{N}))_{k \in \Lambda_N}$.

The  estimate (\ref{error}) shows that 
${\mathbf p}^{(\Lambda)}$ is optimal with regard to the $2$-error compared to any other vector 
${\mathbf q}^{(\Lambda)}$ with support $\Lambda$ and arbitrary nonzero components for indices in $\Lambda$.
However, the aliasing effects, which always occur for reconstruction with the Dirichlet window vector ${\mathbf p}^{(\Lambda)} ={\mathbf p}^{(\Lambda_L)}$, can be reduced by taking other windows, as e.g.\ the Hamming window, which is defined by 
${\mathbf p}^{(Ham)} = (p_{k})_{k=-n}^{n-1}$, where
\begin{align*}
p_{k}^{(Ham)} := \left\{ \begin{array}{ll}
0.54 + 0.46 \cos(\frac{2\pi k}{2\ell}) & -\ell \le k \le \ell.\\
0 & \text{otherwise.} \end{array} \right.
\end{align*}
While the application of the Hamming window instead of the Dirichlet window avoids the special ringing artefacts, we observe a strong oversmoothing.
 Figure \ref{figure1} 
shows  the zero refilling reconstruction of the $512 \times 512$ pepper image 
for reduction rate $r=6$,  and $L=85= \big\lfloor \frac{512}{6} \big\rfloor $ (i.e., $\Lambda= \Lambda_{L} =\{-42, \ldots , 42\}$  and Dirichlet window ${\mathbf p}^{(\Lambda)}$), 
the reconstruction  with the Hamming window
${\mathbf p}^{(Ham)}$ for   $r=6$ and $L=85$ , and 
 the zero filling reconstruction with the sampling set $\Lambda$ in (\ref{set}) with $r=6$ and $L=43$, 
i.e., $\Lambda=\Lambda_{43} \cup \Lambda^{(2)}_{62}$.
The bottom row shows the corresponding  (normalized) vectors  $\frac{1}{\|\check{\mathbf p}^{(\Lambda_{L})}\|_{\infty}} \check{\mathbf p}^{(\Lambda_{L})}$,
$\frac{1}{\|\check{\mathbf p}^{(Ham)}\|_{\infty}}\check{\mathbf p}^{(Ham)}$ and $\frac{1}{\|\check{\mathbf p}^{(\Lambda)}\|_{\infty}}\check{\mathbf p}^{(\Lambda)}$, which are approximations of the ideal window $(\delta_{0,k})_{k=-n}^{n-1}
= \frac{1}{\sqrt{N}} {\mathbf F}_N^{-1} {\mathbf 1}$ that we would obtain  for complete Fourier data.
In Theorem \ref{thm:linear} we had shown that for $L=0$, a non-adaptive scheme always leads to a reconstruction, where $\tilde{a}_{k_1,k_2}$ contains the factor $a_{k_1,k_2} - a_{k_1-n,k_2}$. 
This effect is only partially reduced by employing an index set 
$\Lambda= \Lambda_L \cup \Lambda_K^{(2)}$ with $L>0$ and large $K$, where still the high magnitude of boundary coefficients of $\check{\mathbf p}^{(\Lambda)}$ can be recognized, see  Figure \ref{figure1} (right).

Using the flexibility of the sampling set $\Lambda$ in (\ref{set}), where we can fix the width $L \in \{0, \ldots , \lfloor \frac{N}{r} \rfloor\}$ of the low-pass set, we are interested in data-adaptive algorithms that significantly improve the recovery results achieved by low-pass reconstruction using zero refilling.

\begin{figure}[t]
\begin{center} \hspace*{-5mm}
	\includegraphics[width=0.29\textwidth,height=0.29\textwidth]{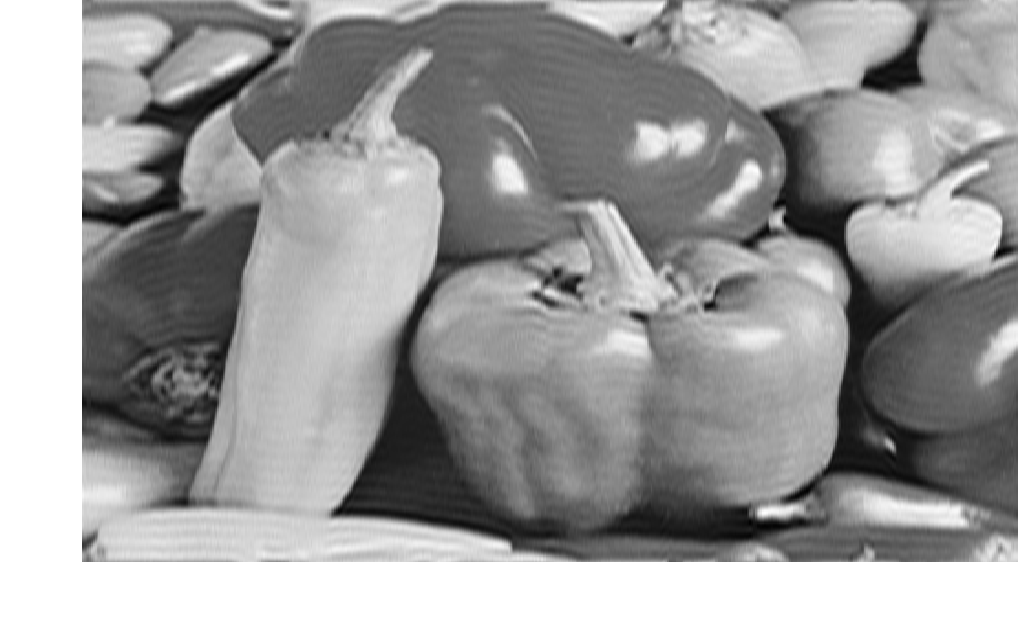} 
	\includegraphics[width=0.29\textwidth,height=0.29\textwidth]{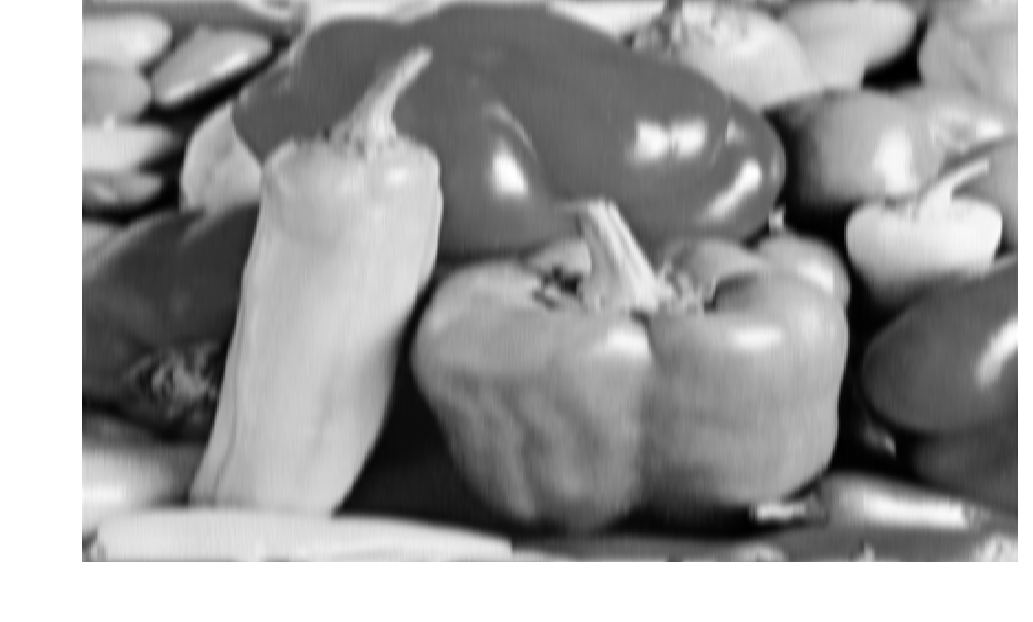} 
    \includegraphics[width=0.29\textwidth,height=0.29\textwidth]{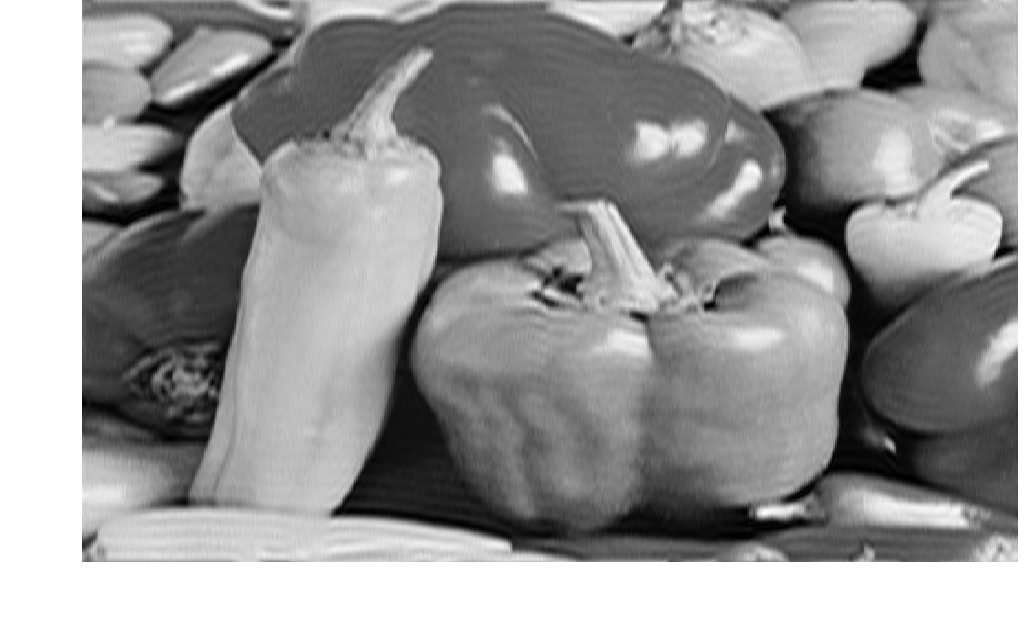}\\
       \includegraphics[scale=0.45]{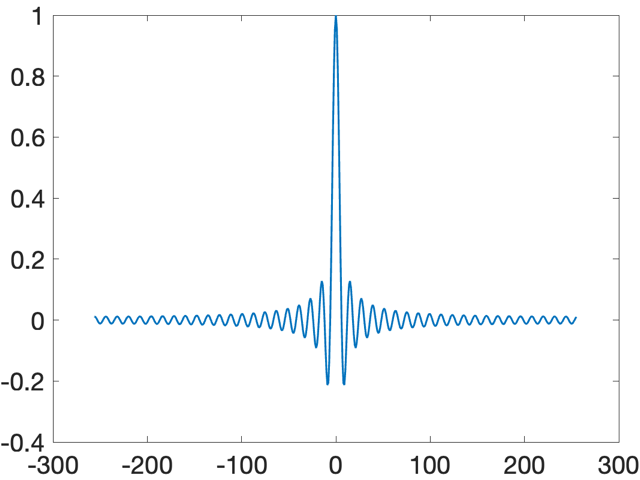}
       \includegraphics[scale=0.45]{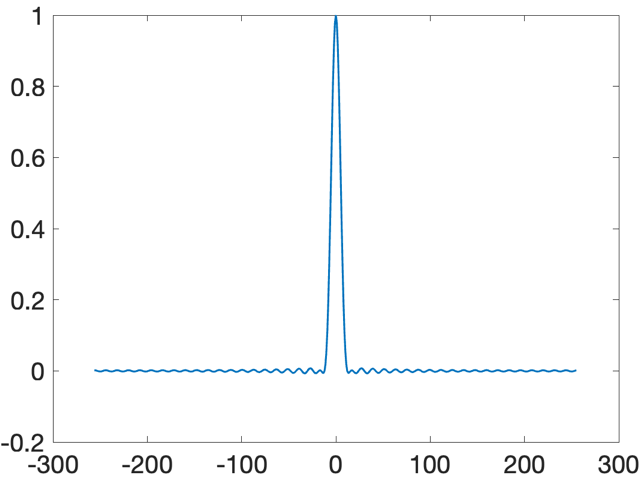}
     \includegraphics[scale=0.45]{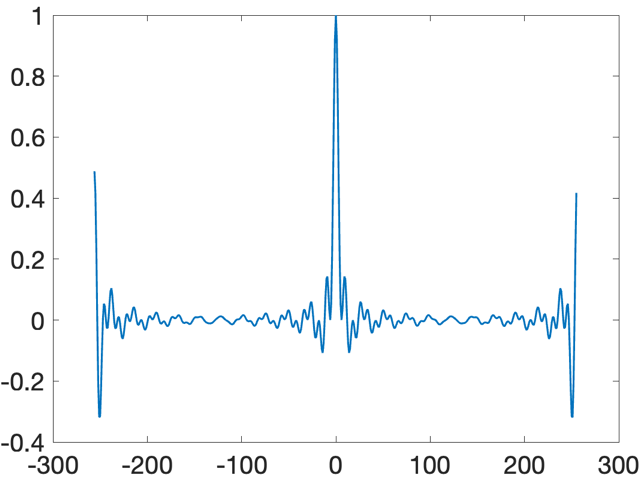}
	\end{center}
\caption{
Top: Reconstructions of the $512 \times 512$ pepper image from Fourier data with reduction rate $r=6$.
Left: 
Approximation with Dirichlet window using $L=\lfloor512/r \rfloor=85$ with PSNR $28.21$; 
Middle: Approximation with Hamming window using $L=85$  with PSNR $25.57$;
Right: Reconstruction by zero refilling with reduction rate  $r=6$ and $L=43$ with PSNR $26.75$.
Bottom: Corresponding representations of the windows $\check{\mathbf p}^{(\Lambda_{L})}$, 
$\check{\mathbf p}^{(Ham)}$  (middle), and $\check{\mathbf p}^{(\Lambda_{1})}$. }
\label{figure1}
\end{figure}

\subsection{Non-adaptive reconstruction approach}
\label{non-adaptive}
Next we shortly show, why a non-data-adaptive reconstruction approach does generally not yield recovery results which outperform the simple zero refilling procedure.

Recall that we want to find an optimal approximation $\tilde{\mathbf A}$ of ${\mathbf A}$ from ${\mathbf P}^{(\Lambda)} \circ \hat{\mathbf A}$. Equivalently, we need to approximate the non-acquired Fourier data to get an optimal approximation $\hat{\tilde{\mathbf A}}$ from ${\mathbf P}^{(\Lambda)} \circ \hat{\mathbf A}$. 

Using a model based on a linear transform, we can rewrite the (non-adaptive) reconstruction problem in a vectorized form as follows: Find  a matrix ${\mathbf Q} \in {\mathbb C}^{NM \times NM}$ to recover $\vec(\tilde{\mathbf A})$ by 
$$\vec(\hat{\tilde{\mathbf A}}) = {\mathbf Q} \, \vec({\mathbf P}^{(\Lambda)} \circ \hat{\mathbf A})
= {\mathbf Q} \, (\vec({\mathbf P}^{(\Lambda)})\circ \vec(\hat{\mathbf A})) = 
{\mathbf Q} \, \diag(\vec({\mathbf P}^{(\Lambda)})) \vec(\hat{\mathbf A}) $$ 
such that 
\begin{align*} \| \vec({\mathbf A}) - \vec(\tilde{\mathbf A})\|_{2} &= 
\| \vec(\hat{\mathbf A}) - \vec(\hat{\tilde{\mathbf A}})\|_{2}
= \| \vec(\hat{\mathbf A}) - {\mathbf Q} \, \diag(\vec({\mathbf P}^{(\Lambda)})) \vec(\hat{\mathbf A})\|_{2} \\
&= \| ({\mathbf I} - {\mathbf Q} \, \diag(\vec({\mathbf P}^{(\Lambda)}))) \vec(\hat{\mathbf A})\|_{2}
\end{align*}
is minimized for arbitrary matrices $\hat{\mathbf A}$.   
Thus, 
$\| {\mathbf I} - {\mathbf Q} \, \diag(\vec({\mathbf P}^{(\Lambda)})) \|_{2}$ has to be minimized.
Denote the components of ${\mathbf Q}$ by $q_{\kk,\elll}$, $\kk, \elll \in \Lambda_{N,M}$, and the diagonal elements of $\diag(\vec({\mathbf P}^{(\Lambda)}))$ by $p^{(\Lambda)}_{\elll}$.
Separating  the diagonal and the non-diagonal entries, we obtain 
\begin{align}\nonumber  
\|{\mathbf I}_{MN} - {\mathbf Q} \, \diag(\vec({\mathbf P}^{(\Lambda)}))\|^2_{2} & = \sum_{\kk, \elll \in \Lambda_{N,M}} |\delta_{\kk,\elll} - q_{\kk,\elll}\,  {p}^{(\Lambda)}_{\elll}|^{2} 
+ \\
&=  \sum_{\elll \in \Lambda} |1 - q_{\elll,\elll}|^{2}+ (NM-|\Lambda|)  +  \sum_{\elll \in \Lambda} \sum_{\substack{\kk \in \Lambda_{N,M}\\ \kk \ne \elll}} |q_{\kk,\elll}|^{2} , \label{na}
\end{align}
where $(NM-|\Lambda|)$ is  the number on non-acquired Fourier components.
Obviously, (\ref{na})  is minimized for ${\mathbf Q}=  \diag(\vec({\mathbf P}^{(\Lambda)}))$, since then the first and the last sum vanish.
 The matrix ${\mathbf Q}=  \diag(\vec({\mathbf P}^{(\Lambda)})$ is also an optimal solution if the Frobenius norm in (\ref{na}) is replaced by the spectral norm. Thus, regarding these norms, there is no better nonadaptive solution than the reconstruction by zero refilling.

\subsection{Low-rank approximation}

For matrix completion problems, often a low-rank constraint has been successfully employed. 
We shortly explain, why a low-rank constraint is unfortunately not helpful to solve our problem of image recovery from structured incomplete Fourier data.

Let $\tilde{\mathbf A}$  denote the wanted  image reconstruction, where we assume that the given constraint ${\mathbf P}^{(\Lambda)} \circ \hat{\tilde{\mathbf A}} = {\mathbf P}^{(\Lambda)} \circ \hat{\mathbf A}$ is satisfied. 
Then, we observe that 
the solution $\tilde{\mathbf A}^{(z)}$ in (\ref{refill}),  obtained by zero refilling, already satisfies
$$ \textstyle \rank(\tilde{\mathbf A}^{(z)}) = \rank({\mathbf F}_{N}^{-1} ( {\mathbf P}^{(\Lambda)} \circ \hat{\mathbf A}) {\mathbf F}_{M}^{-1}) = \rank( {\mathbf P}^{(\Lambda)} \circ \hat{\mathbf A}) \le \frac{N}{r} $$
since at most $\frac{N}{r}$ rows of ${\mathbf P}^{(\Lambda)} \circ \hat{\mathbf A}$ are non-zero rows. Thus, if we request for $\tilde{\mathbf A}$ a low rank being larger than or equal to $\rank( {\mathbf P}^{(\Lambda)} \circ \hat{\mathbf A})$, we obtain the solution $\tilde{\mathbf A}^{(z)}$ in (\ref{refill}), while for requesting a smaller rank, the Fourier constraints can no longer be satisfied exactly. 
Therefore the low-rank constraint seems not to be well applicable.

\subsection{Locally adaptive interpolation in Fourier domain}
\label{ssec:inter}

Some Fourier interpolation methods in MRI, like GRAPPA \cite{grappa}  and SPIRiT \cite{spirit},  use the fully sampled low-pass area to learn interpolation weights to reconstruct the unacquired  data in Fourier domain.  In this section, we summarize these two methods, which are frequently used in parallel MRI.

 Note that in parallel MRI, GRAPPA and SPIRiT make use of different \textit{coils} that gather Fourier data (so-called $k$-space data in MRI) in parallel, something that is absent from our model. Therefore, our reconstruction problem can be regarded as  a special case of GRAPPA and SPIRiT with only one coil. 
The main idea of these methods is to determine suitable weights for a local interpolation scheme in the first step and then to apply these weights to reconstruct the missing Fourier data.  

We start with the idea of \cite{grappa}. 
Let $\mathcal{N} = \Lambda_{2p_1+1, 2p_2+1}$ be a small centered index set (window), i.e., ${\mathcal N} =\{-p_1, \ldots , p_1 \} \times \{-p_2 , \ldots , p_2\}$ with $p_{1}, p_{2} \in {\mathbb N}$ and with $|{\mathcal N}| = (2p_1+1)(2p_2+1)$ indices.
Using the  acquired Fourier data  $\hat{a}_{\nuu}$ of $\hat{\mathbf A}$, 
$\nuu \in \Lambda \times \Lambda_M$, a local interpolation scheme for unacquired data $\hat{a}_{\nuu}$, $\nuu \in (\Lambda_{N} \setminus \Lambda) \times \Lambda_M$, is taken in the form 
\begin{align}\label{interpol1}   \hat{\tilde{a}}_{\nuu} 
= \sum_{\mycom{\jj \in {\mathcal N}}{\nuu + \jj \in \Lambda \times \Lambda_M}} g_{\jj} \, \hat{a}_{\nuu+\jj}
= \sum_{\jj \in {\mathcal P}(\nuu)} g_{\jj} \, \hat{a}_{\nuu+\jj},
 \end{align}
where on the right-hand side only acquired values with indices in the window $\nuu + {\mathcal N}$ come into play.
The index set ${\mathcal P}(\nuu) := \{ \jj \in {\mathcal N}, \, 
{\mathbf j} + \nuu \in \Lambda \times \Lambda_M \} \subset {\mathcal N}$ depends on the location of $\nuu$.
Figure \ref{grappa-fig} illustrates different  index sets ${\mathcal P}(\nuu)$ (green) around the index $\nuu$ of an unacquired pixel value (red) for $p_1=2$, $p_2=3$.
\begin{figure}
\begin{center}
	\includegraphics[scale=0.25]{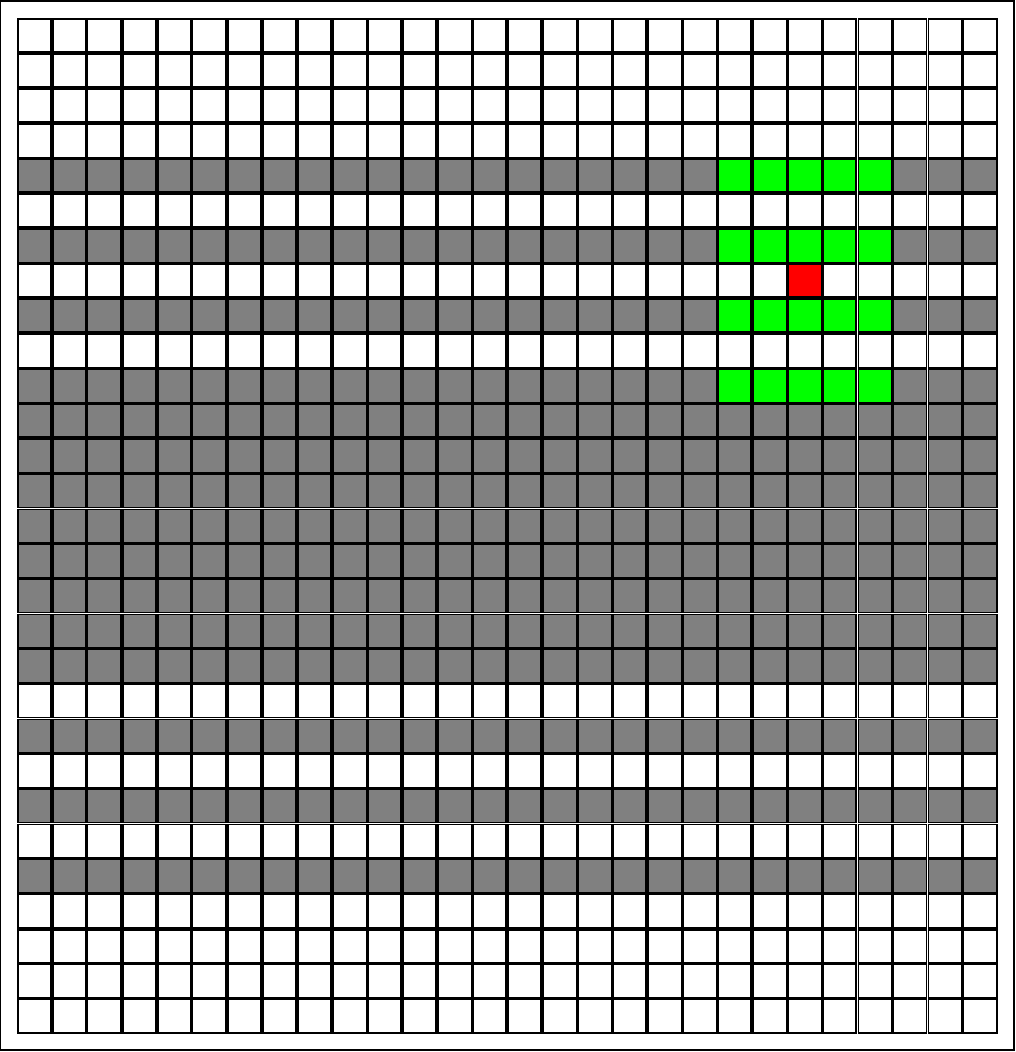}  \qquad 
	\includegraphics[scale=0.25]{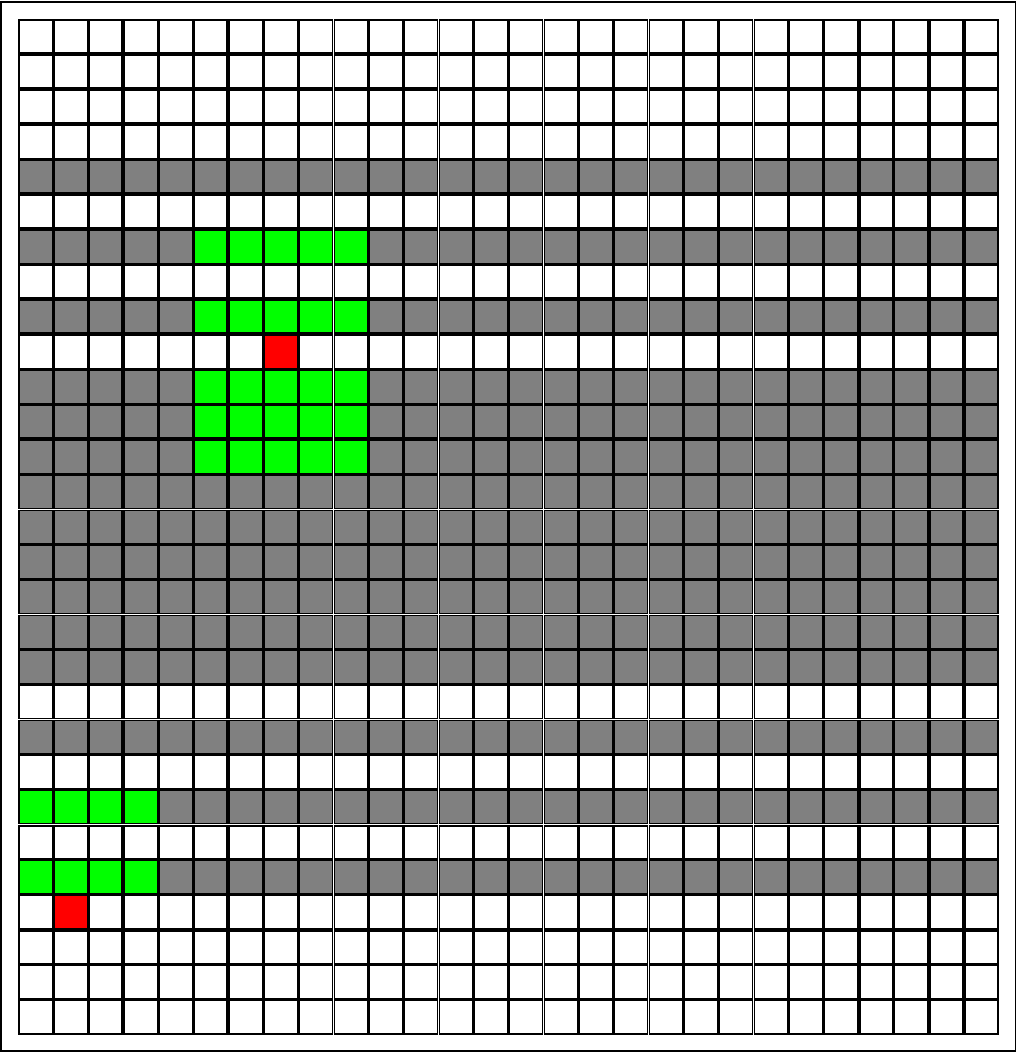} 
	\end{center}
\caption{
Example for different patterns ${\mathcal P}$  for GRAPPA interpolation using local interpolation in a $7 \times 5$ window. All gray- and green-valued pixel values are acquired. Green pixel values are employed to interpolate the red pixel value}
\label{grappa-fig}
\end{figure}

Depending on the location of the pixel values that one  wants to recover, 
a finite number of different patterns ${\mathcal P}= {\mathcal P}(\nuu)$ of acquired neighboring pixel values occurs. 
Next, for every occurring pattern ${\mathcal P} \subset {\mathcal N}$ the (complex) weights $(g_{\jj})_{ \jj \in {\mathcal P}}= (g_{\jj}^{({\mathcal P})})_{\jj \in {\mathcal P}} $ are computed from the given Fourier values in the low-pass area $\Lambda_{L,M} \subset \Lambda \times \Lambda_M$ by solving the least squares problem 
\begin{align}\label{grappa1}
 (g_{\jj}^{({\mathcal P})})_{\jj \in {\mathcal P}} := \argmin_{(g_{\jj})_{\jj \in {\mathcal P}} \in {\mathbb C}^{|{\mathcal P}|}}
  \sum_{\nuu + {\mathcal P} \in (\Lambda_L \times \Lambda_M)}  |\hat{a}_{\nuu} -  
\sum_{\jj \in {\mathcal P}} g_{\jj} \, \hat{a}_{\nuu + \jj} |^2.
\end{align}
Then, all missing pixel values $\hat{\mathbf a}_{\nuu}$ are recovered by (\ref{interpol1}).
Unacquired pixel values $\hat{\mathbf a}_{\nuu}$ with larger row index, whose ${\mathcal N}$ neighborhood does not contain acquired values, stay to be filled with zeros, since we have no neighbour information about these values.

\begin{remark}
In the original paper \cite{grappa}, it is proposed to use several interpolation schemes with small windows, which need not to be centered, and to apply an averaging procedure in the end.
\end{remark}

In \cite{spirit}, a different procedure is proposed  which can also applied to unstructured sampling. Again, it is assumed that the Fourier data $\hat{a}_{\nuu}$ can be approximated by a suitable linear combination of neighboring  Fourier data with indices in a small window $\nuu + {\mathcal N}$. This time, an interpolation scheme for all data,
regardless of being acquired or not, is derived and later taken as a constraint in a minimization problem to recover the missing data.
Using periodic boundary conditions, the weights are supposed to satisfy  the constraints
\begin{align}\label{interspirit}  
 \hat{{a}}_{\nuu}  &=  \textstyle \sum\limits_{\jj \in {\mathcal N} \setminus \{ {\mathbf 0} \} } g_{\jj} \, 
 \hat{a}_{\nuu + \jj} 
= \sum\limits_{\jj' \in \Lambda_{N,M} } g_{\jj' - \nuu} \, 
 \hat{a}_{\jj'}
 , \qquad  \nuu \in \Lambda_{M,N},
 \end{align}
where  $g_{\jj}=0$ for $\jj \not\in {\mathcal N} \setminus \{{\mathbf 0} \}$ and 
$\hat{a}_{\nuu + \jj} = \hat{a}_{\nuu + \jj \, \text{mod}\,  \Lambda_{N,M}}$, $g_{\jj' - \nuu}= g_{\jj' - \nuu \, \text{mod}\, \Lambda_{N,M}}$. 

 In a first step,  the weights $g_{\jj}$, $\jj \in {\mathcal N} \setminus \{ {\mathbf 0} \}$, are computed from the given data in the low-pass area  by solving the least squares problem in (\ref{grappa1}) for the  pattern ${\mathcal P}= {\mathcal N} \setminus \{ {\mathbf 0} \}$.
 Applying a vectorization to the constraint (\ref{interspirit}) it can be rewritten as
\begin{align}\label{eig} \vec (\hat{\mathbf A}) = {\mathbf G} \, \vec(\hat{\mathbf A}),
\end{align}
where ${\mathbf G}:=\big(g_{\jj - \nuu} \big)_{\nuu \in \Lambda_{N,M}, \, \jj \in \Lambda_{M,N}}$ is a (sparse) block Toeplitz matrix. 
To find an  approximation $\hat{\tilde{\mathbf A}}$ of $\hat{\mathbf A}$,   
one considers in \cite{spirit}  an optimization problem of the form
\begin{align}\label{spiritopt} \textstyle 
\min\limits_{\tilde{\mathbf A}}   \Big( \| 
{\mathbf P}^{\Lambda} \circ (\hat{\tilde{\mathbf A}} - \hat{\mathbf A}) \|_2^2 +  \lambda \| ({\mathbf G} - {\mathbf I}_{MN}) \vec(\hat{\tilde{\mathbf A}}) \|_2^2 \Big),
\end{align}
where the first term takes care of the approximation of the given Fourier values and the second ensures (\ref{eig}).
The first term has the vectorized form 
$ \|({\mathbf I}_M \otimes \diag({\mathbf p}^{(\Lambda)})) (\vec(\hat{\tilde{\mathbf A}}) - \vec(\hat{\mathbf A})) \|_2^2$.
The minimization problem (\ref{spiritopt}) can be directly solved to compute $\tilde{\mathbf A}$.
Calculating the gradient with respect to $\vec(\hat{\tilde{\mathbf A}})$, we obtain the large linear system
$$  \Big(({\mathbf I}_M \otimes \diag({\mathbf p}^{(\Lambda)})) + \lambda ({\mathbf G} - {\mathbf I}_{MN})^* ({\mathbf G} - {\mathbf I}_{MN}) \Big)\, 
\vec(\hat{\tilde{\mathbf A}}) =   2({\mathbf I}_M \otimes \diag({\mathbf p}^{(\Lambda)}))   \vec(\hat{\mathbf A}).
$$
This system can be iteratively solved using the fixed-point formulation 
\begin{align*}
\vec(\hat{\tilde{\mathbf A}})_{n+1} &= \textstyle 
\vec(\hat{\tilde{\mathbf A}})_{n} - \mu \Big(({\mathbf I}_M \otimes \diag({\mathbf p}^{(\Lambda)}))(\vec(\hat{\tilde{\mathbf A}})_n - \vec(\hat{\mathbf A})) \\
 & \textstyle \quad + 
\lambda ({\mathbf G} - {\mathbf I}_{MN})^* ({\mathbf G} - {\mathbf I}_{MN})\vec(\hat{\tilde{\mathbf A}})_n \Big),
\end{align*} 
which converges if $\mu>0$ is taken small enough.

Unfortunately, the two interpolation methods presented in the section usually do not provide satisfying reconstruction results in our setting compared to zero refilling and compared to the further recovering algorithms that we will consider in the next two sections.  This is not surprising, if we look back at our results in  Theorem \ref{thm:linear}
and Section \ref{non-adaptive}, since adaptivity only relies here to the data in the low-pass area and the 2D DFT data usually do not  possess small total variation  that would ensure  a good approximation  by local interpolation.

\section{Reconstruction by TV functional minimization}
\label{sec:TV}
In this section we want to adapt  a technique to our setting which is  often used in image denoising and image reconstruction, namely the application of a regularizer based on minimizing the total variation. In contrast to other approaches based on  TV regularization using randomly sampled measurements \cite{Candes06} or Fourier measurements on a polar grid \cite{Block07,Yang10,Knoll11}, our 2D-DFT measurements are acquired  on the the sampling pattern $\Lambda \times \Lambda_M$ in  (\ref{set}). We will adapt  the primal-dual algorithm of Chambolle and Pock \cite{CP10,CP16} to minimize the obtained functional, where in our case one of the fixed-point iterations is performed in Fourier domain.
We assume that the image ${\mathbf A}$ is real-valued, otherwise we reconstruct here the real and imaginary part of the ${\mathbf A}$ separately.

We start with  some notations.
Let 
\begin{align}\label{DN}
{\mathbf D}_{N}:= \left(\begin{array}{ccccc} 
-1 & 1 &  & & \\
 & -1 & 1 & & \\
  & & \ddots & \ddots &  \\
  & & & -1 & 1 \\
  & & & & 0 \end{array} \right) \in {\mathbb R}^{N \times N}
  \end{align}
  denote the difference matrix, then the discrete gradient of ${\mathbf A} \in {\mathbb R}^{N \times M}$ is defined as 
  $\nabla: {\mathbb R}^{N \times M} \to {\mathbb R}^{2N \times M}$, 
  \begin{align*} \textstyle  \nabla ({\mathbf A}):= \left( \begin{array}{c} {\mathbf D}_{N} {\mathbf A} \\ {\mathbf A} {\mathbf D}_{M}^{T} \end{array} \right).
  \end{align*}
 For  ${\mathbf A} = (a_{\kk})_{\kk \in \Lambda_{N,M}} \in {\mathbb R}^{N \times M}$
 and ${\mathbf B} = (b_{\kk})_{\kk \in \Lambda_{N,M}} \in {\mathbb R}^{N \times M}$ 
  let 
\begin{align*}
 \|{\mathbf A}\|_{2} := \textstyle  \sum\limits_{\kk \in \Lambda_{N,M}} |a_{\kk}|^2, 
\quad 
\langle {\mathbf A}, {\mathbf B} \rangle := \textstyle \sum\limits_{\kk \in \Lambda_{N,M}} {a_{\kk}} b_{\kk}. %
\end{align*}
Further, for  ${\mathbf X}= \left( \!\!\!\begin{array}{c} {\mathbf X}_{1} \\ {\mathbf X}_{2} \end{array}\!\!\! \right) \in {\mathbb R}^{2N \times M}$ with ${\mathbf X}_{1} = (x_{\kk,1})_{\kk \in \Lambda_{N,M}}$, ${\mathbf X}_{2} = (x_{\kk,2})_{\kk \in \Lambda_{N,M}}$ we define 
\begin{align} \label{inf}
 \|{\mathbf X}\|_{1} :=  \textstyle \sum\limits_{\kk \in \Lambda_{N,M}} \sqrt{x_{\kk,1}^{2} + x_{\kk,2}^{2}},  \qquad 
\|{\mathbf X}\|_{\infty} := \max\limits_{\kk \in \Lambda_{N,M}} \sqrt{x_{\kk,1}^{2} + x_{\kk,2}^{2}}. 
\end{align}

To find a good image reconstruction $\tilde{\mathbf A}$ from the incomplete data ${\mathbf P}^{(\Lambda)} \hat{\mathbf A}$ in (\ref{aufgabe})
our goal is to compute the minimizer of the functional 
\begin{align}\label{primal}
\textstyle \min\limits_{\tilde{\mathbf A} \in {\mathbb R}^{N \times M}} \Big( \frac{\lambda}{2} \|{\mathbf P}^{(\Lambda)} \circ (\hat{\tilde{\mathbf A}} - \hat{\mathbf A}) \|^{2}_{2} +  \| \nabla (\tilde{\mathbf A}) \|_{1} \Big).
\end{align}
The regularization term $ \| \nabla (\tilde{\mathbf A}) \|_{1}$ is called \textit{discrete total variation} of $\tilde{\mathbf A}$.
The regularization parameter $\lambda$ has to be taken large to ensure that $\|{\mathbf P}^{(\Lambda)} \circ (\hat{\tilde{\mathbf A}} - \hat{\mathbf A}) \|^2_{2}$ is very small. 

As in \cite{CP10}, this problem is first transferred into a saddle-point problem. For this purpose, we introduce the mappings $G:{\mathbb R}^{M \times N} \to {\mathbb R}$  and $H: {\mathbb R}^{2N \times M} \to {\mathbb R}$  given by 
 \begin{align*}
\textstyle 
 G(\tilde{\mathbf A}) := \frac{\lambda}{2} \|{\mathbf P}^{(\Lambda)} \circ (\hat{\tilde{\mathbf A}} - \hat{\mathbf A}) \|^2_{2}, \qquad
  H({\mathbf C} ) := \|{\mathbf C} \|_{1}, \qquad 
 \tilde{\mathbf A} \in {\mathbb R}^{N \times M},\quad  {\mathbf C} \in {\mathbb R}^{2N \times M},
\end{align*}
such that (\ref{primal}) can shortly be written $\min\limits_{\tilde{\mathbf A} \in {\mathbb R}^{N \times M}} (H(\nabla \tilde{\mathbf A}) + G(\tilde{\mathbf A}))$.
 With $H^{*} \!:  {\mathbb R}^{2N \times M} \to {\mathbb R}$ we denote the convex conjugate of the convex function $H$, 
and  obtain in our case
\begin{align}\label{F*}
H^{*}({\mathbf X}) &= \textstyle \max\limits_{{\mathbf C} \in {\mathbb R}^{2N \times M}} \Big(\langle {\mathbf X}, {\mathbf C} \rangle - \|{\mathbf C} \|_{1} \Big)
= \left\{ \begin{array}{ll} 0 & \text{if} \; \|{\mathbf X} \|_{\infty} \le 1, \\
 + \infty & \text{otherwise,} \end{array} \right. 
 \end{align}
with $\|{\mathbf X} \|_{\infty}$ as in (\ref{inf}), see also \cite{CP10}.  
The Fenchel-Moreau theorem, see e.g.\ \cite{Borwein}, implies  that $H^{**} = H$. Therefore, with ${\mathbf C}= \nabla (\tilde{\mathbf A})$
we also have 
\begin{align*} 
 \textstyle H({\mathbf C}) =H(\nabla(\tilde{\mathbf A})) = \|\nabla (\tilde{\mathbf A})\|_{1}= \max\limits_{{\mathbf X} \in {\mathbb R}^{2N \times M}} \Big( \langle \nabla (\tilde{\mathbf A}), {\mathbf X} \rangle - H^{*}({\mathbf X}) \Big),
\end{align*}
and (\ref{primal}) can equivalently be written as a saddle-point problem
\begin{align}\label{saddle}
\textstyle \max\limits_{{\mathbf X} \in {\mathbb R}^{2N \times M}} \min\limits_{\tilde{\mathbf A} \in {\mathbb R}^{N \times M}} \Big(  G(\tilde{\mathbf A}) + \langle \nabla(\tilde{\mathbf A}), {\mathbf X} \rangle   - H^{*}({\mathbf X}) \Big).
\end{align}
Taking the subgradients with respect to ${\mathbf X}$ and $ \tilde{\mathbf A}$, it follows that any solution $( {\mathbf X}, \tilde{\mathbf A})$ of (\ref{saddle}) necessarily satisfies 
${\mathbf 0} \in \Big(\nabla(\tilde{\mathbf A}) - \partial H^{*}({\mathbf X})\Big)$ and 
 ${\mathbf 0}  \in \Big(\nabla^{*} ({\mathbf X}) + \partial G(\tilde{\mathbf A})\Big)$, i.e.,  
\begin{align}\label{partial}
\textstyle  \nabla(\tilde{\mathbf A})  \in \partial H^{*}({\mathbf X}), \qquad - \nabla^{*} ({\mathbf X}) \in \partial G(\tilde{\mathbf A}).
\end{align}
Here, $\nabla^{*}$ is defined by $\langle \tilde{\mathbf A}, \nabla^{*}({\mathbf X}) \rangle = \langle \nabla(\tilde{\mathbf A}), {\mathbf X} \rangle $, i.e., $\nabla^{*}({\mathbf X}) = {\mathbf D}_{N} {\mathbf X}_{1} + {\mathbf X}_{2} {\mathbf D}_{M}^{T}$. 
The first condition in (\ref{partial}) yields by multiplying  with $\sigma >0$ and adding ${\mathbf X}$ at both sides
\begin{align*} \textstyle {\mathbf X} + \sigma \nabla(\tilde{\mathbf A})  \in {\mathbf X} + \sigma \partial H^{*}({\mathbf X}) = ({\mathbf I} + \sigma \partial H^{*}) ({\mathbf X}),  
\end{align*} 
 where ${\mathbf I}$ denotes the identity operator.
Applying a similar procedure to the second condition in (\ref{partial}) with $\tau >0$, we arrive at  the two fixed-point equations, 
\begin{align}\label{fix1}
 {\mathbf X}&= \textstyle ({\mathbf I} + \sigma \partial H^{*})^{-1} ({\mathbf X} + \sigma \nabla(\tilde{\mathbf A})), \\
 \label{fix2}
\tilde{\mathbf A}&= \textstyle ({\mathbf I} + \tau \partial {G})^{-1} (\tilde{\mathbf A} -\nabla^{*} ({\mathbf X})).
\end{align}
The primal-dual algorithm introduced in \cite{CP10}, which we adapt here to our setting,  is based on alternating application of the two corresponding fixed-point iterations for the primal variable $\tilde{\mathbf A}$ and the dual variable ${\mathbf X}$.
The two resolvent operators (or proximity operators) $({\mathbf I} + \sigma \partial H^{*})^{-1}$ and $({\mathbf I} + \tau \partial {G})^{-1}$ occurring in (\ref{fix1})-(\ref{fix2}) are determined  by
\begin{align*}
({\mathbf I} + \sigma \partial H^{*})^{-1}({\mathbf X}) &:= \textstyle \argmin\limits_{{\mathbf Y} \in {\mathbb R}^{2N \times M}}
\Big( \frac{1}{2\sigma} \|{\mathbf X} -{\mathbf Y}\|^2_{2} + H^{*}({\mathbf Y})\Big), \\
({\mathbf I} + \tau \partial {G})^{-1}(\tilde{\mathbf A}) &:= \textstyle \argmin\limits_{{\mathbf Z} \in {\mathbb R}^{N \times M}}
\Big( \frac{1}{2\tau} \|\tilde{\mathbf A} -{\mathbf Z}\|^2_{2} + G({\mathbf Z}) \Big),
\end{align*}
 and can be directly computed. 
Indeed, for $\tilde{\mathbf Y}= \argmin\limits_{{\mathbf Y} \in {\mathbb R}^{2N \times M}}
\Big( \frac{1}{2\sigma} \|{\mathbf X} -{\mathbf Y}\|_{F}^{2} + H^{*}({\mathbf Y})\Big)$ it follows 
$$ \textstyle {\mathbf 0} \in \partial \Big( \frac{1}{2\sigma} \|{\mathbf X} -\tilde{\mathbf Y}\|_{F}^{2} + H^{*}(\tilde{\mathbf Y})\Big)
= \frac{1}{\sigma} (\tilde{\mathbf Y}- {\mathbf X}) + \partial H^{*}(\tilde{\mathbf Y}) ,$$
i.e., ${\mathbf X} \in \tilde{\mathbf Y} + \sigma \partial H^{*}(\tilde{\mathbf Y})= ({\mathbf I} + \sigma \partial H^{*}) (\tilde{\mathbf Y})$, and thus $\tilde{\mathbf Y} = ({\mathbf I} + \sigma \partial H^{*})^{-1}({\mathbf X})$.

In our special case, we find from the definition of $H^{*}$ in (\ref{F*}) that
\begin{align*} 
({\mathbf I} + \sigma \partial H^{*})^{-1}({\mathbf X}) &= \textstyle \argmin\limits_{\mycom{{\mathbf Y} \in {\mathbb R}^{2N \times M}}{\|{\mathbf Y} \|_{\infty} \le 1}} \Big( \frac{1}{2\sigma} \|{\mathbf X} -{\mathbf Y}\|_{F}^{2} \Big), 
\end{align*}
that is, $({\mathbf I} + \sigma \partial H^{*})^{-1}({\mathbf X})$ is the projection of ${\mathbf X}$ onto the unit ball in ${\mathbb R}^{2N \times M}$ with respect to the Frobenius norm, i.e., 
$ ({\mathbf I} + \sigma \partial H^{*})^{-1}({\mathbf X}) = \tilde{\mathbf Y} = \left( \begin{array}{c} \tilde{\mathbf Y}_{1} \\ \tilde{\mathbf Y}_{2} \end{array} \right)$ with components 
$$ \textstyle \tilde{y}_{\kk,1} = \frac{x_{\kk,1}}{\max\{ 1, \sqrt{x_{\kk,1}^{2} + x_{\kk,2}^{2}} \}}, \qquad 
 \tilde{y}_{\kk,2} = \frac{x_{\kk,2}}{\max\{ 1, \sqrt{x_{\kk,1}^{2} + x_{\kk,2}^{2}} \}}, \quad \kk \in \Lambda_{N,M}. $$
To compute $({\mathbf I} + \tau \partial {G})^{-1}(\tilde{\mathbf A})$, we recall that the multiplication with an orthonormal matrix leaves the Frobenius norm invariant, such that $\|\tilde{\mathbf A} -{\mathbf Z}\|_{F}^{2}
= \|{\mathbf F}_{N}(\tilde{\mathbf A} -{\mathbf Z}) {\mathbf F}_{M} \|_{F}^{2} = \|\hat{\tilde{\mathbf A}} -\hat{\mathbf Z}\|_{F}^{2}$. Since $\tilde{\mathbf A}$ and ${\mathbf A}$ are both real matrices, we obtain
\begin{align*}
({\mathbf I} + \tau \partial {G})^{-1}(\tilde{\mathbf A}) &=
\textstyle \argmin\limits_{{\mathbf Z} \in {\mathbb R}^{N \times M}}
\Big( \frac{1}{2\tau} \|\tilde{\mathbf A} -{\mathbf Z}\|_{F}^{2} + \frac{\lambda}{2} \|{\mathbf P}^{(\Lambda)} \circ (\hat{{\mathbf Z}} - \hat{\mathbf A}) \|_{F}^{2} \Big) 
\end{align*}
where 
\begin{align*} \textstyle 
\min\limits_{{\mathbf Z} \in {\mathbb R}^{N \times M}}
\!\!\Big(\! \frac{1}{2\tau} \|\tilde{\mathbf A} -{\mathbf Z}\|_{F}^{2}\! + \!\frac{\lambda}{2}\! \|{\mathbf P}^{(\Lambda)}\! \circ \!(\hat{{\mathbf Z}} - \hat{\mathbf A}) \|_{F}^{2}\! \Big)\!\!
 = \!\! \min\limits_{\hat{\mathbf Z} \in {\mathbb C}^{N \times M}}\!\!
\Big(\! \frac{1}{2\tau} \|\hat{\tilde{\mathbf A}} -\hat{\mathbf Z}\|_{F}^{2} + \frac{\lambda}{2} \|{\mathbf P}^{(\Lambda)}\! \circ \!(\hat{{\mathbf Z}} - \hat{\mathbf A}) \|_{F}^{2} \!\Big).
\end{align*}
The solution $\hat{\mathbf Z}$ of the minimization problem in Fourier domain satisfies the necessary condition
$(\hat{\mathbf Z}- \hat{\tilde{\mathbf A}}) + \tau \lambda {\mathbf P}^{(\Lambda)} \circ (\hat{\mathbf Z}-\hat{\mathbf A}) ={\mathbf 0}$, i.e., 
\begin{align}\label{FB}
\hat{\mathbf Z} = (\hat{\tilde{\mathbf A}}  + \tau \lambda {\mathbf P}^{(\Lambda)} \circ \hat{\mathbf A}) / ({\mathbf E} + \tau\lambda {\mathbf P}^{(\Lambda)}),
\end{align}
where $/$ denotes the componentwise division and ${\mathbf E}$ is the $(N \times M)$-matrix of ones. Thus, we finally obtain 
\begin{align*}
({\mathbf I} + \tau \partial {G})^{-1}(\tilde{\mathbf A}) &= {\mathbf F}_{N}^{-1} \hat{\mathbf Z} {\mathbf F}_{M}^{-1} = {\mathbf F}_{N}^{-1}((\hat{\tilde{\mathbf A}}  + \tau \lambda {\mathbf P}^{(\Lambda)} \circ \hat{\mathbf A}) / ({\mathbf E} + \tau\lambda {\mathbf P}^{(\Lambda)})) {\mathbf F}_{M}^{-1}.
\end{align*}
The  primal-dual algorithm  \cite{CP10} can for our setting be summarized as follows, where here  one fixed-point equation is solved in the Fourier domain.

{\small
\begin{algorithm} (Reconstruction from incomplete Fourier data using TV minimization)
\\
\textbf{Input:} incomplete Fourier data ${\mathbf P}^{(\Lambda)} \circ \hat{\mathbf A}$ of ${\mathbf A} \in {\mathbb R}^{N \times M}$ with 
$N=2n$, $M=2m$,  \\
\phantom{\textbf{Input:}} $N$ multiple of $8$, $L=2\ell+1 <N$,\\
\phantom{\textbf{Input:}} $N_{I}$ number of iterations,\\
\phantom{\textbf{Input:}} parameters $\tau>0$, $\sigma >0$, $\theta \in [0,1]$, $\lambda \gg 0$. 

\noindent
\textbf{Initialization:} ${\mathbf A}^{(0)} := {\mathbf F}_{N}^{-1}({\mathbf P}^{(\Lambda)} \circ \hat{\mathbf A}) {\mathbf  F}_{M}$, ${\mathbf X}^{(0)} = \left(\! \begin{array}{c} {\mathbf X}_{1}^{(0)} \\ {\mathbf X}_{2}^{(0)} \end{array}\!\right)
:= \nabla({\mathbf A}^{(0)}) = \left(\! \begin{array}{c} {\mathbf D}_{N} {\mathbf A}^{(0)} \\
{\mathbf A}^{(0)} {\mathbf D}_{M} \end{array} \!\right)$.\\

\noindent
For $j=0:N_{I}-1$ do 
\begin{enumerate}
\item Compute $\nabla({\mathbf A}^{(j)}) = \left( \!\!\begin{array}{c} {\mathbf D}_{N} {\mathbf A}^{(j)} \\
{\mathbf A}^{(j)} {\mathbf D}_{M} \end{array} \!\! \right)$ and
apply one fixed-point iteration step to solve $(\ref{fix1})$: 
Compute ${\mathbf X}^{(j+1)} \in {\mathbb C}^{2N \times M}$, 
\begin{align*} \textstyle {\mathbf X}^{(j+1)} := ({\mathbf X}^{(j)} + (\sigma \nabla({\mathbf A}^{(j)})) / \max\{ 1, \|{\mathbf X}^{(j)} + \sigma \nabla({\mathbf A}^{(j)}) \|_{\infty} \},
\end{align*}
where $/$ denotes the pointwise division and $ \|{\mathbf X}^{(j)} + \nabla({\mathbf A}^{(j)}) \|_{\infty}$  is defined according to $(\ref{inf})$.
\item Apply one fixed-point iteration step to solve $(\ref{fix2})$:\\
Write  ${\mathbf X}^{(j+1)} = \left( \begin{array}{c} {\mathbf X}_{1}^{(j+1)} \\ {\mathbf X}_{2}^{(j+1)} \end{array} \right)$ with 
${\mathbf X}_{1}^{(j+1)},\, {\mathbf X}_{2}^{(j+1)} \in {\mathbb C}^{N \times M}$
and compute
$$ \hat{\tilde{\mathbf A}}^{(j+1)} := {\mathbf F}_{N} ({\mathbf A}^{(j)}- \nabla^{*} {\mathbf X}^{(j+1)}) {\mathbf F}_{M} = {\mathbf F}_{N} ({\mathbf A}^{(j)}- ({\mathbf D}_{N}{\mathbf X}_{1}^{(j+1)}+ {\mathbf X}_{2}^{(j+1)} {\mathbf D}_{M}^{T})) {\mathbf F}_{M}.$$
Compute $\hat{\mathbf A}^{(j+1)} := (\hat{\tilde{\mathbf A}}^{(j+1)}  + \tau \lambda {\mathbf P}^{(\Lambda)} \circ \hat{\mathbf A}) / ({\mathbf E} + \tau\lambda {\mathbf P}^{(\Lambda)})$
with notations as in $(\ref{FB})$.\\
Compute ${\mathbf A}^{(j+1)} := {\mathbf F}_{N}^{-1} \hat{\mathbf A}^{(j +1)} {\mathbf F}_{M}^{-1}$.
\item Update ${\mathbf A}^{(j+1)} := {\mathbf A}^{(j+1)} + \theta ({\mathbf A}^{(j+1)} - {\mathbf A}^{(j)})$.
\end{enumerate}
end(for) \\
\textbf{Output:} ${\mathbf A}^{(PD)} ={\mathbf A}^{(N_{I})}$.
\label{algTV}
\end{algorithm}
}
As we will see in Section \ref{sec:num}, this algorithm usually provides very good reconstruction results 
which significantly outperform the reconstruction by zero refilling and the  interpolation algorithms from Section \ref{sec2} for reduction rates $r>2$.
Moreover, it is not restricted to the special sampling scheme, which we focussed on in this paper.
The convergence of the iteration is ensured for parameters $8 \tau \sigma <1$ as shown in \cite{C04,CP10}. 
In our implementation  we always use $\theta=1$ and $\sigma = 0.01 + \frac{1}{8\tau}$ as suggested by Gilles \cite{Gilles}.

\section{Hybrid method}
\label{sec:hybrid}

While Algorithm \ref{algTV} usually provides a very good performance,  we want to improve the reconstruction result further by incorporating our knowledge on the structure of the given set of acquired Fourier data. 
Algorithm \ref{algTV} tends to provide reconstruction results  containing staircasing artifacts and with a total variation which is much smaller than that of the original image. For example, for the normalized ''cameraman'' image, see Figure \ref{figure2} and Table \ref{tab1}, the total variation, i.e., the sum of all absolute values of $\nabla({\mathbf A})$, is $9060.3$, 
while for the resulting image of Algorithm \ref{algTV}  in Figure \ref{figure2} (bottom, middle) the total variation 
 is  $5597.3$ for $\lambda=100$. For larger $\lambda$ the total variation increases, we get $6687.9$ for $\lambda=1000$.  Further, the TV reconstruction tends to have an incorrect ''distribution'' of  corresponding index values $\tilde{a}_{k_1,k_2}$ and $\tilde{a}_{k_1-n,k_2}$ at some places in  the upper and the lower half of the image. This effect is already strongly reduced compared to the zero refilling reconstruction,  see e.g.\ Figure \ref{figure4},  but it increases for larger $\lambda$. 
We recall that this issue is due to the sampling scheme that we have at hand.

Our hybrid algorithm is an iterative procedure that locally enlarges the total variation of the image.
For this purpose, we start with a smooth approximation $\tilde{\mathbf A}^{(0)}$, which is either obtained by Algorithm \ref{algTV} with a regularization parameter $\lambda$, which is not too large, or by applying a linear smoothing procedure to the obtained approximation  $\tilde{\mathbf A}^{(PD)}$ in a first step.
Then this initial image $\tilde{\mathbf A}^{(0)}$ 
 already provides some knowledge about important local total variations (as edges) of ${\mathbf A}$, but  does hardly contain undesirable artifacts caused by badly estimated sums $a_{k_1,k_2} + a_{k_1-n,k_2}$ in the upper and the lower half of the image (as it happens e.g.\ in Figure \ref{figure2} (right) for zero refilling).
 The decision, where the total variation of the image approximation should be enlarged, will be taken by 
comparing a median local total variation (MTV) for every pixel value  in the upper half of the image $\tilde{\mathbf A}^{(0)}$ with the MTV  of the corresponding pixel value in the lower half of the image.

To improve the approximation $\tilde{\mathbf A}^{(j)}$ in the $j$-th iteration step, 
we consider the difference image ${\mathbf R}^{(j)}$ given by
$$\hat{\mathbf R}^{(j)} := {\mathbf P}^{(\Lambda)} \circ (\hat{\mathbf A} - \hat{\tilde{\mathbf A}}^{(j)}). $$
Then, obviously, $\tilde{\mathbf A}^{(j)}+ {\mathbf R}^{(j)}$ satisfies ${\mathbf P}^{(\Lambda)} \circ (\hat{\tilde{\mathbf A}}^{(j)}+ \hat{\mathbf R}^{(j)} )= {\mathbf P}^{(\Lambda)} \circ \hat{\mathbf A}$. 
To update the image ${\tilde{\mathbf A}}^{(j)}$, we proceed as follows. 
If the MTV of  $\tilde{a}_{k_1,k_2}^{(0)}$  in the upper half of $\tilde{\mathbf A}^{(0)}$ is (significantly) 
larger than the MTV of  $\tilde{a}_{k_1-n,k_2}^{(0)}$ in the lower half, then 
we add the component
$r_{k_1,k_2}^{(j)}$ (or even  an amplification $\mu r_{k_1,k_2}^{(j)}$ with $\mu >1$) of  ${\mathbf R}^{(j)}$ to $a_{k_1,k_2}^{(j)}$
while leaving $\tilde{a}_{k_1-n,k_2}^{(j)}$ (almost) untouched, otherwise we add $\mu r_{k_1-n,k_2}^{(j)}$ to $\tilde{a}_{k_1-n,k_2}^{(j)}$ and (almost) do not change $\tilde{a}_{k_1,k_2}^{(j)}$. If the  MTV  for $\tilde{a}_{k_1,k_2}^{(0)}$ and $\tilde{a}_{k_1-n,k_2}^{(0)}$ is almost of the same size, we add the corresponding (weighted) components of ${\mathbf R}^{(j)}$ at both positions.
The rationale behind this procedure is the following.
If the local total variation in a neighborhood of  a pixel value almost vanishes, then this indicates that the  image is locally smooth, i.e., the local total variation should not be enlarged, and the corresponding pixel value is kept, 
whereas if the local total variation in a neighborhood of a pixel value is large, then this indicates an important feature and the local total variation should be  enlarged. 
We will show that this iteration leads to an image reconstruction $\tilde{\mathbf A}$ that satisfies the Fourier data constraints (\ref{aufgabe}).

To compute the MTV, we apply the following formulas. First we compute the local total variations of $\tilde{\mathbf A}^{(0)}$ at all pixels $(k_1,k_2)$, 
\begin{align}\label{lTV}
\textstyle  \hspace*{-3mm} \text{TV}^{(\textrm{loc})}(k_1,k_2) :=\!\!\! \sum\limits_{j_2=-1}^{1} \!|\tilde{a}_{k_1,k_2}^{(0)}- \tilde{a}_{k_1,k_2-j_2}^{(0)}| + \!\!\sum\limits_{j_1=-1}^2 \sum\limits_{j_2=-1}^{1} \!|\tilde{a}_{k_1-j_1+1,k_2-j_2}^{(0)}- \tilde{a}_{k_1-j_1,k_2-j_2}^{(0)}|,
\end{align}
where for boundary pixels only existing neighbor values are taken.
Then we compute the median of the $\text{TV}^{(\textrm{loc})}$ values in a fixed window $[-\gamma_1,\gamma_1] \times [-\gamma_2,\gamma_2]$ around $(k_1,k_2)$, 
\begin{align}\label{MTV} \textrm{MTV}(k_1,k_2) := \text{median}( (\text{TV}^{(\textrm{loc})}(k_1+j_1, k_2+j_2)_{j_1=-\gamma_1,j_2=-\gamma_2}^{\gamma_1,\gamma_2})
\end{align}
for $k_1 = -n \ldots , n-1$, $k_2=-m, \ldots , m-1$, where for boundary pixels only the remaining values of the window $[-\gamma_1, \gamma_1] \times [-\gamma_2, \gamma_2]$ are involved.
At each iteration step, we can either always take the same local TV$^{(\textrm{loc})}$ and MTV values obtained from the initial image $\tilde{\mathbf A}^{(0)}$, or update these values using $\tilde{\mathbf A}^{(j)}$. 
The image update $\tilde{{\mathbf A}}^{(j+1)}$ is then derived by 
$$ \tilde{{\mathbf A}}^{(j+1)} = \tilde{{\mathbf A}}^{(j)} + \mu ({\mathbf W} \circ {\mathbf R}^{(j)})$$
with $\mu \in [1, 2)$ and a weight matrix ${\mathbf W}= (\w_{k_1,k_2})_{k_1=-n,k_2=-m}^{n-1,m-1} \in {\mathbb R}^{N \times M}$ which is determined  according to the MTV values at every pixel, 
\begin{align} \label{ww} \w_{k_1,k_2} := \left\{ \begin{array}{ll}
1-\epsilon &  |MTV(k_1,k_2)| > 1.5 |MTV(k_1 \pm n, k_2)|, \\
\epsilon &   |MTV(k_1\pm n,k_2)| >  1.5 |MTV(k_1, k_2)|, \\
\frac{MTV(k_1,k_2)}{MTV(k_1,k_2) + MTV(k_1 \pm n,k_2)} & \text{otherwise}, 
\end{array} \right. 
\end{align}
where $\pm$ means that we take $+$ for $k_1 <0$ and $-$ for $k_1 \ge 0$.
The parameter $\epsilon>0$ is taken to be small, in the numerical experiments we have used $\epsilon= 0.05$ or $\epsilon=0.1$. 
This procedure is repeated  until the remainder ${\mathbf R}^{(j)}$ is close to the zero matrix.
The algorithm is summarized as follows.

{\small
\begin{algorithm} (Image reconstruction  improvement from incomplete Fourier data) \\
\textbf{Input:} incomplete Fourier data ${\mathbf P}^{(\Lambda)} \circ \hat{\mathbf A}$ of ${\mathbf A} \in {\mathbb R}^{N \times M}$ with 
$N=2n$, $M=2m$,  \\
\phantom{\textbf{Input:}}  $\tilde{\mathbf A}^{(PD)}$ reconstructed image of 
Algorithm $\ref{algTV}$. \\
\phantom{\textbf{Input:}} $N$ multiple of $8$, $L=2\ell+1 <N$,\\
\phantom{\textbf{Input:}} $N_{I}$ number of iterations for reconstruction,\\
\phantom{\textbf{Input:}} $N_s$ number of iterations for linear smoothing,\\
\phantom{\textbf{Input:}} $\mu \in [1,2)$ parameter for frequency reconstruction, \\
\phantom{\textbf{Input:}} $(\gamma_1,\gamma_2)$ local window size for local TV computation,\\
\phantom{\textbf{Input:}} $\epsilon>0$ (e.g.\ $0.05 \le \epsilon \le 0.1$).
\begin{enumerate}
\item (Optional) Apply a smoothing filter to  $\tilde{\mathbf A}^{(PD)} = (\tilde{a}_{k_1,k_2}^{(PD,0)})_{k_1=-n,k_2=-m}^{n-1,m-1}.$ \\ 
For $s=0:N_s-1$ \\
\null \quad For  $k_2=-m:m-1$\\
$$ \tilde{a}_{k_1,k_2}^{(PD,s+1)} := \left\{ \begin{array}{ll}
\frac{1}{4} (\tilde{a}_{k_1-1,k_2}^{(PD,s)} + 2 \tilde{a}_{k_1,k_2}^{(PD,s)} + \tilde{a}_{k_1+1,k_2}^{(PD,s)})  & -n+1\le k_1 \le n-2,\\[1ex]
\frac{1}{4} (3 \tilde{a}_{k_1,k_2}^{(PD,s)} + \tilde{a}_{k_1+1,k_2}^{(PD,s)}) & k_1=-n ,\\[1ex]
\frac{1}{4} (\tilde{a}_{k_1-1,k_2}^{(PD,s)} + 3 \tilde{a}_{k_1,k_2}^{(PD,s)} ) & k_1=n-1. \end{array} \right.
$$
Set $\tilde{\mathbf A}^{(0)}:= \tilde{\mathbf A}^{(PD,N_s)}= (\tilde{a}_{k_1,k_2}^{(0)})_{k_1=-n,k_2=-m}^{n-1,m-1}$. 
\item Compute the local total variation $\mathrm{TV}^{(\mathrm{loc})}(k_1,k_2)$  of $\tilde{\mathbf A}^{(0)} = (\tilde{a}_{k_1,k_2}^{(0)})_{k_1=-n,k_2=-m}^{n-1,m-1}$ according to $(\ref{lTV})$ for $k_{1}=-n, \ldots , n-1$,  $k_{2}=-m, \ldots , m-1$. \\
Compute the median local total variation $\mathrm{MTV}(k_1,k_2)$ in $(\ref{MTV})$
for $k_1 = -n, \ldots , n-1$, $k_2=-m, \ldots , m-1$.
\item Compute the matrix of weights ${\mathbf W}:=(\w_{k_1,k_2})_{k_1=-n, k_2=-m}^{n-1,m-1}$ 
as given in $(\ref{ww})$.
\item For $j=0:N_{I}-1$ do 
\begin{enumerate}
\item Compute $\hat{\tilde{{\mathbf A}}}^{(j)} := {\mathbf F}_N {\tilde{{\mathbf A}}}^{(j)} {\mathbf F}_M$ and 
${\mathbf R}^{(j)} := {\mathbf F}_N^{-1} ({\mathbf P}^{(\Lambda)} \circ (\hat{\mathbf A}- \hat{\tilde{{\mathbf A}}}^{(j)})) {\mathbf F}_M^{-1} $.
\item Compute the update
$$ \tilde{\mathbf A}^{(j+1)}  := \tilde{\mathbf A}^{(j)} + \mu \, {\mathbf W} \circ {\mathbf R}^{(j)}. $$
\end{enumerate}
end(do)
\end{enumerate}

\textbf{Output:} Image reconstruction $\tilde{\mathbf A}^{(H)} = \tilde{\mathbf A}^{(N_{I})}$.
\label{alghyp}
\end{algorithm}
}

\noindent
Instead of fixing the number of iterations $N_I$ in Step 4 of Algorithm \ref{alghyp}, we can also apply a stopping criteria
based on the  norm of the remainder ${\mathbf R}^{(j)}$.
Next, we show that Algorithm \ref{alghyp} always converges to an image satisfying the constraint ${\mathbf P}^{(\Lambda)}
\circ \hat{\mathbf A} = \lim\limits_{N_I \to \infty}  ({\mathbf P}^{(\Lambda)} \circ \hat{\tilde{\mathbf A}}^{(N_{I})})$.

\begin{theorem}\label{theo2}
For given Fourier data ${\mathbf P}^{(\Lambda)} \circ \hat{\mathbf A}$ with ${\mathbf P}^{(\Lambda)}$ 
in $(\ref{PL})$,
Algorithm $\ref{alghyp}$ converges for $\mu \in [1,2)$ to an image 
$\tilde{\mathbf A}= \lim\limits_{N_I \to \infty}  \tilde{\mathbf A}^{(N_{I})}$
 satisfying ${\mathbf P}^{(\Lambda)} \circ \hat{\mathbf A} = {\mathbf P}^{(\Lambda)} \circ \hat{\tilde{\mathbf A}}$.
\end{theorem}

\begin{proof}
It is sufficient to show that $\lim\limits_{j  \to \infty} {\mathbf R}^{(j)} = {\mathbf 0}$ for ${\mathbf R}^{(j)} = {\mathbf F}_N^{-1} ({\mathbf P}^{(\Lambda)} \circ (\hat{\mathbf A}- \hat{\tilde{{\mathbf A}}}^{(j)})) {\mathbf F}_M^{-1}$. 
Step 4 of Algorithm \ref{alghyp} yields  with $\hat{\mathbf R}^{(0)} = {\mathbf F}_N {\mathbf R}^{(0)} {\mathbf F}_M
= {\mathbf P}^{(\Lambda)} \circ (\hat{\mathbf A} - \hat{\tilde{\mathbf A}}^{(0)}) $  the recursion formula
\begin{align*}
\hat{\mathbf R}^{(j+1)} &= {\mathbf P}^{(\Lambda)} \circ (\hat{\mathbf A} - \hat{\tilde{\mathbf A}}^{(j+1)}) 
= {\mathbf P}^{(\Lambda)} \circ (\hat{\mathbf A} - ( \hat{\tilde{\mathbf A}}^{(j)} + \mu \, {\mathbf F}_N ( {\mathbf W}\circ {\mathbf R}^{(j)}) {\mathbf F}_M)) \\
& =  \hat{\mathbf R}^{(j)} - \mu  {\mathbf P}^{(\Lambda)} \circ ({\mathbf F}_N ( {\mathbf W}\circ {\mathbf R}^{(j)}) {\mathbf F}_M) =  {\mathbf P}^{(\Lambda)}  \circ (\hat{\mathbf R}^{(j)} - \mu ({\mathbf F}_N ( {\mathbf W}\circ {\mathbf R}^{(j)}) {\mathbf F}_M)) \\
&= {\mathbf P}^{(\Lambda)}  \circ \big({\mathbf F}_N ( {\mathbf R}^{(j)} -  \mu ( {\mathbf W}\circ {\mathbf R}^{(j)})){\mathbf F}_M \big),
\end{align*}
since $\hat{\mathbf R}^{(j)}$ has  by definition vanishing components for all 
$\nuu =(\nu_1,\nu_2) \in (\Lambda_{N} \setminus \Lambda) \times \Lambda_M$.
Taking the inverse DFT and applying vectorization 
leads to
\begin{align}\label{rj} \vec ({\mathbf R}^{(j+1)}) \!= \! ({\mathbf F}_M \otimes {\mathbf F}_N)^{-1} \diag(\vec({\mathbf P}^{(\Lambda)}))  ({\mathbf F}_M \otimes {\mathbf F}_N) \big({\mathbf I}_{MN} - \mu   \diag (\vec({\mathbf W})) \big) \vec ({\mathbf R}^{(j)}),
\end{align}
where we have used that $\vec({\mathbf F}_{N} {\mathbf R}^{(j)} {\mathbf F}_{M}) = ({\mathbf F}_{M} \otimes {\mathbf F}_{N}) \vec({\mathbf R}^{(j)})$, see e.g.\ \cite{PPST23}, Section 3.4.
We observe that  all matrix factors in (\ref{rj}) have a spectral norm smaller than or equal to one. In particular, $({\mathbf F}_M \otimes {\mathbf F}_N)$ is orthonormal, $\diag( \vec({\mathbf P}^{(\Lambda)}) )$ only contains zeros or ones in the diagonal, and $\big({\mathbf I}_{MN} - \mu \,  \diag (\vec({\mathbf W})) \big)$ contains diagonal entries $1- \mu \w_{k_1,k_2} \in (-1+2\epsilon,1-\epsilon]$, since  $\mu \in [1, 2)$, $\w_{k_1,k_2} \in [\epsilon, 1-\epsilon]$. 
Thus, the spectral radius of  $\big({\mathbf I}_{MN} - \mu \,  \diag (\vec({\mathbf W})) \big)$ is at most $1-\epsilon$, and we can directly conclude that 
$$  \|\vec ({\mathbf R}^{(j+1)})\|_2 \le (1-\epsilon)  \|\vec ({\mathbf R}^{(j)})\|_2, $$
such that convergence is ensured for $j \to \infty$.
 \end{proof}
\medskip

\section{Numerical Results}
\label{sec:num}

In this section we compare the described algorithms for reconstruction from structured incomplete Fourier data with emphasis to the sampling pattern given in (\ref{set}). 
We particularly consider different images of size $512 \times 512$ and compare the PSNR (peak signal to noise ratio) for the   reconstruction methods zero refilling in Section \ref{ssec:zero}, local interpolation using GRAPPA in Section \ref{ssec:inter}, TV functional minimization in Section \ref{sec:TV}, the hybrid method  in Section \ref{sec:hybrid}, and 
the low-pass reconstruction, obtained if for a reduction rate $r$, we simply take all middle rows with indices $-\lfloor \frac{N}{2r} \rfloor, \ldots , \lfloor \frac{N}{2r} \rfloor$. The reconstruction method SPIRiT in Section \ref{ssec:inter} requires a higher computational effort and does not give better results than GRAPPA in our setting.
All used images are ``Standard''  test images and have been taken from the open source platform \texttt{https://www.imageprocessingplace.com/root\_ \break files\_V3/image\_databases.htm}.
The MATLAB software to reproduce the results in this section can be found at 
\texttt{https://na.math.uni-goettingen.de} under \texttt{software}.

We will study reduction rates $r=2,4,6,8$, taking in the corresponding experiments $256$, $128$, $85$ or $64$
rows according to the scheme fixed in (\ref{set}).
Recall that for  the considered real images this corresponds to almost double reduction rates  up to $16$, since we have $\hat{a}_{\nuu} = \overline{\hat{a}}_{-\nuu}$.
For the width of the low-pass area we consider different values  $L < 512/r$.
For all image reconstructions, the red colored PSNR values in Tables \ref{tab1}--\ref{tabneu} indicate the best reconstruction result for the considered reduction rate.

\begin{table}[htbp]
\scriptsize
\caption{Comparison of the reconstruction performance for incomplete Fourier data for the $512 \times 512$ 
''cameraman'' image (PSNR values)}
\begin{center}
\begin{tabular}{llccccc}
\hline
reduction rate& low-pass width  &  zero refilling &  only low pass  & {GRAPPA} & TV-minimization  & hybrid\\
\hline
$r=2$ & $L=43$  &  28.4611  & {25.2711} & 28.6108 & 34.2899  & {\bf 35.6529} \\
$r=2$ &  $L=63$  &  30.4380  & {27.2604} & 30.4282 & 35.8917  & {{\bf 37.5627}} \\
$r=2$ &  $L=83$  &  31.9667 & {28.8285} & 31.9652 & {37.8554}  & {{\bf 39.0476}} \\ 
$r=2$ &  $L=103$  &  33.6619  & {30.5811} & 33.6409 & {39.4806}  & {{\bf 40.5393}}\\ 
$r=2$ &  $L=163$  &  38.4152  & {35.7034} & 38.4017 & {42.1741}  & {{\bf 44.5973}} \\ 
$r=2$ &  $L=183$  &  39.7616  & {37.3065} & 39.7681 & {42.1954}  & {\bf 45.1643}\\
$r=2$ & $L=255$ & {42.2019} & {42.1127} & {42.1749}& {41.1258} 
& {{\bf 44.3725}} 
\\
\hline
$r=4$ & $L=43$  &  28.2920  & {25.2711}& 28.4313 & 33.4018 &  {\bf 34.8641}\\
$r=4$ &  $L=63$  &  30.0766  & {27.2604}& 30.0655 & {34.5426} &  {{\bf 35.8546}}\\ 
$r=4$ &  $L=83$  &  31.2739  & {28.8285}& 31.2704 & {34.6659} &  {\color{red}{\bf 36.0348}}\\ 
$r=4$ &  $L=103$  &  32.2709  & {30.5811}& 32.2589 & {34.2138} &  {{\bf 35.5611}}\\ 
$r=4$ &  $L=127$ &  {32.6784} & {32.5435} & {32.6479} 
& {33.2822} & {{\bf 34.3394}} \\ 
\hline
$r=6$ &  $L=35$  & 26.9866 & {24.2923} & {27.0308} &  {30.9692} & {{\bf 31.8935}} \\ 
$r=6$ &  $L=43$  &  27.6986 & {25.2711} & 27.8012 & 31.1364 & {\color{red}{\bf 32.1167}} \\
$r=6$ &  $L=63$  &  28.8521  & {27.2604} & 28.8466& {30.8352} & {{\bf 31.6424}} \\ 
$r=6$ &  $L=83$  & {29.1611} & {28.8285} & {29.1612} &  {30.2491} & {{\bf 30.8172}} \\ 
\hline
$r=8$ &  $L=31$  & {25.9750} & {23.7100} & {25.9541} &  {29.2602} & {\color{red}{\bf 29.8517}} \\ 
$r=8$ & $ L=35$  &  26.3523 &  {24.2923} & 26.3809 &  29.2534 & {{\bf 29.8214}}\\
$r=8$ &  $L=43$  &  26.8194  & {25.2711} &  26.8797 &  29.0807 & {\bf 29.6174} \\
$r=8$ &  $L=55$  & {27.2679} & {26.4991} & {27.2523} &  {28.5636} & {{\bf 29.0588}} \\ 
\hline
\end{tabular}
\end{center}
\label{tab1}
\end{table}

\begin{figure}[htbp]
\begin{center}
\hspace*{-5mm}	
        \includegraphics[width=0.29\textwidth,height=0.29\textwidth]{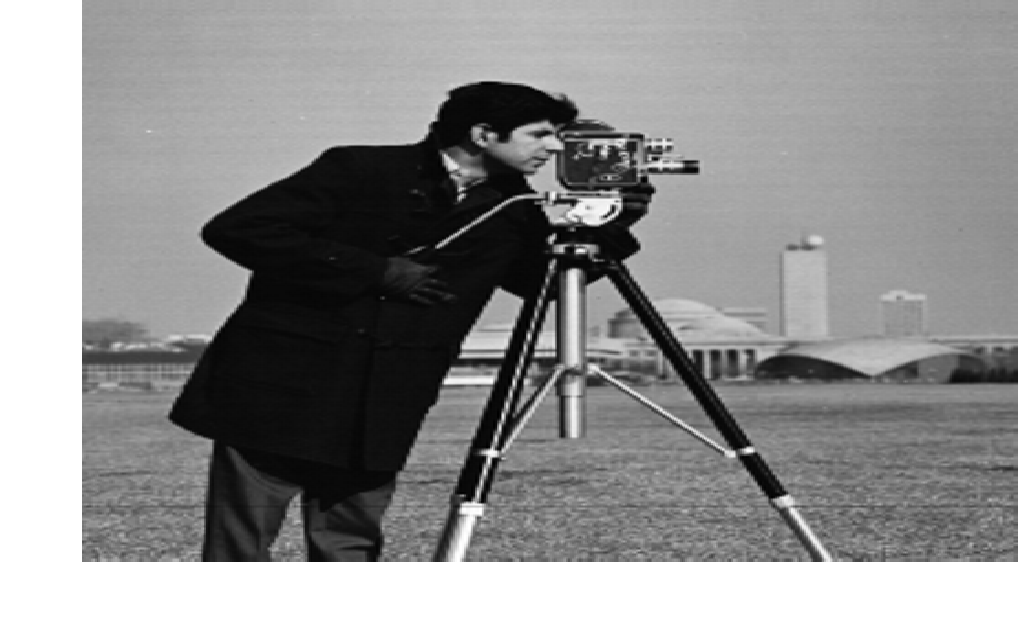}  %
	\includegraphics[width=0.29\textwidth,height=0.29\textwidth]{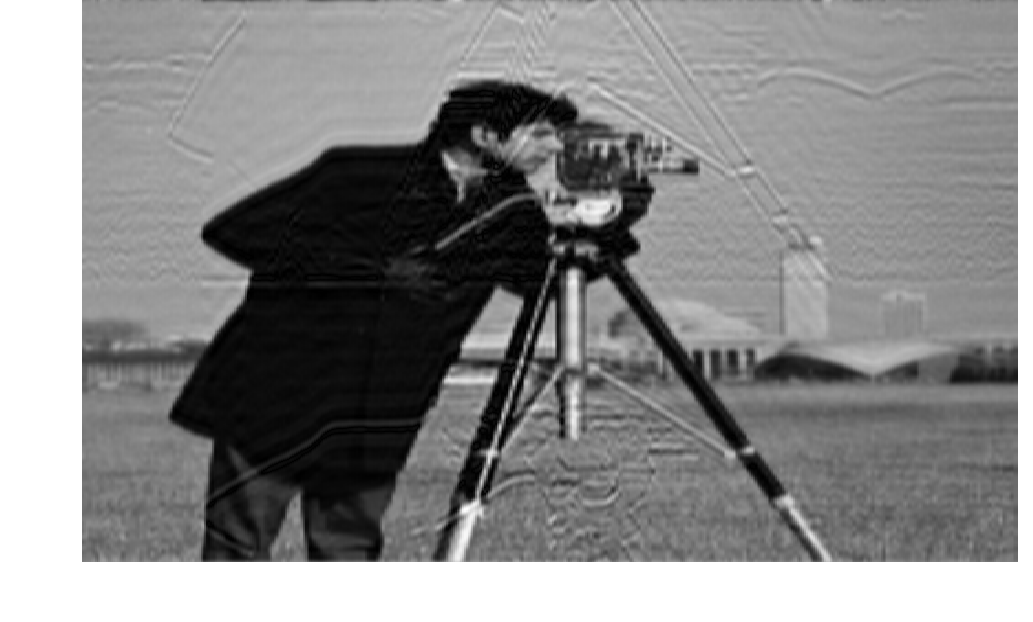}  
	\includegraphics[width=0.29\textwidth,height=0.29\textwidth]{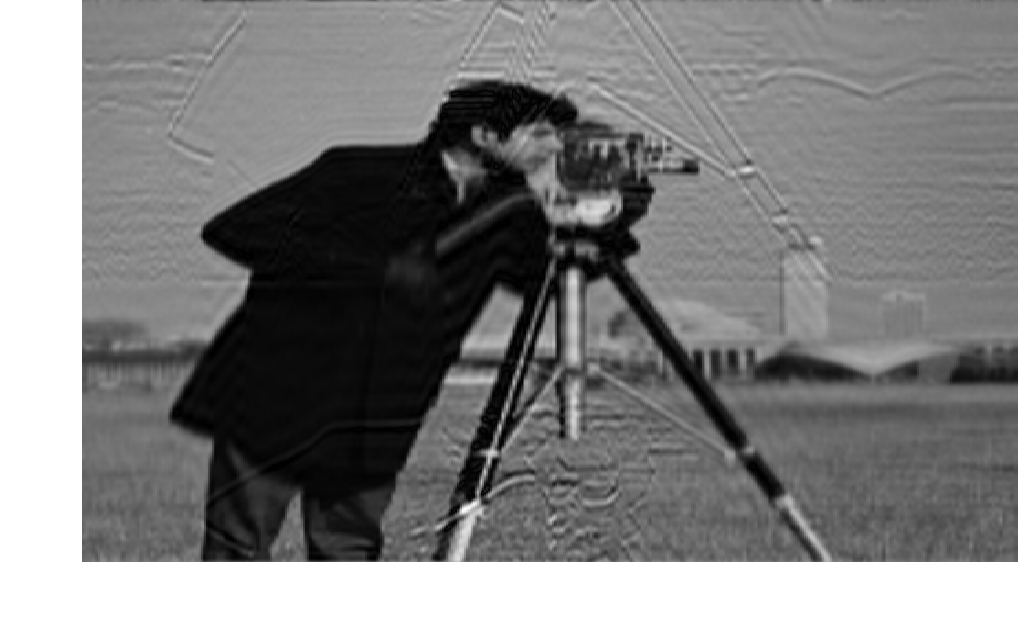} \\[1ex]
\hspace*{-5mm}	
        \includegraphics[width=0.29\textwidth,height=0.29\textwidth]{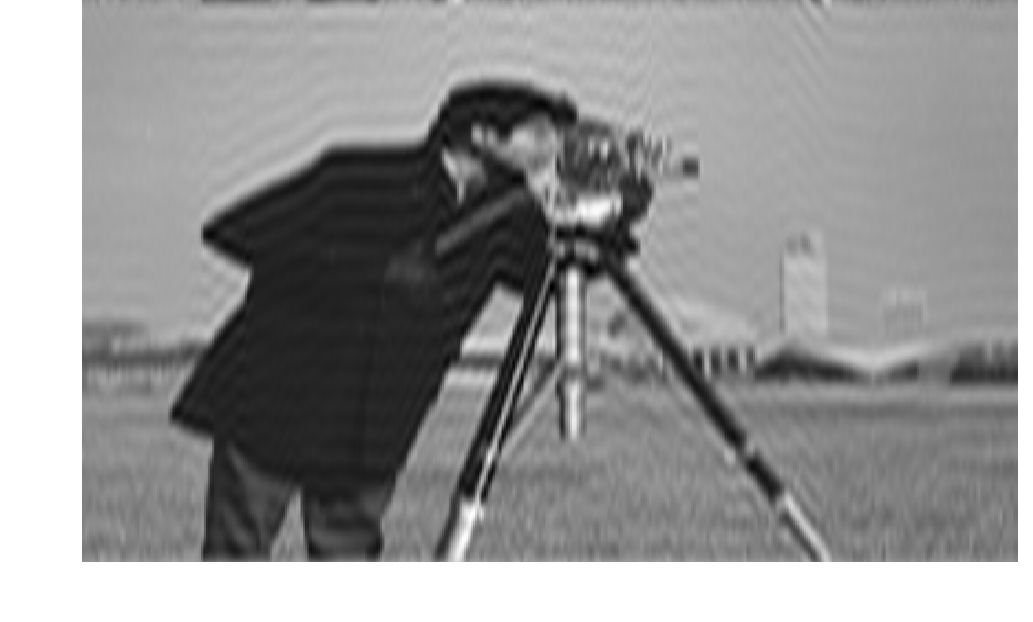} 
	\includegraphics[width=0.29\textwidth,height=0.29\textwidth]{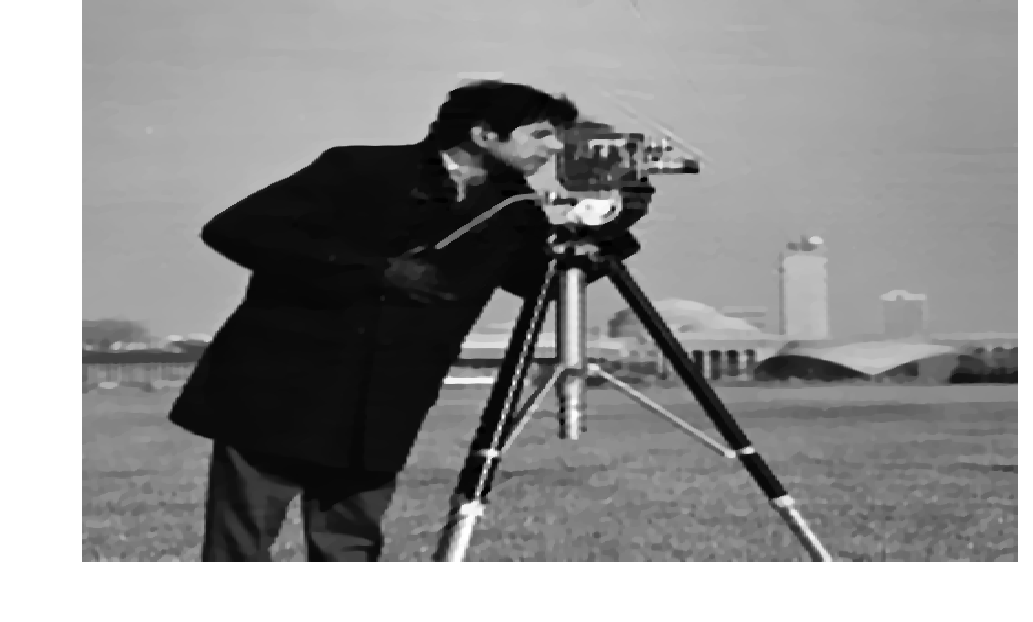}  %
	\includegraphics[width=0.29\textwidth,height=0.29\textwidth]{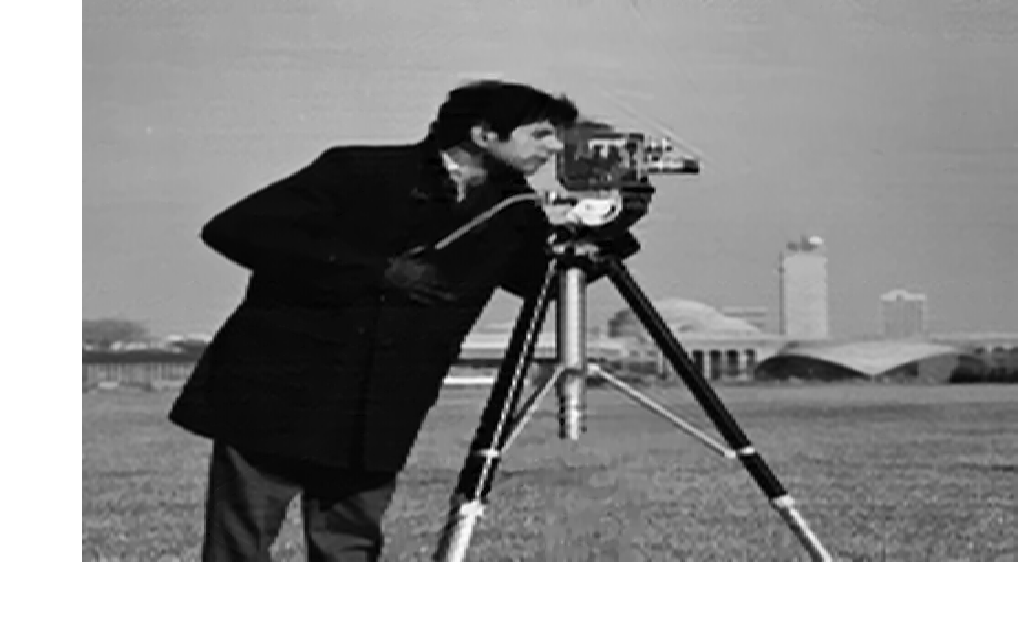}
	\end{center}
\caption{Original image and image reconstructions for  $r=6$ and $L=43$.
Top: Left: ''cameraman'' original image $512 \times 512$;
 Middle:  zero refilling with PSNR $27.6986$; 
 Right: GRAPPA method with PSNR $27.7064$;
Bottom: Left: low-pass reconstruction from  $L=43$ rows with PSNR 25.2711; 
Middle: TV minimization with PSNR 31.1364; Right: hybrid method 
 with PSNR 32.1167 (see Table  \ref{tab1})}.
\label{figure2}
\end{figure}

The reconstruction results for the $512 \times 512$ image ''cameraman'' are given in Table \ref{tab1}, see also Figure \ref{figure2} for the special case of reduction rate $r=6$, where $85$ rows of the 512 rows of $\hat{\mathbf A}$ are acquired. The interpolation  algorithm  GRAPPA uses  a local window of size $11 \times 11$ for the interpolation weights.
For Algorithm \ref{algTV} (TV-minimization) we  used the parameters $N_I=250$, 
$\tau=0.03$, $\theta=1.0$, $\sigma= 0.01+1/(8\tau)$ as well as $\lambda=100$ for $L=31, 35, 43, 55$, $\lambda=200$ for $L=63$, $\lambda=500$ for $L=83, 103$, and $\lambda=1000$ for $L=163, 183, 255$.
For the Hybrid Algorithm \ref{alghyp} we have taken the result of Algorithm \ref{algTV} as initial image, $N_I=10$, $N_s=3$,
$\mu = 1.6$, $\epsilon=0.05$, and window size $\gamma_1=\gamma_3=3$.

\begin{table}[htbp]
\scriptsize
\caption{Comparison of the reconstruction performance for incomplete Fourier data for the $512 \times 512$ ''boat'' 
image (PSNR values)}
\begin{center}
\begin{tabular}{llccccc}
\hline
reduction rate& low-pass width  &  zero refilling & only low pass & GRAPPA & TV-minimization  & hybrid\\
\hline
$r=2$ & $L=43$  &  27.5472  &  {{24.3435}} & 27.5607  & 30.4695  &  {\bf 30.9669} \\ 
$r=2$ &  $L=63$  &  29.2590 &{{26.1116}} & 29.2600 & 31.7578  & {\bf 32.4685} \\ 
$r=2$ &  $L=83$  &  30.6754 & {{27.6409}} & 30.6383 & 32.8004  & {\bf 33.7154} \\ 
$r=2$ &  $L=103$  &  32.1627  &{{29.1438}} & 32.1603 & 34.2960  & {\bf 34.9491}\\ 
$r=2$ &  $L=163$  &  35.5347 & {{32.9782}} & 35.5281 & 36.6132  & {\bf 37.6685} \\ 
$r=2$ &  $L=183$  &  36.4173  & {{34.1915}} & 36.4159 & 36.9317  & {\bf 38.2374} \\ 
$r=2$ &  $L=223$  &  37.8877  & {{36.5739}} & 37.8881 & 37.0482  & {\color{red}{\bf 38.8041}} \\ 
$r=2$ &  $L=255$  &   {38.2792}  & {{38.2014}} & {38.2529} & {36.9537}  & {{\bf 38.7321}} \\ 
\hline
$r=4$ & $L=43$  &  27.2689  & {24.3435}& 27.2808& 29.7853 &{\bf 30.4302}\\ 
$r=4$ &  $L=63$  &  28.7329  & {26.1116} & 28.7340 & 30.5864 &  {\bf 31.4227}\\ %
$r=4$ &  $L=83$  &  29.7577 & {27.6409}& 29.7333 & 30.8849 & {\bf 31.8436}\\ %
$r=4$ &  $L=103$  &  30.5436  & {29.1438}& 30.5414 &  {31.1424} &  {{\bf 31.9322}}\\ 
$r=4$ &  $L=127$ &  {30.7728}  & {30.6800}& {30.7369} &  {30.8902} &  {{\bf 31.6302}}\\ 
\hline
$r=6$ &  $L=43$  &  26.6263 & {24.3435}  & 26.6326 & 28.5838 & {\bf 29.1021} \\ %
$r=6$ &  $L=63$  &  27.5990  & {26.1116} & 27.6004 & 28.6371 & {\color{red}{\bf 29.1912}} \\ %
$r=6$ &  $L=83$  &  {27.8879}  & {27.6409} & {27.8875} & {28.3311} & {{\bf 28.8302}} \\ %
\hline
$r=8$ & $ L=35$  &  25.3020 &  {23.4108} & 25.2796 &  27.0459& {\bf 27.4014}\\  %
$r=8$ &  $L=43$  &  25.8638  & {24.3435} &  {25.8638} &  27.2016 & {\color{red}{\bf 27.5753}} \\ %
$r=8$ & $L=63$  &  {26.2281}  & {26.1116} & {26.1876} & {26.6439} & {{\bf 26.9781}} \\ %
\hline
\end{tabular}
\end{center}
\label{tab2}
\end{table}

The reconstruction results for the $512 \times 512$ image ''boat'' are given in Table \ref{tab2}. 
In  Figure \ref{figure3}  we present the obtained reconstructions for the special case of reduction rate $r=4$, where $128$ rows of the 512 rows of $\hat{\mathbf A}$ are acquired, with a low-pass area containing $L=103$ centered rows.
For Algorithm \ref{algTV} (TV-minimization) we  applied the parameters $N_I=250$, 
$\tau=0.03$, $\theta=1.0$, $\sigma= 0.01+1/(8\tau)$ and $\lambda=100$ for $r=4,6,8$.
We used $\lambda=100$ for $L\le83$,
$\lambda=200$ for $L=103, 127$, and $\lambda=500$ for $163 \le L \le 223$, and $L=1000$ for $L>223$.
For the Hybrid Algorithm \ref{alghyp} we took the result of Algorithm \ref{algTV} as  initial image, $N_I \le 10$, 
$\mu = 1.6$, $\epsilon=0.1$, window size $\gamma_1=\gamma_2=3$. Further, we applied two smoothing steps ($N_s=2$) for $r=4,6,8$ and $r=2$ with $L\le 83$ and only one smoothing step for $r=2$ and $L \ge 103$.
The GRAPPA algorithm uses a local window of size $11 \times 11$ to compute the interpolation weights form the low-pass area.

\begin{figure}[htbp]
\begin{center}
\hspace*{-5mm}	\includegraphics[width=0.29\textwidth,height=0.29\textwidth]{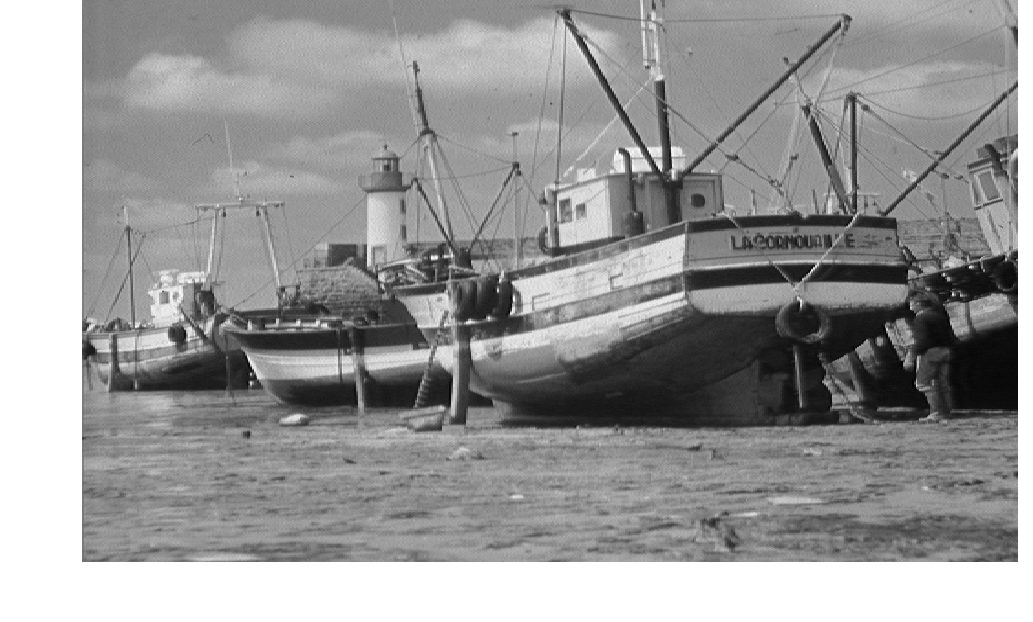}   
	\includegraphics[width=0.29\textwidth,height=0.29\textwidth]{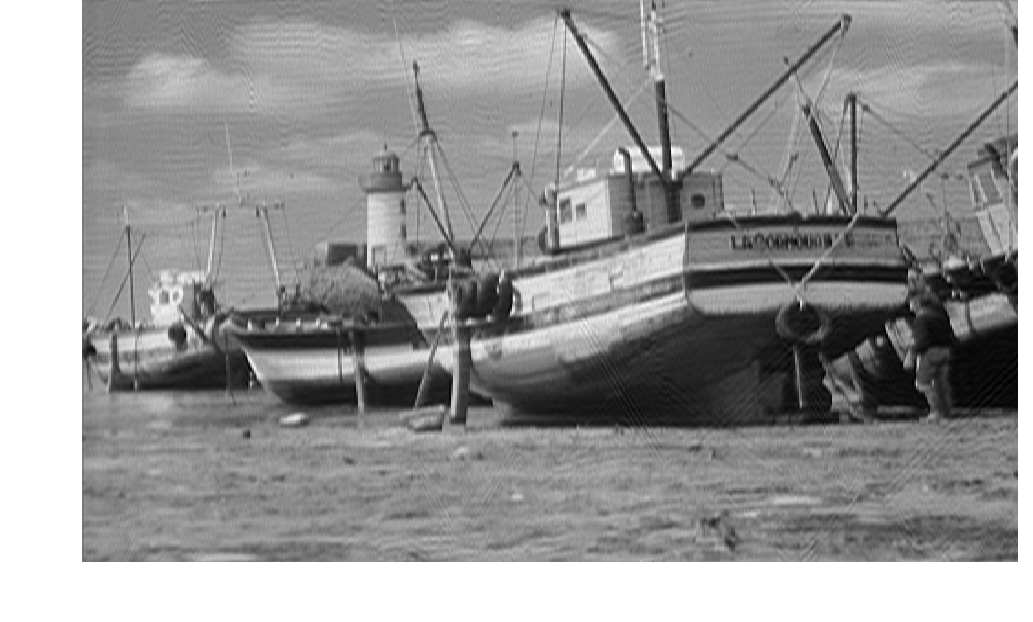} 
	\includegraphics[width=0.29\textwidth,height=0.29\textwidth]{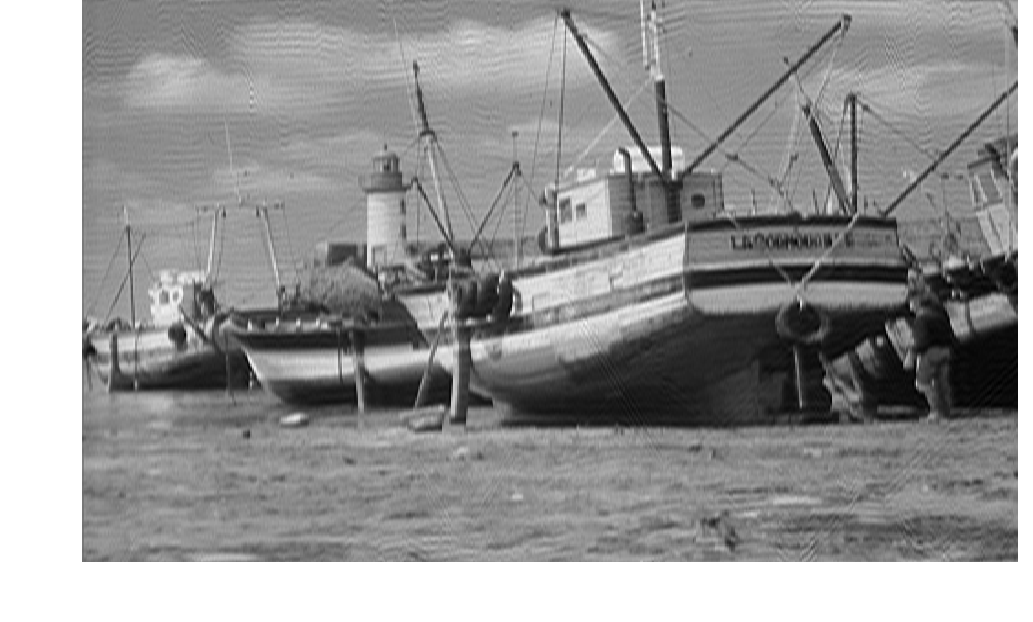} \\[1ex]
\hspace*{-5mm}	\includegraphics[width=0.29\textwidth,height=0.29\textwidth]{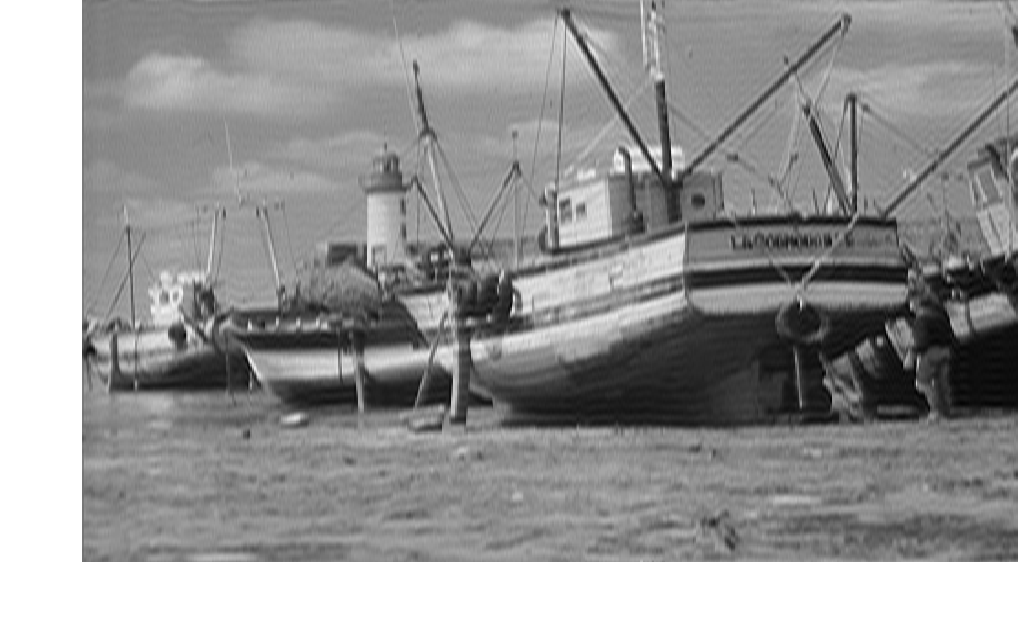}   
	\includegraphics[width=0.29\textwidth,height=0.29\textwidth]{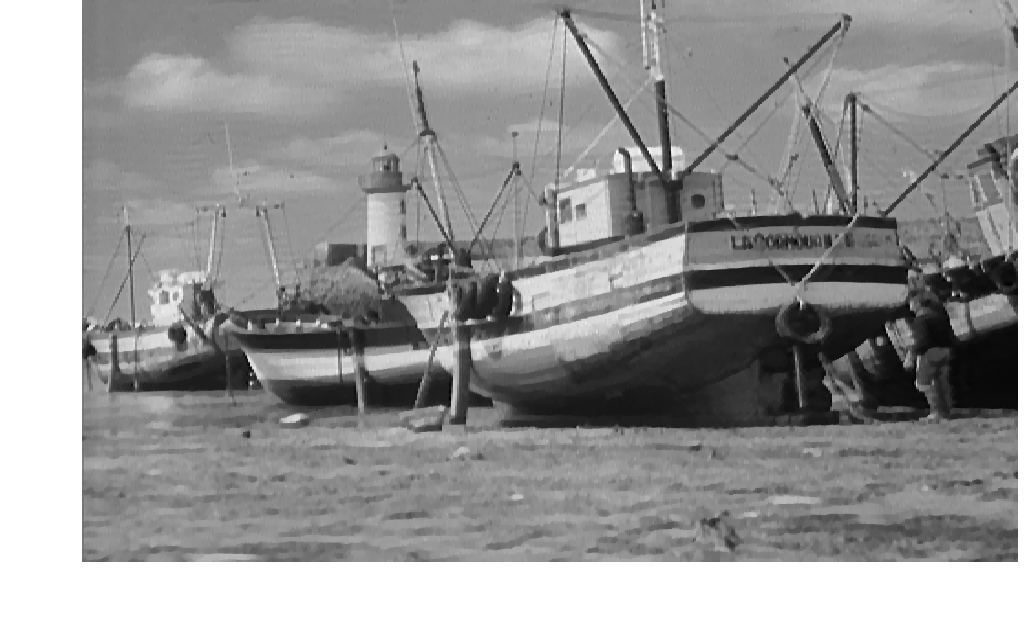}  
	\includegraphics[width=0.29\textwidth,height=0.29\textwidth]{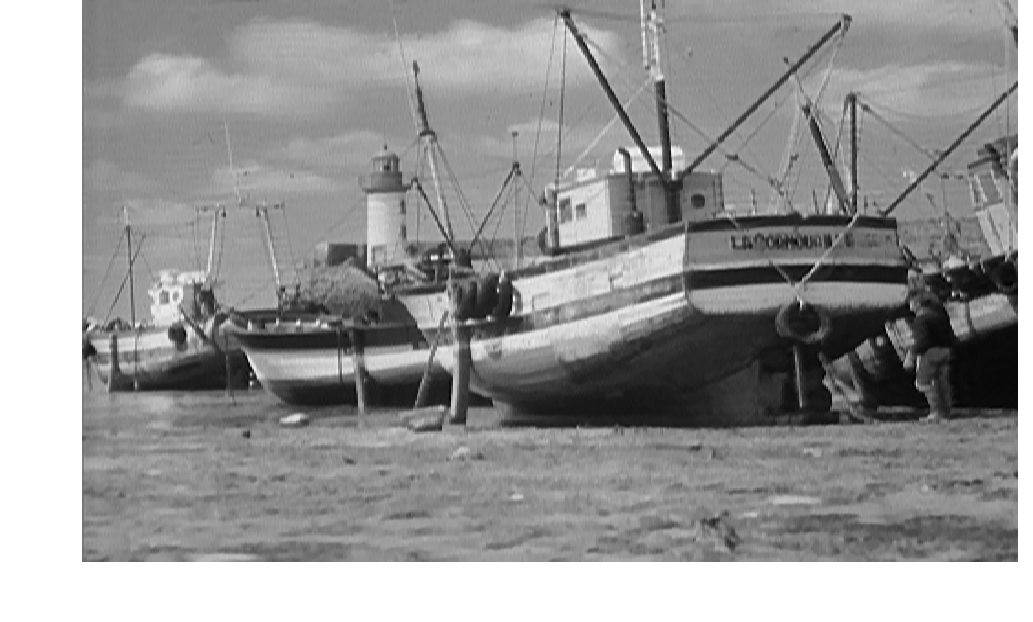}
	\end{center} 
\caption{Original image and image reconstructions for reduction rate $r=4$ and $L=103$.
Top: Left: ''boat'' original image $512 \times 512$;
 Middle:  zero refilling with $L=103$ and PSNR $30.5436$; 
 Right:  GRAPPA method with PSNR $30.5419$;
Bottom: Left: low-pass reconstruction from only $L=103$ rows with  PSNR $29.1438$; Middle:  TV minimization with  PSNR 31.1424; Right: hybrid method with PSNR 31.9459, see Table \ref{tab2}.
}
\label{figure3}
\end{figure}

\begin{table}[htbp]
\scriptsize
\caption{Comparison of the reconstruction performance for incomplete Fourier data for the $512 \times 512$ ``phantom''
image (PSNR values)}
\begin{center}
\begin{tabular}{llccccc}
\hline
reduction rate& low-pass width  &  zero refilling & only low pass & {GRAPPA} & TV-minimization  & hybrid\\
\hline
$r=2$ & $L=43$  & 28.0710 & {24.4681}  &  30.5464 &  41.3545 & {\bf 41.7740}\\ 
$r=2$ &  $L=63$  &  29.8508 & {26.4199}  & 32.6006 & 42.6828 & {\bf 43.2157} \\  
$r=2$ &  $L=83$  & 31.3967 & {28.0753}   &  34.0137& 43.1131& {\color{red} {\bf 43.6607}} \\ 
$r=2$ &  $L=103$  & 32.9290 &  {29.7686} & 35.2493 & 42.9847 & {\bf 43.5047} \\ 
$r=2$ &  $L=163$  &  35.8249 & {33.3771} & 37.1173 & 41.9584  & {\bf 42.3462}\\ 
$r=2$ &  $L=255$ &  {36.5454}  &  {36.5028}  & {36.5365} &   {39.5498}  & {{\bf 39.7747}} \\ 
\hline
$r=4$ & $L=35$  & 26.5838 & {23.4970} & 28.8051 & 37.1276 & {\bf 37.2101}\\ %
$r=4$ & $L=43$  & 27.6676& {24.4681} & 29.6379 & 37.2433 & {\bf 37.3264} \\ 
$r=4$ &  $L=63$  & 29.1792 & {26.4199} & 31.0389 & 37.3916 & {\color{red} {\bf 37.4674}}\\ 
$r=4$ &  $L=83$  & 30.3271  & {28.0753}  & 31.7063 & 37.1588 & {\bf 37.2312}\\ 
$r=4$ & $L=127$ &  {31.5192}  &  {31.4228}  & {31.5420} &   {36.1592}  & {{\bf 36.2218}} \\ 
\hline
$r=6$ & $L=19$  &  24.3204  &  21.2634  & 25.7926 &   34.7709  & {\bf 34.8233} \\ 
$r=6$ &  $L=27$ & 25.5429 & 22.5247 & 27.4372& 35.2243 & {\color{red} {\bf 35.2759}}\\
$r=6$ &  $L=43$  & 27.1459 & 24.4681  & 28.5582 &  34.9036  & {\bf 34.9483}         \\ 
$r=6$ &  $L=63$  &  28.0999  & 26.4199 & 28.9733 & 33.2123 & {\bf 33.2621} \\  
$r=6$ & $L=83$  &  {28.4304}  &  {28.0753}  & {28.5262} &   {31.5902}  & {{\bf 31.6129}} \\ 
\hline
$r=8$ & $ L=19$  &   24.0342 &  {21.2634} & 25.2564 &   32.2327 &  {\color{red} {\bf 32.2591}} \\ 
$r=8$ & $ L=27$  & 25.1295 &  {22.5247} &  26.5647 &    31.9117 & {\bf 31.9433} \\ 
$r=8$ & $ L=35$  &  25.6499  & {23.4970}   & 26.8477 & 31.0930 & {\bf 31.2566} \\  
$r=8$ &  $L=43$  &  26.2553  &  {24.4681}  & 26.9402 &   30.2271  & {\bf 30.4698} \\  
$r=8$ &  $L=63$  &  26.6015  &  26.4199  & {26.5920} &   {28.6672}  & {{\bf 28.8872}} \\  
\hline
\end{tabular}
\end{center}
\label{tab3}
\end{table}

Next, we consider the MRI ``phantom'' image of size $512 \times 512$, see Table \ref{tab3} and Figure \ref{figure4}.
 Because of the special cartoon-like structure of this image, the Hybrid Algorithm \ref{alghyp} does not add much improvement to the results achieved by the TV minimization. In contrast to the other images, for higher reduction rates we obtain better recovery results  when taking a  low-pass area with smaller width $L$, since it is more important here to catch the higher frequencies.
For Algorithm \ref{algTV} (TV-minimization) we  applied the parameters $N_I=250$, 
$\tau=0.03$, $\theta=1.0$, $\sigma= 0.01+1/(8\tau)$ and $\lambda=500$.
For the Hybrid Algorithm \ref{alghyp} we took the result of Algorithm \ref{algTV} as the initial image, applied one smoothing step $N_{s}=1$ for $r=8$, $L=35,\, 43$, and $N_{s}=0$ otherwise, $N_I\le 15$, 
$\mu = 1.6$, $\epsilon=0.1$, window size $\gamma_1=\gamma_3=3$.  The interpolation algorithm (GRAPPA) uses a local window of size $11 \times 11$ to compute the interpolation weights form the low-pass area.
For GRAPPA we employed the implementation by M.~Lustig in the ESPIRiT toolbox for the special case of only one coil, which for this image provided better results than our implementation (in contrast to the other two images). Corresponding software can be found under \texttt{https://people.eecs.berkeley.edu/\textasciitilde mlustig/Software.html}, the applied GRAPPA algorithm is also contained in our software package.
For phantom image example, the best reconstruction results for different reduction rates are presented in Figure \ref{figure5}.

\begin{figure}[htbp]
\begin{center}
\hspace*{-5mm}	
        \includegraphics[width=0.29\textwidth,height=0.29\textwidth]{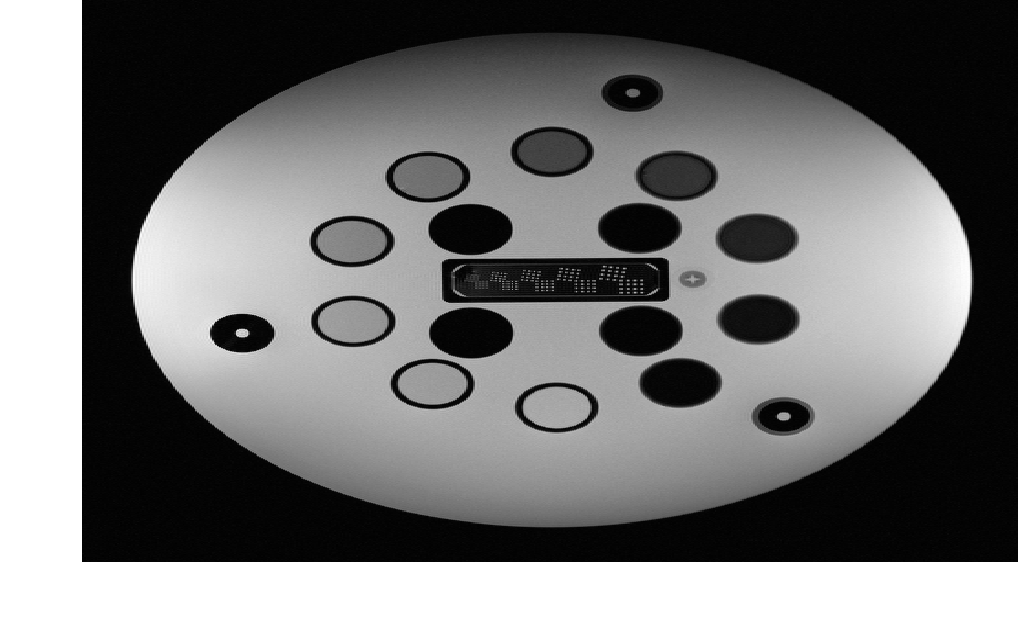}  
	\includegraphics[width=0.29\textwidth,height=0.29\textwidth]{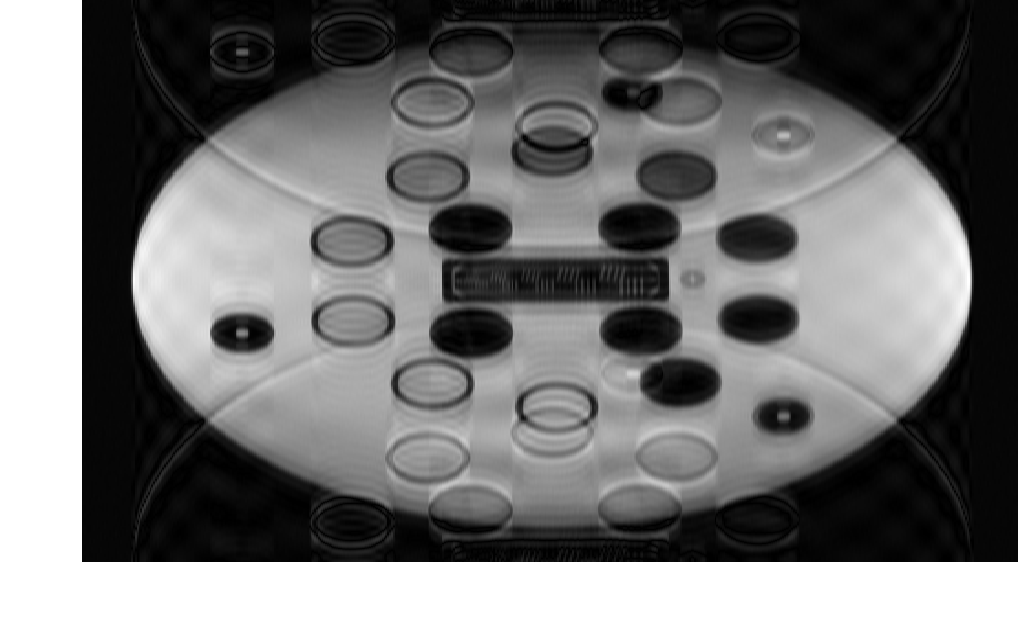} 
	\includegraphics[width=0.29\textwidth,height=0.29\textwidth]{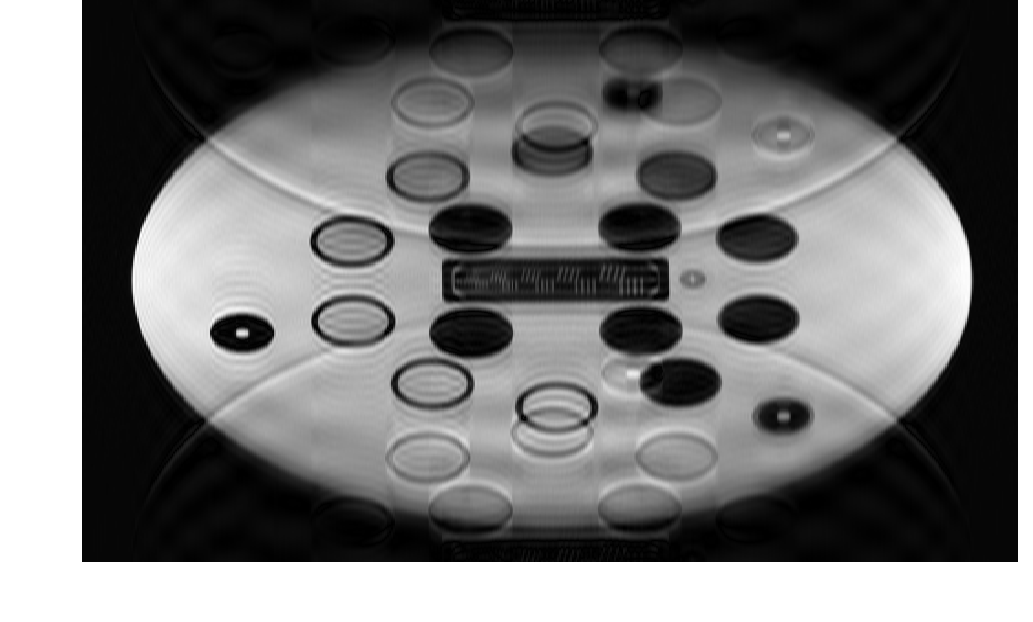} \\[1ex]
\hspace*{-5mm}
	\includegraphics[width=0.29\textwidth,height=0.29\textwidth]{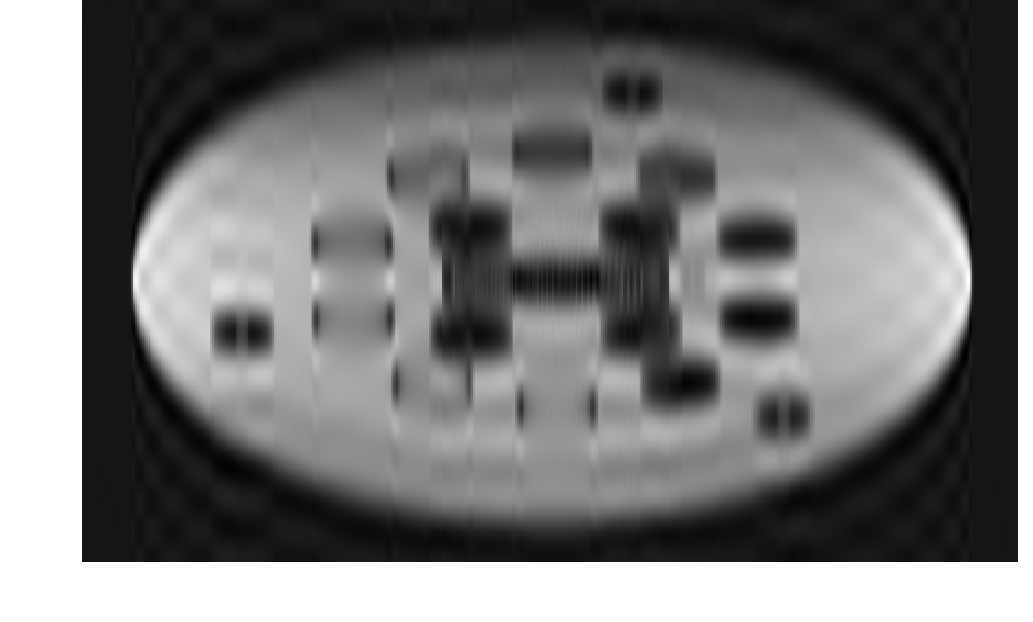}  
	\includegraphics[width=0.29\textwidth,height=0.29\textwidth]{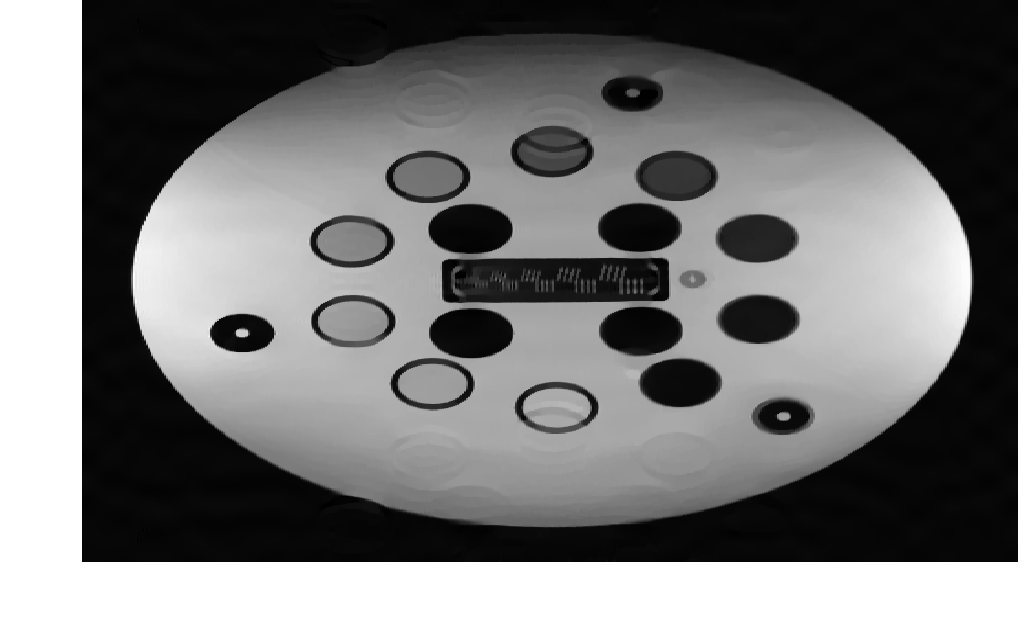}  
	\includegraphics[width=0.29\textwidth,height=0.29\textwidth]{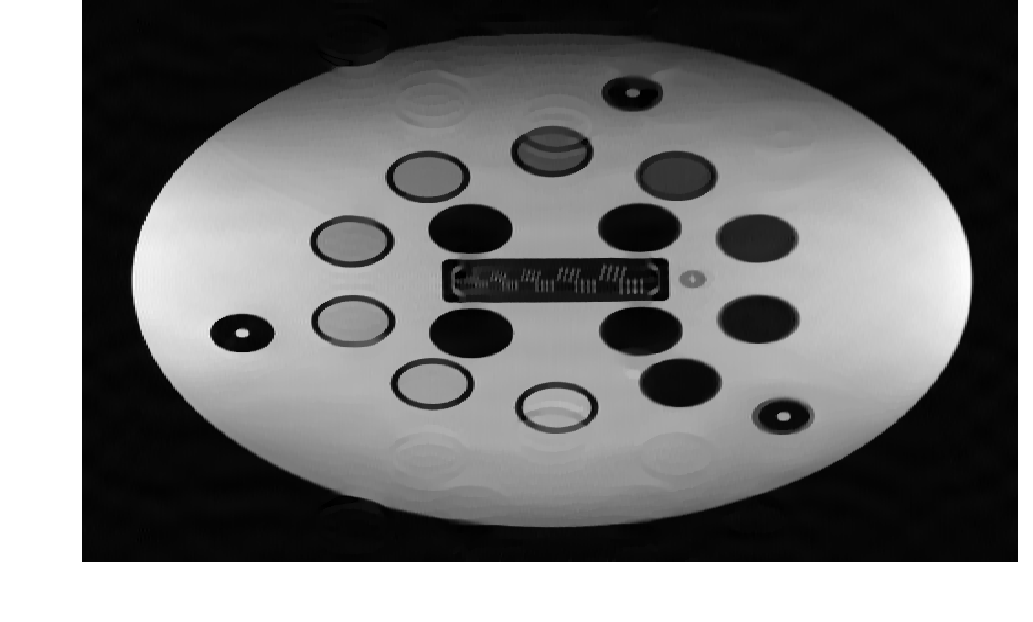}
	\end{center}
\caption{Original image and image reconstructions for reduction rate $r=8$ and $L=19$.
Top: Left: ``phantom'' original image $512 \times 512$;
 Middle:  zero refilling with PSNR $24.0342$; 
 Right: GRAPPA method with PSNR $25.2564$;
Bottom: Left: low-pass reconstruction from only $L=19$ rows with PSNR 21.2634; Middle: TV minimization with PSNR 32.2327; Right:  hybrid method
 with PSNR 32.2591, see Table \ref{tab3}.}
\label{figure4}
\end{figure}

\begin{figure}[htbp]
\begin{center}\hspace*{-5mm}
\includegraphics[width=0.21\textwidth,height=0.21\textwidth]{Figures1/phantom19_8hybrid.png}  
\includegraphics[width=0.21\textwidth,height=0.21\textwidth]{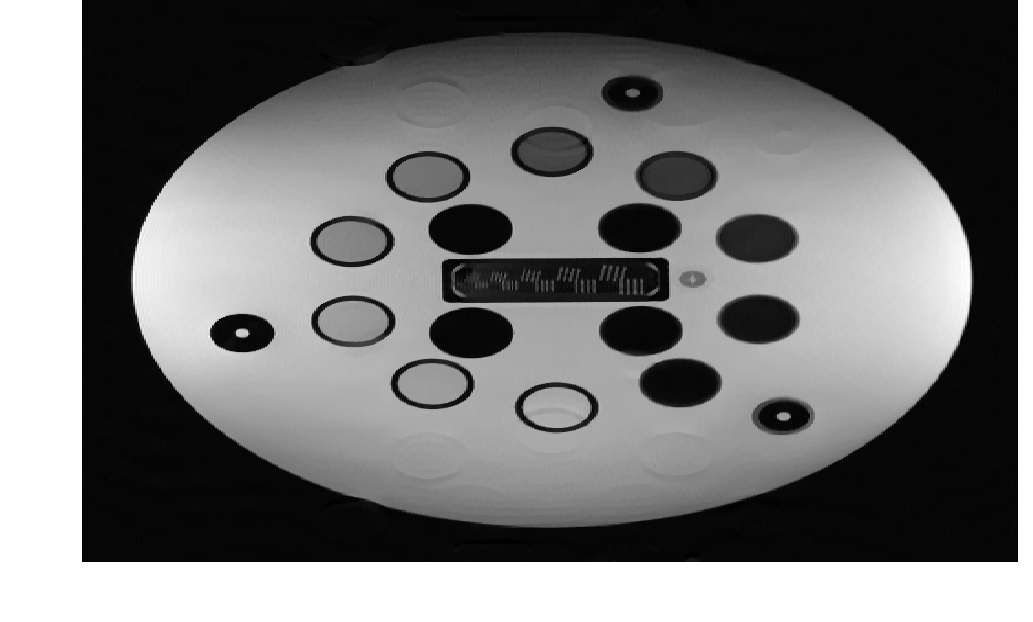} 
\includegraphics[width=0.21\textwidth,height=0.21\textwidth]{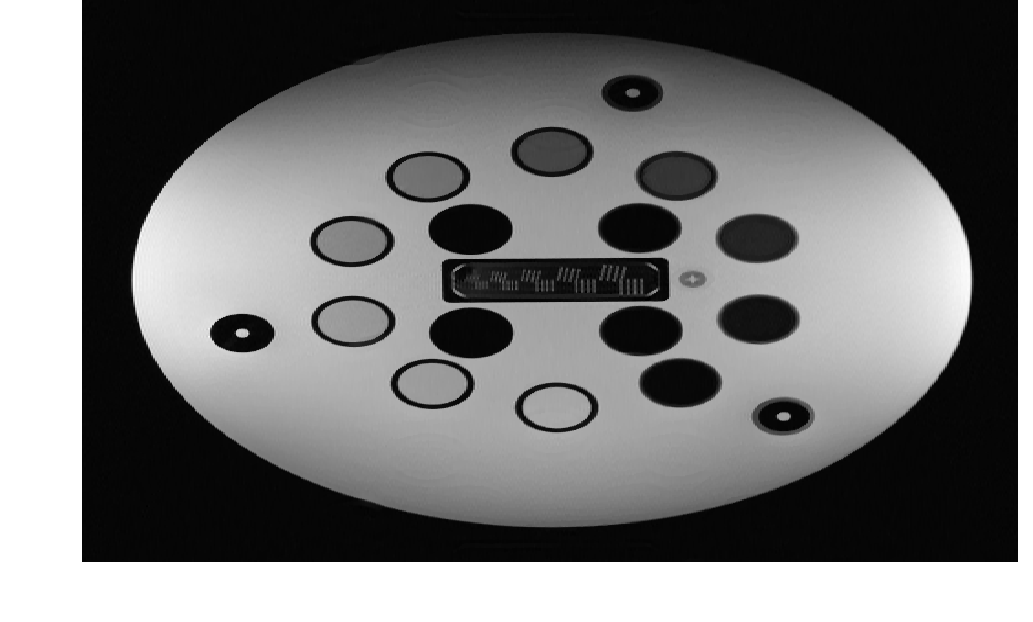} 
\includegraphics[width=0.21\textwidth,height=0.21\textwidth]{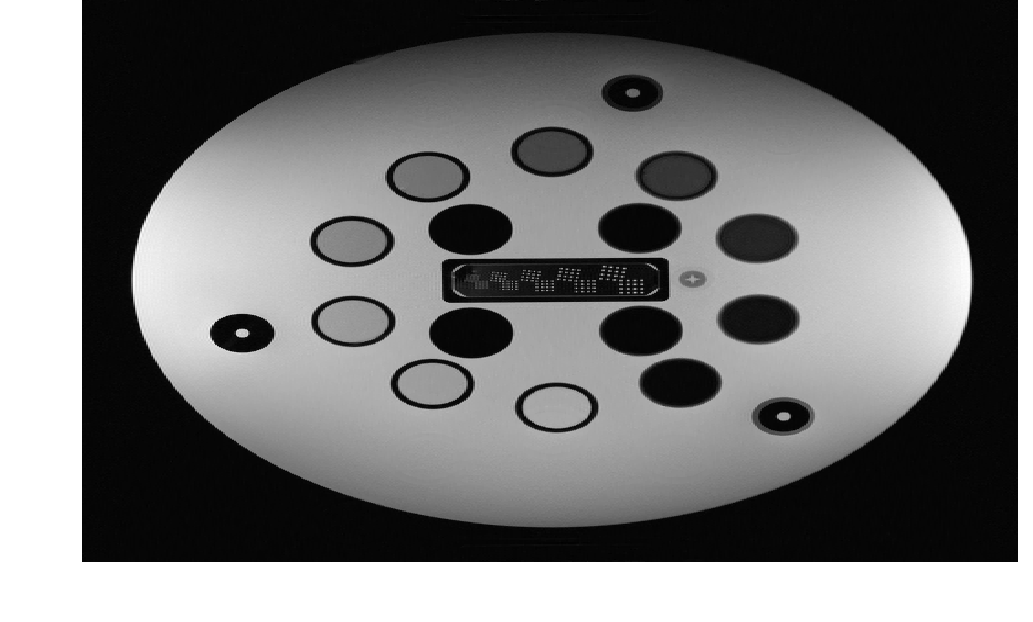}  
	\end{center} 
\caption{Best image reconstructions using the proposed hybrid method for the ``phantom'' image for different reduction rates.
From left to right: first: $r=8$, $L=19$;   second: $r=6$, $L=27$;  third: $r=4$, $L=63$; fourth: $r=2$, $L=83$.
}
\label{figure5}
\end{figure}

\paragraph{Generalizations of the sampling grid.}
Note that the TV-minimization Algorithm \ref{algTV} can be applied to every sampling pattern.
To show that the chosen sampling pattern (\ref{set}) is favourable, we compare our sampling set $\Lambda \times \Lambda_M$  with sampling sets $\Lambda^{(3)} \times \Lambda_M$ and $\Lambda^{(4)} \times \Lambda_M$, where beside the low-pass rows, only every third or only every fourth row of Fourier data is taken, until a total of $\lfloor N/r\rfloor$ of rows is reached, i.e., 
$$ \Lambda^{(3)} = \Lambda_K^{(3)} \cup \Lambda_L, \qquad  \Lambda^{(4)} = \Lambda_K^{(4)} \cup \Lambda_L, $$
where  $\Lambda_K^{(3)}= \{ -3\kappa+2 \ldots , -1, \, 2, \, 5 ,\,  8, \ldots , 3\kappa-1\}$ with $K=2\kappa$ indices,
and $\Lambda_K^{(4)}= \{ -4\kappa-1, \ldots , -1, \, 3, \, 7 , \, 11, \ldots , 4\kappa-1\}$ with  $K=2\kappa+1$ indices.

Algorithm  \ref{algTV} is applied here with $N_I=250$, 
$\tau=0.03$, $\theta=1.0$, $\sigma= 0.01+1/(8\tau)$ and $\lambda=500$. The results of this comparison, presented in Table \ref{tab4}, show that these other sampling strategies usually do not lead to better reconstruction results, compared to  $\Lambda \times \Lambda_M$ in (\ref{set}).


\begin{table}[htbp]
\scriptsize
\caption{Comparison of the reconstruction performance for incomplete Fourier data for the $512 \times 512$ ''cameraman'' 
image for different sampling sets (PSNR values)}
\begin{center}
\begin{tabular}{llccc}
\hline
reduction rate& low-pass width  &   $\Lambda \times \Lambda_M$ &  $\Lambda^{(3)} \times \Lambda_M$ &  $\Lambda^{(4)}\times \Lambda_M$\\
\hline
$r=4$ & $L=43$  & {\bf 33.6802}& 27.3641 & 30.0970  \\
$r=4$ &  $L=63$  & {\color{red}{\bf 34.6699}} & 29.3119 & 32.0558 \\
$r=4$ &  $L=83$  & {\bf 34.6659} & 33.4319 & 33.3220 \\
\hline
$r=6$ &  $L=35$ & {\bf 31.0799} & 29.5245 & 28.8486 \\
$r=6$ &  $L=43$  & {\color{red} {\bf 31.2923}} & 26.8227 & 29.9065 \\
$r=6$ &  $L=63$  &  30.8950 & 28.8150 & {\bf 31.1178} \\
\hline
$r=8$ & $ L=27$  & {\bf 29.2082} & 24.6907 & 27.6181 \\
$r=8$ & $ L=35$  & {\color{red} {\bf 29.3146}} & 28.8143 & 28.5050 \\
$r=8$ &  $L=43$  & 29.1544 & 26.7357 & {\bf 29.1909} \\
\hline
\end{tabular}
\end{center}
\label{tab4}
\end{table}

Finally, we remark that also  the Hybrid-Algorithm \ref{alghyp} can be generalized to other symmetric sampling schemes. 
Considering for example a symmetric sampling set defined by 
\begin{align}\label{set2}  (\Lambda_L \cup \Lambda_K^{(2)}) \times (\Lambda_L \cup \Lambda_K^{(2)}), 
\end{align}
where, beside the low-pass box of size $L \times L$, only every second data value from every second row and second column is taken.
Here $K$ is chosen such that the sampling set has (approximately) $\lfloor \frac{NM}{r} \rfloor$ elements for reduction rate $r$.  Generalizing  Algorithm \ref{alghyp}, we start with a symmetric smoothing filter in step 1.
Then we apply a symmetric scheme to compute the local total variation in step 2.
In step 3, we compute  two  weight matrices,  ${\mathbf W}_1$ as before by comparing the MTV values of the upper and the lower half of the image, and  ${\mathbf W}_2$ by comparing the MTV values of the left  and the right half of the image.  In step 4 we compute the two updates of the image obtained using ${\mathbf W}_1$ and ${\mathbf W}_2$ and take the average.

While this sampling pattern is of less interest in the MRI application, it can produce even better reconstruction results for suitable choices of the width of the low-pass box $L$, see Table \ref{tabneu}.
Algorithm  \ref{algTV} is applied here with $N_I=250$, 
$\tau=0.03$, $\theta=1.0$, $\sigma= 0.01+1/(8\tau)$ and $\lambda=1000$. The modified Algorithm  \ref{alghyp} is applied with one smoothing step (in both dimensions), $\mu=1.6$, $\epsilon=0.05$, and window size $\gamma_1=\gamma_3=3$.
The best reconstruction results are consistently obtained when only using the low-pass box, with 
Algorithm \ref{alghyp} providing the overall highest reconstruction quality.

\begin{table}[htbp]
\scriptsize
\caption{Reconstruction results for the  $512 \times 512$ ''cameraman'' image 
using incomplete Fourier data with the sampling set (\ref{set2}) (PSNR values)}
\begin{center}
\begin{tabular}{llcccc}
\hline
reduction rate& low-pass box width  &   low-pass  & zero refilling &   TV-minimization & hybrid\\
\hline
$r=4$ & $L=83$  &  25.1954 & 26.4749  & 30.7861 & \textbf{31.1797} \\
$r=4$ & $L=123$  & 28.4682 & 29.7242  & 33.3370 & \textbf{34.4250} \\
$r=4$ & $L=163$  & 31.8562 & 32.9807  & 34.9754 & \textbf{37.1528} \\
$r=4$ & $L=203$  & 35.1089 & 36.0386  & 35.9176 & \textbf{39.2978} \\
$r=4$ & $L=243$  & 38.0618 & 38.4462 & 36.2637 & {\color{red}\textbf{40.7206}} \\
\hline
$r=6$ & $L=83$  &  25.1954 & 26.4259  & 30.4327 & \textbf{31.0496} \\
$r=6$ & $L=123$ &   28.4682         &   29.6170     &  32.5612       &  \textbf{33.9776}\\
$r=6$ & $L=163$ &      31.8562             &  32.7321       &  33.6692    & \textbf{36.1276}\\
$r=6$ &  $L=203$ & 35.1089 & 35.3186 & 34.1665 & {\color{red}\textbf{37.2759}} \\
\hline
$r=8$ &  $L=83$  &   25.1954   & 26.3600 & 30.0377 &\textbf{30.7614} \\
$r=8$ &  $L=123$ & 28.4682  & 29.4712    &  31.7941   &\textbf{33.3062} \\
$r=8$ &  $L=179$  &   33.2193  & 33.2725  &  32.6592  &{\color{red}\textbf{35.1124}} \\
\hline
\end{tabular}
\end{center}
\label{tabneu}
\end{table}

\section{Conclusion}
Our study of image recovery from  structured  2D DFT data  and the numerical experiments provide several interesting insights.\\
The reconstruction results achieved with the presented methods strongly depend on the considered images.
Clearly, images containing many small details, which correspond to more information in  high Fourier frequencies can be reconstructed only with smaller accuracy. 
The newly  proposed  TV-Minimization Algorithm 
\ref{algTV} and the Hybrid Algorithm \ref{alghyp}  always lead to 
significantly  better reconstruction results than a low-pass reconstruction. \\
The width $L$ of the low-pass area in the sampling pattern (\ref{set})  plays an important role for  the reconstruction performance, where  the best choice of $L$  depends on the specific  image.
Cartoon like images (such as the ``phantom'' image)  can be better reconstructed taking a low-pass area with smaller width $L$, while the further band- and high-pass information provided by the  rows of Fourier data outside the low-pass area is very important.\\
 For the ''cameraman'' image, the Hybrid Algorithm \ref{alghyp} improves the PSNR result of the 
 TV minimization
significantly, while for the ``phantom'' image only slight improvements are obtained. 
The Hybrid Algorithm \ref{alghyp}, which can be seen as a post-processing method  of Algorithm \ref{algTV}, achieves very good reconstruction performance while using a rather low number of iteration steps for TV minimization. \\
The new insights  about suitable  sampling patterns  will be exploited to improve 
 current approaches for image recovery in parallel MRI, where undersampled $k$-space data of several coil images is available to reconstruct the desired magnetization image.

\subsection*{Acknowledgements}
We are grateful for several suggestions of the reviewers to improve this manuscript.
G. Plonka acknowledges funding from the European Union’s Horizon 2020 Research and Innovation Staff Exchange program under the MSCA grant agreement No.\ 101008231 (EXPOWER) and from German Research Foundation in the framework of the CRC 1456. A. Riahi acknowledges support from the German Research Foundation in the framework of the RTG 2088.

\medskip
\subsection*{Declarations}

\textbf{Data Availability.}  All used images are ``Standard''  test images and have been taken from the open source platform \texttt{https://www.imageprocessingplace.com/root\_files\_V3/ \break image\_databases.htm}.
The MATLAB software to reproduce the numerical results in this paper can be found at 
\texttt{https://na.math.uni-goettingen.de} under \texttt{software}.
\medskip

\noindent
\textbf{Conflict of Interest.} The authors declare no competing interests.

{\small
\bibliographystyle{plain}
\bibliography{references1}

\begin{thebibliography}{10}

\bibitem{Gelb16}
R.~Archibald, A.~Gelb, and R.B. Platte.
\newblock Image reconstruction from undersampled {F}ourier data using the
  polynomial annihilation transform.
\newblock {\em J. Sci. Comput.}, 67:432--452, 2016.

\bibitem{Beinert22}
Robert Beinert and Michael Quellmalz.
\newblock Total variation-based reconstruction and phase retrieval for
  diffraction tomography.
\newblock {\em SIAM Journal on Imaging Sciences}, 15(3):1373--1399, 2022.

\bibitem{Block07}
Kai~Tobias Block, Martin Uecker, and Jens Frahm.
\newblock Undersampled radial {MRI} with multiple coils. iterative image
  reconstruction using a total variation constraint.
\newblock {\em Magnetic Resonance in Medicine}, 57(6):1086--1098, 2007.

\bibitem{Borwein}
Jonathan Borwein and Adrian Lewis.
\newblock {\em Convex {A}nalysis and {N}onlinear {O}ptimization: {T}heory and
  {E}xamples}.
\newblock Springer New York, 2nd edition, 2006.

\bibitem{Candes06}
E.J. Candes, J.~Romberg, and T.~Tao.
\newblock Robust uncertainty principles: exact signal reconstruction from
  highly incomplete frequency information.
\newblock {\em IEEE Transactions on Information Theory}, 52(2):489--509, 2006.

\bibitem{C04}
Antonin Chambolle.
\newblock An algorithm for total variation minimization and applications.
\newblock {\em J. Math. Imaging Vis.}, 20:89--97, 2004.

\bibitem{CP10}
Antonin Chambolle and Thomas Pock.
\newblock A first-order primal-dual algorithm for convex problems with
  applications to imaging.
\newblock {\em J. Math. Imaging Vis.}, 40:120--145, 2010.

\bibitem{CP16}
Antonin Chambolle and Thomas Pock.
\newblock An introduction to continuous optimization for imaging.
\newblock {\em Acta Numerica}, 25:161--319, 2016.

\bibitem{Feng14}
L.~Feng, R.~Grimm, K.T. Block, H.~Chandarana, S.~Kim, J.~Xu, L.~Axel, D.K.
  Sodickson, and R.~Otazo.
\newblock Golden-angle radial sparse parallel {MRI}: combination of compressed
  sensing, parallel imaging, and golden-angle radial sampling for fast and
  flexible dynamic volumetric {MRI}.
\newblock {\em Magn. Reson. Med.}, 72(3):707--717, 2014.

\bibitem{Gilles}
Jerome Gilles.
\newblock Primal{D}ual-{M}atlab,
  https://github.com/jegilles/{P}rimal{D}ual-{M}atlab, 2024.

\bibitem{grappa}
M.A. Griswold, P.~M. Jakob, R.M. Heidemann, M.~Nittka, V.~Jellus, J.~Wang,
  B.~Kiefer, and A.~Haase.
\newblock Generalized autocalibrating partially parallel acquisitions
  ({GRAPPA}).
\newblock {\em Magn. Reson. Med.}, 47(6):1202--1210, 2002.

\bibitem{Hyun18}
Chang~Min Hyun, Hwa~Pyung Kim, Sung~Min Lee, Sungchul Lee, and Jin~Keun Seo.
\newblock Deep learning for undersampled {MRI} reconstruction.
\newblock {\em Physics in Medicine \& Biology}, 63(13):135007, 2018.

\bibitem{Karp88}
J.S. Karp, G.~Muehllehner, and R.~M. Lewitt.
\newblock Constrained {F}ourier space method for compensation of missing data
  in emission computed tomography.
\newblock {\em IEEE Trans. Med. Imaging}, 7(1):21--25, 1988.

\bibitem{Keeling}
S.~L. Keeling, C.~Clason, M.~Hinterm\"uller, F.~Knoll, A.~Laurain, and G.~von
  Winckel.
\newblock An image space approach to {C}artesian based parallel {MR} imaging
  with total variation regularization.
\newblock {\em Med. Image Anal.}, 16(1):189--200, 2012.

\bibitem{Knoll11}
Florian Knoll, Kristian Bredies, Thomas Pock, and Rudolf Stollberger.
\newblock Second order total generalized variation ({TGV}) for {MRI}.
\newblock {\em Magnetic Resonance in Medicine}, 65(2):480--491, 2011.

\bibitem{Krahmer14}
Felix Krahmer and Holger Rauhut.
\newblock Structured random measurements in signal processing.
\newblock {\em GAMM-Mitt.}, 39(2):217--238, 2014.

\bibitem{spirit}
M.~Lustig and J.M. Pauly.
\newblock {SPIRIT}: Iterative self-consistent parallel imaging reconstruction
  from arbitrary $k$-space.
\newblock {\em Mag. Reson. Med.}, 64(2):457--471, 2010.

\bibitem{Lustig07}
Michael Lustig, David Donoho, and John~M. Pauly.
\newblock Sparse mri: The application of compressed sensing for rapid {MR}
  imaging.
\newblock {\em Magnetic Resonance in Medicine}, 58(6):1182--1195, 2007.

\bibitem{Muck15}
M.~J. Muckley, D.~C. Noll, and J.~A. Fessler.
\newblock Fast parallel {MR} image reconstruction via {B}1-based, adaptive
  restart, iterative soft thresholding algorithms ({BARISTA}).
\newblock {\em IEEE Trans. Med. Imaging}, 34(2):578--588, 2015.

\bibitem{PPST23}
G.~Plonka, D.~Potts, G.~Steidl, and M.~Tasche.
\newblock {\em Numerical {F}ourier {A}nalysis}.
\newblock Birkh\"auser, Cham, 2nd edition, 2023.

\bibitem{MOCCA}
Gerlind Plonka and Yannick Riebe.
\newblock {MOCCA}: A fast algorithm for parallel {MRI} reconstruction using
  model based coil calibration.
\newblock {\em preprint}, page 32 pages, 2024.

\bibitem{Sezan84}
M.~Ibrahim Sezan and Henry Stark.
\newblock Tomographic image reconstruction from incomplete view data by convex
  projections and direct {F}ourier inversion.
\newblock {\em IEEE Trans. Med. Imaging}, 3(2):91--98, 1984.

\bibitem{ESPIRIT}
M.~Uecker, P.~Lai, Mark~J. Murphy, P.~Virtue, M.~Elad, J.M. Pauly, S.S.
  Vasanawala, and M.~Lustig.
\newblock {ESPIRiT}—an eigenvalue approach to autocalibrating parallel {MRI}:
  {W}here {SENSE} meets {GRAPPA}.
\newblock {\em Magn. Reson. Med.}, 71(3):990--1001, 2014.

\bibitem{Wu21}
Tingting Wu, Lixin Shen, and Yuesheng Xu.
\newblock Fixed-point proximity algorithms solving an incomplete {F}ourier
  transform model for seismic wavefield modeling.
\newblock {\em Journal of Computational and Applied Mathematics}, 385:113208,
  2021.

\bibitem{Xiao22}
Yao Xiao, Jan Glaubitz, Anne Gelb, and Guohui Song.
\newblock Sequential image recovery from noisy and under-sampled {F}ourier
  data.
\newblock {\em J. Sci. Comput.}, 91:79, 2022.

\bibitem{Yang10}
Junfeng Yang, Yin Zhang, and Wotao Yin.
\newblock A fast alternating direction method for {TVL1-L2} signal
  reconstruction from partial {F}ourier data.
\newblock {\em IEEE Journal of Selected Topics in Signal Processing},
  4(2):288--297, 2010.

\bibitem{Zhang22}
Mingli Zhang, Mingyan Zhang, Fan Zhang, Ahmad Chaddad, and Alan Evans.
\newblock Robust brain {MR} image compressive sensing via re-weighted total
  variation and sparse regression.
\newblock {\em Magnetic Resonance Imaging}, 85:271--286, 2022.

\end{thebibliography}
}

\end{document}